\documentclass[11pt,reqno]{amsproc}

\title[]{Global regularity for critical SQG in bounded domains}

\author{Peter Constantin}
\address{Department of Mathematics, Princeton University, Princeton, NJ 08544}
\email{const@math.princeton.edu}

\author{Mihaela Ignatova}
\address{Department of Mathematics, Temple University, Philadelphia, PA 19122}
\email{ignatova@temple.edu}

\author{Quoc-Hung Nguyen}
\address{Academy of Mathematics and Systems Science - Chinese Academy of Sciences}
\email{qhnguyen@amss.ac.cn}
\usepackage[margin=1in]{geometry}
\usepackage{amsmath, amsthm, amssymb}
\usepackage{times}
\usepackage{color}
\usepackage{hyperref}

\newcommand{\vp}{\varphi}
\newcommand{\pa}{\partial}
\newcommand{\la}{\label}
\newcommand{\fr}{\frac}
\newcommand{\na}{\nabla}
\newcommand{\be}{\begin{equation}}
	\newcommand{\ee}{\end{equation}}
\newcommand{\ba}{\begin{array}{l}}
	\newcommand{\ea}{\end{array}}
\newcommand{\Rr}{{\mathbb R}}
\newtheorem{thm}{Theorem}
\newtheorem{prop}{Proposition}
\newtheorem{lemma}{Lemma}
\newtheorem{coro}{Corollary}
\newtheorem{rem}{Remark}

\newcommand{\beg}{\begin}
\newcommand{\ov}{\overline}

\newcommand{\D}{\Delta_D}
\renewcommand{\l}{\Lambda_D}

\date{today}
\begin{document}
	\begin{abstract}
	We prove the existence and uniqueness of global smooth solutions of the critical dissipative SQG equation in bounded domains in $\mathbb R^2$. This solves an open problem. We introduce a new methodology of transforming  the single nonlocal nonlinear evolution equation in a bounded domain into an interacting system of extended nonlocal nonlinear evolution equations in the whole space. The proof then uses the method of the nonlinear maximum principle for nonlocal operators in the extended system.
	\end{abstract}
	
	
	\noindent
	
	\maketitle
	\section{Introduction}

	The Surface Quasigeostrophic equation (SQG) of geophysical origin (\cite{held}) 
	was proposed as a two dimensional model for the study of inviscid incompressible
	formation of singularities (\cite{c}, \cite{cmt}). The equation has been studied extensively. Blow up from smooth initial data is still an open problem, although  the original blow-up scenario of \cite{cmt} has been ruled out analytically (\cite{cord}) and numerically (\cite{cnum}). The addition of fractional dissipation yields globally regular solutions if the power of the Laplacian is larger or equal than one half. When  the linear dissipative operator is precisely the square root of the Laplacian, the equation is commonly referred to as the ``critical dissipative SQG'', or ``critical SQG''. The global regularity of solutions for critical SQG in the whole space or on the torus was obtained independently in \cite{caf} and \cite{knv} by very different methods.  Several subsequent proofs were obtained (see \cite{cv1}, \cite{cvt} and references therein). 
	
	 The basic ingredients used in \cite{cv1} are specific nonlinear maximum principle lower bounds for  $\Lambda = \sqrt{-\Delta}$, the square root of the Laplacian in the whole space $\Rr^d$.  A typical example is
\be
D(f) =f\Lambda f- \fr{1}{2}\Lambda\left({f^2}\right) \ge c\left(\|\theta\|_{L^{\infty}}\right)^{-1} {f^3}
\la{corcorv}
\ee
pointwise, for $f=\pa_i\theta$ a component of the gradient of a bounded function $\theta$. 
This is a useful cubic lower bound for a quadratic expression, when  $\|\theta\|_{L^{\infty}}\le\|\theta_0\|_{L^{\infty}}$ is known to be bounded above. The critical SQG equation in $\Rr^2$ is 
\be
\pa_t \theta + u\cdot\na\theta + \Lambda\theta = 0
\la{critsqg}
\ee
where
\be
u = \na^{\perp}\Lambda^{-1}\theta  = R^{\perp}{\theta}
\la{ucr}
\ee
and $\na^{\perp} = (-\pa_2, \pa_1)$. The equation has a weak maximum principle, the $L^{\infty}$ norm of $\theta$ does not grow in time. 
	In \cite{cv1} and \cite{cvt}, instead of estimating directly gradients, the proof of global regularity proceeds by estimating finite difference quotients, with the aim of first obtaining bounds for $C^{\alpha}$ norms. A basic feature of the critical SQG equation in the whole space is the fact that once the solution is bounded in $C^{\alpha}$, for some $\alpha>0$, then it follows that the solution is in fact $C^{\infty}$ smooth. More generally, if a generalized SQG equation has a dissipation of order $s$, i.e., $\Lambda$ is replaced by $\Lambda^s$ with $0<s\le 1$, then if $\theta$ is bounded in $C^{\alpha}$ with $\alpha>1-s$, then the solution is smooth (\cite{cw}). This condition is sharp in the class of general linear advection diffusion equations, (\cite{svz}). In \cite{cvt}, the smallness of $\alpha$ is crucially used to show that the nonlinear term  appearing in the evolution of the finite difference quotient $q=D_h^{\alpha}\theta$  of solutions of  \eqref{critsqg} is entirely dominated by the term corresponding to $D(q)$. This is no longer the case in bounded domains.

	The critical SQG equation in bounded domains is given by
	
	\be
	\pa_t \theta + u\cdot\na\theta + \l \theta = 0
	\la{sqg}
	\ee
	with
	\be
	u = \na^{\perp}\l^{-1}\theta.
	\la{u}
	\ee
	Here $\Omega\subset \Rr^d$ is a bounded open set with 	smooth oriented boundary, $\l$ is the square root of the Laplacian with vanishing Dirichlet boundary conditions, and $\na^{\perp} = J\na$ with $J$ an invertible antisymmetric matrix.  
	
The problem of global regularity of critical SQG in bounded domains was open until the present work. Interior regularity
was investigated in \cite{sqgb}. The approach, initiated in \cite{ci}, was based on bounds on the heat kernel. One of the main obstacles to implementing a proof of regularity in bounded domains is the lack of translation invariance. It has as consequence the fact that the Riesz transforms $R_D= \na\l^{-1}$ are not spectral operators, i.e., they do not commute with functions of the Dirichlet  Laplacian. In \cite{sqgb} the method of the nonlinear maximum principle was used in conjunction with estimates for the  commutator between difference quotient operators and  $\l$. These estimates degenerate at the boundary but they can be used to obtain a priori global in time  interior Lipschitz bounds of solutions. A construction of solutions with this degree of regularity was done in \cite{igsqg}. Global weak solutions in bounded domains were studied in \cite{cn, cn1}.
In \cite{sqgzapde}, necessary and sufficient conditions for global $C^{\alpha}$ bounds up to the boundary with $\alpha < 1-\fr{d}{p}$ were given in terms of quantitative information on supercritical ($p>d$) $L^p$ norms of $\fr{\theta}{w_1}$, where $w_1$ is the first eigenfunction of the Dirichlet Laplacian. Because $w_1$ vanishes linearly at the boundary, this implies that global $C^{\alpha}$  bounds are available if and only if solutions have a Holder rate of vanishing of $\theta$ at $\pa\Omega$. 

The work
\cite{vas} presented $C^{\alpha}$ bounds of weak solutions, using an approach based on the method of De Giorgi, employed first in the whole space in \cite{caf}. However, unlike the case of the whole space, going from $C^{\alpha}$ to higher regularity does not follow using this approach. Global Holder continuous solutions were not known to be unique, nor smooth. In this work we prove the existence and uniqueness of global smooth solutions. In order to obtain this result we introduce a new methodology consisting of the extension of the single equation in the bounded domain to an interacting system of equations in the whole space. We then employ the method of the nonlinear maximum principle in the analysis of the extended nonlinear nonlocal system.

\subsection{ Main Results and Description of Ideas of Proofs}

	In this paper we prove 
	\beg{thm}\la{globreg} Let $\Omega\subset \Rr^2$ be a bounded domain with smooth boundary. Let $\theta_0\in H_0^1(\Omega)\cap H^m(\Omega)$, $m> 2.5$, and let $T>0$. Then there exists a unique solution of \eqref{sqg}, \eqref{u} with initial data $\theta_0$ and which belongs to $L^{\infty}(0,T; H_0^1(\Omega)\cap H^m	{{(\Omega))}}$.
	\end{thm}
	The solution is in fact smooth for all time and eventually exponentially decays to zero. The initial data need not be smooth. By parabolic regularization, if the initial data are $C^{\alpha_0}$  for some $\alpha_0>0$ and vanish at the boundary, then the solution exists locally, is unique, becomes instantly smooth (Lemma \ref{keyle}), persists globally, and decays (Theorem \ref{calphaglob} and Remark \ref{highgre2}). 

	The main result we prove is a priori bound, on which the proof of Theorem \ref{globreg} rests. This is
	\beg{thm}\la{apriori} Let $\theta\in C^{1+\alpha_0}([0,T]\times\Omega)$ be a classical solution of \eqref{sqg}, \eqref{u} for some $\alpha_0\in (0,1)$. There exists a small constant $\delta$ depending {{only}} on $T$ and the domain $\Omega$, such that, for 	{{$0<\alpha<\alpha_0$}} satisfying
	\be
	 \alpha(\|\theta_0\|_{L^{\infty}(\Omega)} +1) \le \delta,
	 \la{cond}
	 \ee
there exists a constant $C_{\alpha}$ depending  (continuously and explicitly) only on $\|\theta_0\|_{C^{\alpha}(\Omega)}$, $T$, the domain $\Omega$ and $\alpha$, such that
\be
\sup_{0\le t\le T}\|\theta(\cdot, t)\|_{C^{\alpha}( \Omega)} + \int_0^T\|	{{\theta(\cdot, t)}}\|_{C^{1 +\fr{\alpha}{2}}(\Omega)}dt \le C_{\alpha}.
\la{apriorib}
\ee 
	\end{thm}
		The detailed result is given in Theorem \ref{holderR}. The factor $\fr{1}{2}$ is not structural, it is there only to signify that the gain of regularity is less than 1, but regularity above $L^1(dt; C^1(\Omega))$ is attained. Once this result is obtained, Theorem \ref{globreg} follows from the  local existence and uniqueness of smooth solutions of (\ref{sqg}) given in \cite{sqgb} and a natural continuation result. More precisely, the local existence theorem  is
	\beg{thm}\la{locex} Let $\Omega$ be a bounded open domain with smooth boundary in $\Rr^2$. Let $m\ge 2$ and let $\theta_0\in H_0^1(\Omega)\cap H^{m}(\Omega)$.  There exists a time $T_0>0$ and a unique solution $\theta$ of \eqref{sqg}  satisfying
	\be
	\theta \in L^{\infty}(0, T_0; 	{{H_0^1(\Omega)}}\cap H^m(\Omega)) \cap L^2(0, T_0; H^{m+\fr{1}{2}}(\Omega)).
	\la{locb}
	\ee
	The time $T_0$ depends on the initial norm in $H^2(\Omega)$.
	\end{thm}
This result was proved in \cite{sqgb} for $m=2$ using Galerkin approximations \[\theta_N = \sum_{j=1}^N c_j(t) w_j(x),\] Sobolev energy bounds and Sobolev embedding.
The fact that the expansion is in terms of eigenfunctions $w_j$ of the Dirichlet Laplacian allows integration by parts because powers of the  fractional Laplacian applied to the Galerkin approximation vanish at the boundary, $\l^s \theta_N{_{\left |\right.\pa\Omega}}=0$. The general $m$ case follows in the same manner.

The local existence result is combined with the following natural continuation result.
\beg{thm}\la{cont}Let $\theta_0\in H_0^1(\Omega)\cap H^m(\Omega)$, $m\ge 2$ be given and let $\theta\in L^{\infty}([0, T_0), H^{m}(\Omega))$ be a solution of \eqref{sqg},\eqref{u}. Assume that for some  $0<\beta<1$ there exists a constant $C_{\beta}$ such that
\be
\int_0^{T_0}\|\theta(\cdot, t)\|_{C^{1+\beta}(\Omega)} dt \le C_{\beta}
\la{gradcond}
\ee
holds. Then there exists a constant $C_m$ depending (continuously and explicitly) only on $\|\theta_0\|_{H^m(\Omega)}$,
the domain $\Omega$ and $C_\beta$, such that
\be
\sup_{0\le t\le T_0}\|\theta(\cdot, t)\|_{H^m(\Omega)}\le C_m.
\la{hmb}
\ee
\end{thm}
Combined with the local existence result, this implies that the solution can be uniquely extended beyond $T_0$.
The proof of Theorem \ref{cont}  is based on energy estimates and well-known facts about the boundedness of Riesz transforms in $C^r(\Omega)$ \cite{cabre}.  The condition \eqref{gradcond} is sufficient for uniqueness and persistence of smoothness of solutions in inviscid SQG as well. A detailed proof is left for the interested reader. In this paper we provide an independent local existence and persistence proof directly based on $C^r$ spaces, without use of Sobolev spaces.\\

The proof of Theorem \ref{apriori} requires the introduction of a number of new elements which we believe are of general interest. As in \cite{sqgb} we use functional calculus to represent the square root $\l$ of the Dirichlet Laplacian in terms of the heat kernel, but unlike in \cite{sqgb}, a direct commutator between finite difference quotients and $\l$ is not attempted. We consider instead an appropriate cover of $\overline {\Omega}$ with open balls and smooth subordinated localizers $\chi$. We associate to the balls corresponding to the boundary $\pa\Omega$ smooth diffeomorphisms $Y: B\cap\Omega \to \mathbb R^2_{+}$ which flatten the boundary.  For interior balls the diffeomorphisms are just the identity.  We consider then maps $\mathcal F$ which take functions $\chi g$ defined in patches  $B\cap \overline{\Omega}$ to functions defined in the whole space $\Rr^2$ by  $\mathcal F(\chi g) = \mathcal O(\chi g\circ Y)$, where $\mathcal O$ is odd extension across the boundary of the half space. While localization and flattening of the boundary is a familiar procedure for proving regularity of elliptic and parabolic equations in bounded domains, our approach requires to extend also the localized equation. This is needed because, unlike the case of local PDE, in the nonlocal case it is difficult to disentangle tangential directions from the normal direction in the principal symbol of the equation. Thus, after the localization and change of variables, the Dirichlet Laplacian is conjugated to (or intertwined with) a second order elliptic operator $L$  with Lipschitz coefficients, defined in the whole space,  plus an error.  The change of variables $Y$ is defined carefully so that
the cross terms involving normal and tangential derivatives vanish near the boundary, see Appendix 1. This allows $L$ to have Lipschitz coefficients. We take advantage of the fact that the heat equation is local, and therefore when we commute a smooth cutoff function $\chi$ with $e^{t\D}$ we obtain a local error, which we represent in terms of the heat operator using the Duhamel formula. We then apply $\mathcal F$ and the functional representations of $\l$ and of $L^{\fr{1}{2}}$ in terms of their respective semigroups  to obtain an expression for the intertwining of the localized $\l$ with the corresponding $L^{\fr{1}{2}}$,
\be
\mathcal F(\chi\l\theta) - L^{\fr{1}{2}}\mathcal F(\chi\theta) = \mathbf{R_\chi}(\theta)
\la{commulr}
\ee
and show (Proposition \ref{lerrep}) that
\be
\|\mathbf R_{\chi}(\theta)\|_{C^r(\Rr^2)}  \lesssim \|\theta\|_{C^r(\Omega)}
\la{Rbddr}
\ee
holds for $0<r<1$.  
We localize and extend the nonlinear term $\left(\na^{\perp}\l^{-1}\theta \right)\cdot \na \theta = \{\l^{-1}\theta, \theta\}$. It is only here that we use the fact that we are in two dimensions. We use properties of the Poisson bracket which allow odd extension across the flattened boundary after composition with $Y$, while maintaining almost intact the Poisson structure (Proposition \ref{Fnon}).

We arrive thus at a representation of the equation \eqref{sqg}, \eqref{u} as a coupled system of equations in the whole space. This constitutes a new methodology to study boundary value problems which we expect to be more  broadly useful. Corresponding to the cover of $\overline\Omega$ with balls, we have $N$ transformations $\mathcal F$ (some of them not requiring changes of variables), and for each patch $B_i\cap \Omega$, $1\le i\le N$, functions
$\theta_i = \mathcal F(\chi_i \theta)$ which obey equation in the whole space
\be
\pa_t \theta_i + u_i\cdot\na\theta_i + L^{\fr{1}{2}}\theta_i = f_i
\la{sqgi}
\ee
The operators $L$ depend on the patch but they have the same second order elliptic, Lipschitz coefficients nature. The velocities $u_i$ depend on the whole $\theta$, not only on  $\theta_i$, but the dependence is quasi-local, meaning that the $u_i = \na^{\perp}\widetilde{L}^{-\fr{1}{2}}(\widetilde{\theta_i}) + \text{error}$ where $\widetilde L$ is like $L$ and $\widetilde{\theta_i}$ covers $\theta_i$, i.e.
$\theta_i = \eta\widetilde{\theta_i}$ with $\eta$ Lipschitz and compactly supported. The "forces" $f_i$ arise from errors of intertwining and extension and depend in a nonlinear manner on $\theta$. The operators $L^{\fr{1}{2}}$ have variable coefficients. The equations \eqref{sqgi} are not stand alone equations, rather they are representations of localizations and extensions of \eqref{sqg}, \eqref{u}. Nevertheless they serve the purpose to estimate derivatives of $\theta$. The system of equations \eqref{sqgi} is sparse (only few nearby patches interact) but it is not treated as an algebraically coupled system, more like a redundantly oversampled contact system.

We apply and extend the method of nonlinear maximum principle to the aggregate \eqref{sqgi}, taking great advantage of an enhanced nonlinear lower bound. The fact that
the operators have variable coefficients, not unexpectedly implies that inequalities for the evolving  $C^{\alpha}$ norms for small $\alpha$ cannot be closed, as they are driven by norms of full derivatives of $\theta$. The nonlinear maximum principle provides though powerful nonlinear damping. When trying to estimate the $C^{\alpha}$ norm, the most dangerous term still comes from the finite difference quotient of the active scalar's velocity, as in the case of critical SQG in the whole space, and  bounding it still requires the use of the $D(q)$ argument.   Like in previous works using the method of the nonlinear maximum principle, in the present work we also have only one small parameter, namely $\alpha$. We consider the evolution of the difference quotient $q =D^{\alpha}_h\theta_i$ in each patch.
In previous works \cite{cvt, sqgb} the smallness of $\alpha$ was used to overcome the contribution of the difference quotient of the active scalar velocity, $D_h^1 u$, by crucially using $D(q)$ in a pointwise manner, and also by using a nonlinear lower bound
\be
|h|^{-2\alpha}D(q)\ge c |h|^{-1 + \alpha} q^3\|\theta\|_{L^{\infty}}^{-1}
\la{qcube}
\ee
in the evolution of $q^2$.  In this work we use the same idea to overcome the contribution of the inner core of $D^1_h u_i$. In addition we use the observation that 
 at the point of maximum of $q$, because $\theta$ is a priori bounded, $|h|$ must be very small, less than $\left(\fr{|q|}{2\|\theta\|_{L^{\infty}}}\right)^{-\fr{1}{\alpha}}$. Thus, the term $D(q)$ provides a nonlinearly enhanced damping with  an excess of order $\fr{1}{\alpha}$,
\be
|h|^{-2\alpha}D(q) \ge c q^{2+\fr{1}{\alpha}}\|\theta\|_{L^{\infty}}^{-\fr{1}{\alpha}},
\la{dqenhanced}
\ee
resulting in a differential inequality for the maximum of the type
\be
\pa_t q +c \|\theta_0\|_{L^{\infty}}^{-\fr{1}{\alpha}} q^{1+ \fr{1}{\alpha}}\le \text{translation and localization errors},
\la{typeq}
\ee 
and the smaller $\alpha$ is, the larger this useful excess is. The error terms due to the localization and the absence of translation invariance are controlled by this high homogeneity of the nonlinear damping. Thus, the smallness of $\alpha$ is used in two ways, once by bounding the worst term by part of $D(q)$, and the other, by affording high homogeneity nonlinear error terms using the excess damping of homogeneity  $\alpha^{-1}$ provided also by $D(q)$. The upshot, described in Lemma \ref{drivencalpha}, is an a priori bound of the supremum in time of the $C^{\alpha}$ norm which is driven by the time integral of the $C^{1+\fr{\alpha}{2}}$ norm. Here the factor $\fr{1}{2}$ is not part of the structure of the equation, it only represents the crucially important fact that less than a whole derivative is lost. The loss of almost a whole derivative is however unavoidable. This loss marks the difference between translation invariant and non-translation invariant equations, and it occurs even if we replace in the usual SQG equation in the whole space, the linear dissipation  $\sqrt{-\Delta}$ by the linear dissipation $a(x)\sqrt{-\Delta}$, where $a$ is a uniformly bounded positive smooth function.

In order to close the estimates we employ a result about linear dissipative advection equations, of the type
\be
\pa_t v + b\cdot\na v  + L^{\fr{1}{2}}v = f
\la{veq}
\ee
with $b$ and $f$ Holder continuous, $b\in C^{\beta}(\Rr^d)$, $f\in C^{\beta}(\Rr^d)$ and initial data in $C^{\beta}(\Rr^d)$. We show (Lemma \ref{estidiffusion}) that the norm of $v$ in $L^1(0,T; C^{1+\alpha}(\Rr^d))$ is bounded in terms of the norms of $b, f, v(0)$ in $C^{\beta}$ for $\alpha<\beta$.   
For the proof of this result we use a method of freezing coefficients, which rectifies the variable coefficient operator $b\cdot\na + L^{\fr{1}{2}}$, that is, it approximates it by its tangent at each frozen point $y$, the constant coefficients operator $b(y)\cdot\na_x + L_y^{\fr{1}{2}}$. This treatment requires a systematic study of the kernels of semigroups $e^{-tL}$, $e^{-tL^{\fr{1}{2}}}$ and their approximations. 
The linear result of Lemma \ref{estidiffusion} is applied to the specific nonlinear equation in Corollary \ref{corol} and is used in conjunction with the high homogeneity of the nonlinear damping to close the estimates and prove the main result, Theorem \ref{apriori}.

The paper is organized as follows. In Section \ref{prelims} we set up the cover of the domain, recall some basic facts about the Dirichlet Laplacian and introduce notation. Ins Section \ref{loex} we describe the procedure of localization and extension. We prove in this section bounds for the intertwining of the localized and extended heat semigroup $e^{t\D}$, and bounds for the intertwining of the localized and extended $\l$. We follow, in Section \ref{exlsqg} with the derivation of the extended localized system \eqref{sqgi} and provide bounds for the errors of nonlinear intertwining. Section \ref{bld} is devoted to proving the useful results on the linear dissipative advection equation \eqref{veq}. In Section \ref{nlmp} we apply the nonlinear maximum principle method to the system and obtain a priori bounds for Holder norms of solutions. The proof of Theorem \ref{apriori}, in its precise form, Theorem \ref{holderR} is given in Section \ref{hore}. A self-contained proof of local existence with Holder initial data and global persistence of smooth solutions is given in Section \ref{glcr}. Appendix 1 (Section \ref{change}) describes the change of variables $Y$ and Appendix 2 (Section \ref{kernelestimes}) provides useful estimates of heat kernels and approximations. 

	\section{Preliminaries}\la{prelims}
	We consider a bounded connected (but not necessarily simply connected) domain $\Omega\subset \Rr^d$ with smooth oriented boundary $\pa\Omega$.	We cover the boundary $\pa\Omega$ with open balls $B_i^0$, $i=1, \dots, N_1$, centered at points on the boundary, and take nested concentric open balls $B_i^0\subset B_i^1\subset B_i^2 \subset B_i^3$, such that the portion of the boundary of each $\pa\Omega\cap B_i^3$  is given after a translation and a rotation by the graph of smooth function with nearly constant gradient.
	We consider smooth cutoffs $\chi_i^{j}$, $i=1,\dots, N_1$, $j=0,1,2.$ such that $\chi_i^j$ is identically equal to $1$ on
	$\overline B_i^j\cap\overline \Omega$ and has compact support in $\overline{\Omega}\cap B_i^{j+1}$. Thus 
	$\chi_i^j = \chi_i^{j+1}\chi_i^j$. 
	The radius of the largest balls $B_i^3$ is denoted $r_0$  and is taken without loss of generality to be the same for all $i$.
	This radius is taken small enough such that, if the boundary $\pa\Omega$ has several connected components, the balls corresponding to one connected component of the boundary do not intersect the balls corresponding to another connected component of the boundary.
	We cover $\Omega\setminus \cup_{i=1}^{N_1}B_i^0$ with balls $B_i^{0}\subset \Omega$, with $i=N_1+1,\dots N$ and take nested  $B_i^0\subset B_i^1\subset B_i^2\subset B_i^3$ with $B_i^3\subset\Omega$, 
	and cutoffs $\chi_i^j$ with $j=0, 1,2$ which identically equal $1$ on  $B_i^j $ and are compactly supported in $B_i^{j+1}$.
	We refer to the balls with index $i\le N_1$ as boundary balls, and to the balls with $N_1<i\le N$ as interior balls.
	The set of balls and cutoffs is entirely based on the geometry of the domain, and is fixed throughout the paper. In each boundary ball we define diffeomorphisms
	\be
	Y_i : B_i^3\cap \Omega \to \Rr^{d}_+ 
	\la{Y_i}
	\ee
	$i=1,\dots N_1$ with certain properties. Without loss of generality we take the cutoffs $\chi_i^{j}$ to be such that
	$\chi_i^{j}\circ Y_1^{-1}$ have smooth even extensions across $y_d=0$.
	We associate to a smooth solution $\theta(x,t)$ of \eqref{sqg} defined for $x\in\Omega$ and $t\in[0,T]$ an array of functions
	\be
	\Theta(y,t) = (\widetilde\theta_i(y,t))_{i=1,\dots N} 
	\la{bigtheta}
	\ee
	defined on for $y\in\Rr^d$ and $t\in [0,T]$ in the following manner.
	For $i=1, \dots N_1$, we set
	\be
	\widetilde\theta_i = \mathcal O( (\chi_i^1\theta)\circ Y_i)
	\la{thetabpatch}
	\ee
	where $\mathcal O$ is odd extension across $y_{d}=0$. For $i=N_1+1, \dots N$, we put
	\be
	\widetilde \theta_i = \chi_i^1\theta
	\la{thetaipatch}
	\ee
	where we denote by the same letter $f$  the extension of a function $f$ that is compactly supported in $\Omega$ by setting it equal to 0 outside the support of $f$. Norms of $\Theta$ in space are equivalent to norms of $\theta$ in $\overline\Omega$. 
	
	We use in particular $C^r$ norms. We frequently use the interpolation inequality
	\be
	\|f\|_{C^\beta} \le \|f\|_{C^{\delta}}^a\|f\|_{C^{\gamma}}^{1-a}
	\la{interdeltagamma}
	\ee
	for $\beta = a \delta + (1-a)\gamma$ 	with $0<a<1$.

	The $L^2(\Omega)$-normalized eigenfunctions of the Dirichlet Laplacian $-\D$ are denoted $w_j$, and its eigenvalues counted with their multiplicities are denoted $\mu_j$: 
	\be
	-\D w_j = \mu_j w_j.
	\la{ef}
	\ee
	It is well known that $0<\mu_1\le...\le \mu_j\to \infty$  and that $-\D$ is a positive selfadjoint operator in $L^2(\Omega)$ with domain ${\mathcal{D}}\left(-\D\right) = H^2(\Omega)\cap H_0^1(\Omega)$.
	The ground state $w_1$ is positive and
	\be
	c_0d(x) \le w_1(x)\le C_0d(x)
	\la{phione}
	\ee
	holds for all $x\in\Omega$, where $c_0, \, C_0$ are positive constants depending on $\Omega$,  $d(x)$ is  the distance from $x$ to the boundary $\partial\Omega$. Functional calculus can be defined using the eigenfunction expansion. In particular
	\be
	\left(-\D\right)^{\beta}f = \sum_{j=1}^{\infty}\mu_j^{\beta} f_j w_j
	\la{funct}
	\ee
	with 
	\[
	f_j =\int_{\Omega}f(y)w_j(y)dy
	\]
	for $f\in{\mathcal{D}}\left(\left (-\D\right)^{\beta}\right) = \{f\left |\right. \; (\lambda_j^{\beta}f_j)\in \ell^2(\mathbb N)\}$.
	We denote by
	\be
	\l^s = \left(-\D\right)^{\fr{s}{2}}, 
	\la{lambdas}
	\ee
	the fractional powers of the Dirichlet Laplacian
	and with $\|f\|_{s,D}$ the norm in ${\mathcal{D}}\left (\l^s\right)$:
	\be
	\|f\|_{s,D}^2 = \sum_{j=1}^{\infty}\mu_j^{s}f_j^2.
	\la{norms}
	\ee
	It is well-known that
	\[
	{\mathcal{D}}\left( \l \right) = H_0^1(\Omega).
	\]
	Note that in view of the identity
	\be
	\lambda^{\fr{s}{2}} = c_{s}\int_0^{\infty}(1-e^{-t\lambda})t^{-1-\fr{s}{2}}dt,
	\la{lambdalpha}
	\ee
	with 
	\[
	1 = c_{s} \int_0^{\infty}(1-e^{-\tau})\tau^{-1-\fr{s}{2}}d\tau,
	\]
	valid for $0\le s <2$, we have the representation
	\be\
	\left(\left(\l\right)^{s}f\right)(x) = c_{s}\int_0^{\infty}\left[f(x)-e^{t\D}f(x)\right]t^{-1-\fr{s}{2}}dt
	\la{rep}
	\ee
	for $f\in{\mathcal{D}}\left(\left (-\l\right)^{s}\right)$.
	
	We use second order elliptic operators in divergence form
\begin{equation}\la{lop}
L=-\operatorname{div}_x(A(x)\na_x.)
\end{equation}
in $\mathbb{R}^d$, where $A$ is a symmetric matrix-valued function in $\mathbb{R}^d$ which satisfies
\begin{align}\label{z65a}
&A(x)\geq  c_1 I\quad\quad \forall x\in \mathbb{R}^d,\\&
||\na A||_{L^\infty}+||A||_{L^\infty}\leq c_2,\label{z65b}
\end{align}
with  {{ constants}} $c_1,c_2>0$.\\
We denote by  $H_L(x,y,t)$ the kernel of $\partial_t+L$ in $\mathbb{R}^d\times (0,\infty)$. When $A$ is a constant matrix it is well known that  
\begin{align}\label{z61}
H_L(x,x+z,t):=\frac{1}{\sqrt{\det A}(4\pi t)^{\frac{d}{2}}}\exp\left(-\frac{(A^{-1}z\cdot z)}{4t}\right)
\end{align}
where $A^{-1}$ is the inverse of matrix $A$ and $(\cdot )$ is the Euclidean scalar product in $\mathbb{R}^d$. 
We define  the square root $L^{\frac{1}{2}}$ of the operator $L$  by 
\begin{align}\label{z52}	L^{\frac{1}{2}}u(x)=c_0\int_{0}^{\infty}\int_{\mathbb{R}^d}H_L(x,x+z,s)(u(x)-u(x+z))dzs^{-\frac{3}{2}}ds,
\end{align}
with 
\begin{equation}\label{consc0}
c_0=\frac{1}{2\Gamma(\frac{1}{2})}.
\end{equation}
In particular, when  $A$ is a constant matrix,  we have 
\begin{equation}\label{z62}
L^{\frac{1}{2}}u(x)=	\frac{	\tilde{c}_0}{\sqrt{\det A}}\int_{\mathbb{R}^d}\frac{u(x)-u(x+z)}{(A^{-1}z\cdot z)^{\frac{d+1}{2}}}dz
\end{equation}
with 
\begin{equation}\label{z64}
\tilde{c}_0=\frac{1}{\pi^{\frac{d}{2}}}\frac{\Gamma(\frac{d+1}{2})}{\Gamma(\frac{1}{2})}.
\end{equation}
For each fixed $y\in \mathbb{R}^d$, we define 
\be
L_{y}=-\operatorname{div}_x(A(y)\na _x).
\la{Ly}
\ee
This is a constant coefficients second order elliptic operator. In view of \eqref{z61}, {{ the}} kernel of $\partial_t+ L_{y}$ is given by 
\begin{equation}\label{z47}
G_{A(y)}(z,t)=\frac{1}{\sqrt{\det A(y)}(4\pi t)^{\frac{d}{2}}}\exp\left(-\frac{(A(y)^{-1}z\cdot z)}{4t}\right),
\end{equation}
and, using \eqref{z62},  the square root of the operator $L_y$ is given by
\begin{align}
L_{y}^{\frac{1}{2}}u(x)=	\frac{	\tilde{c}_0}{\sqrt{\det A(y)}}\int_{\mathbb{R}^d}\frac{u(x)-u(x+z)}{(A(y)^{-1}z \cdot z)^{\frac{d+1}{2}}}dz.
\end{align}
We emphasize that $L_{y}^{\frac{1}{2}}u(x)\vert_{y=x}$ is not identical to $L^{\frac{1}{2}}u(x)$. However, $L_{y}^{\frac{1}{2}}u(x)\vert_{y=x}-	L^{\frac{1}{2}}u(x)$ is a  zero order operator, for which we provide bounds in  Lemma \ref{remindL1/2} of Appendix 2. We prove in Appendix 2,{{  Lemma \ref{heatke}}}  useful quantitative bounds for the difference of heat kernels $H_L(x,x+z,t)-G_{A(x+z)}(z,t)$.

\section{Localization and extension}\label{loex}
In this section, we take $d=2$. The localization and extension of the linear term can be done in any dimension. We use the Poisson structure of the nonlinearity, and there $d=2$ is important.\\
We consider a point on the boundary $x_0\in\pa\Omega = \Gamma$. Without loss of generality, after a translation and a rotation, $x_0=0$ and the domain $\Omega$ is given locally near $0$ as $\{x=(x_1,x_2)\left |\right.\, x_2>\varphi(x_1)\}$ where $\varphi(0) = \varphi'(0) =0$ and the function $\varphi:(-\ell,\ell)\to \Rr$ is smooth. By taking $\ell>0$ small enough, we make sure 
\be
|\varphi'(x_1)| \le \epsilon
\la{varphib}
\ee
where $\epsilon>0$ is a small nondimensional number at our disposal. We extend the function $\varphi$ to all of $\Rr$ so that
\eqref{varphib} holds globally, and moreover, we may assume that $\varphi'$ vanishes outside a compact. 
We consider the global change of variables $\mathbb R^2\to\mathbb R^2$,  $x\mapsto Y(x)=(Y_1(x),Y_2(x))\in C^\infty$  in which $Y_1$ is given in Appendix 1 by \eqref{Y1}, and
\be
Y_2(x) = x_2-\varphi(x_1).
\la{Yx}
\ee
From the construction of $Y_1$ in Appendix 1, we have 
\begin{equation}
||\na Y-I||_{L^\infty}\leq \frac{1}{4}
\end{equation}
and 
\begin{equation}
\na Y_1.\na Y_2=0~~\text{in a neighborhood of } \Gamma.
\la{vanishy1y2}
\end{equation}
We denote the inverse of $Y$ by $X$, $Y^{-1} = X$.\\
The map $Y$ maps the portion near $x_0=0$ of $\Omega$ corresponding to $|x_1| <\ell$ to an open subset of $y_2>0$, and the corresponding portion of the boundary $\pa\Omega$ to 
an open segment $\{ |y_1| <\ell', \, y_2=0\}$.  \\
As it is very well known, under the change of variables $y\mapsto x = X(y)$, the Laplacian becomes
\be
\mathbf{\widetilde a}\pa_i(\frac{\widetilde a_{ij}}{\mathbf{\widetilde a}} \pa_j(\psi\circ X))=(\Delta_x \psi)\circ X,
\la{delX}
\ee
where $\pa_i = \fr{\pa}{\pa y_i}$ and
\be
\widetilde a_{ij} =(\na_x Y_i\cdot\na_xY_j)\circ X, ~~\mathbf{\widetilde a}=(\det\na Y) \circ {{X.}}\la{aij}
\ee
In view of \eqref{vanishy1y2} we have
\begin{align}
\widetilde a_{12}=0~~~~\text{in a neighborhood of } \{y_2=0\}.
\la{vanish12}
\end{align}
We consider functions $g(x)$ defined in $\Omega$ and a cutoffs $\chi\in C_0^{\infty}(\Rr^2\cap{\overline\Omega})$ with support included in $((-{\ell}, \ell )\times \Rr) \cap \overline{\Omega}$. Then functions $\chi g$ can be composed with $X$ and define functions compactly supported in $y_2\ge 0$ near $0$. If $g_{\left | \right.\pa\Omega} =0$, these functions vanish at $y_2=0$. We consider odd and even extensions  of functions $f$ defined on $\Rr^2_{+}$, 
\be
\mathcal O f(y_1,y_2) = \left\{
\ba
f(y_1,y_2), \;\;\quad \quad \text{for}\; y_2>0,\\
-f(y_1,-y_2), \quad \text{for}\; y_2<0,
\ea
\right.
\la{Of}
\ee
and 
\be
\mathcal E f(y_1,y_2) = \left\{
\ba
f(y_1,y_2), \;\;\quad \text{for}\; y_2>0,\\
f(y_1, -y_2), \quad \text{for}\; y_2<0.
\ea
\right.
\la{Ef}
\ee
Because 
\be
\ba
\pa_1\mathcal O f = \mathcal O\pa_1f,\; \text{if}\;  f\in C^1(\ov\Rr^2_{+}),\\
\pa_1\mathcal Ef = \mathcal E \pa_1 f, \; {\text{if}}\; f\in C^1(\ov\Rr^2_{+}),  \\
\pa_2\mathcal E f = \mathcal O\pa_2 f,\; {\text{if}}\; f\in C^1(\ov\Rr^2_{+}),    \\
\pa_2\mathcal O f = \mathcal E\pa_2f,\; {\text{if}}\; f\in C_0^1(\ov\Rr^2_{+}),
\ea
\la{comuodeven}
\ee
and the product rules
\be
\ba 
\mathcal O(fg) = \mathcal O(f)\mathcal E(g) = \mathcal E(f)\mathcal O(g),\\
\mathcal E(fg) = \mathcal E(f)\mathcal E(g) = \mathcal O(f)\mathcal O(g).
\ea
\la{prouodeven}
\ee
In view of \eqref{delX}, it follows that for function $\chi g\circ X$ we have
\be
\mathcal O(\Delta_x(\chi g) \circ X) = \mathbf{a} \pa_i( \frac{a_{ij}}{\mathbf{a}} \pa_j(\mathcal O((\chi g)\circ X))
\la{Odeltachi}
\ee
where we denote
\be
\mathbf{a} = \mathcal E\left(\mathbf {\widetilde a}\right),
\la{awa}
\ee
\be
a_{ii} = \mathcal E( \widetilde{a}_{ii}), \quad i=1,2,
\la{a_ii}
\ee
and
\be
a_{12} = \mathcal O\left( {\widetilde a}_{12}\right).
\la{a12}
\ee
We denote by $L,B$ the operators
\be
L f = -\pa_i(a_{ij}\pa_j f), ~~Bf=\frac{a_{ij}}{\mathbf{a}} \partial_i\mathbf{a}\pa_j f
\la{Lf}
\ee
viewed as  operators defined in $\Rr^2$ (for instance on functions $f\in H^2(\Rr^2)$). 
The coefficients of the extended operators are  $\mathbf{a}, a_{ii}$, even extensions $\mathbf{\widetilde{a}}$ and 
$\widetilde {a}_{ii}$, and $a_{12}= a_{21}$,  odd extensions of the cross terms $\widetilde {a}_{12}=\widetilde{a}_{21}$. This convention is kept throughout the paper. Note that, in view of the construction
of $Y$, and in particular \eqref{vanish12}, we have
\begin{align}\label{z87a}
&	(a_{ij})_{ij}\gtrsim I, \mathbf{a} \gtrsim 1,\\&	||\mathbf{a}||_{W^{1,\infty}(\mathbb{R}^2)}+||a_{ij}||_{W^{1,\infty}(\mathbb{R}^2)}\lesssim 1.\label{z87b}
\end{align}
We fix a smooth cutoff $\chi_2\in C_0^{\infty}(\Rr^2)$ compactly supported in $((-\ell, \ell)\times \Rr) \cap \overline{\Omega}$ and with the property that $\chi_2(x_1,x_2) =1$ for $x_1\in [-\fr{\ell}{2}, \fr{\ell}{2}]$ and $\varphi(x_1)\le x_2\le \varphi(x_1) +\delta$, we denote by $\mathcal F$ the operator
\be
g \mapsto \mathcal F(g)  = \mathcal O ( (\chi_2 g)\circ X)
\la{mathcalF}
\ee
and we note that
\be
\mathcal F :H_0^1(\Omega)\to H^1 (\Rr^2)
\la{h1bound}
\ee
and
\be
\mathcal F : \mathcal D(\D) = H_0^1(\Omega)\cap H^2(\Omega) \to H^2(\Rr^2)
\la{h2bound}
\ee
are bounded linear operators. We formalize the calculation \eqref{Odeltachi} as 
\beg{prop}\la{intertw}
Let $g\in\mathcal D(\D)$. Let $\chi\in C_0^{\infty}(\Rr^2)$ be such that $\chi\chi_2 = \chi$ (i.e. $\chi_2=1$ on the support of $\chi$).  Then $\mathcal F(\chi g)\in H^2(\Rr^2)$ and
\be
-\mathcal F (\D(\chi g) )= L \mathcal F (\chi g)+B\mathcal F (\chi g).
\la{interdeltaL}
\ee
\end{prop}
\beg{proof} We note that $\chi_2 =1$ on the support of $\D(\chi g)$ and the formal calculation \eqref{Odeltachi} is correct for $g\in C_0^{\infty}(\Omega)$, which is dense in $H_0^1(\Omega)$.
\end{proof}
\beg{rem}  Let $g$ {{ be }}a smooth compactly supported function in $\ov\Rr^2_{+}$ with $g(y) =1$ for $y\in [-\delta, \delta]\times \{0\}$.
Then the function $f(y) = y_2^2g(y)$  is smooth in $\ov\Rr^2_+$ , vanishes quadratically but $\mathcal O(f)$ has discontinuous second derivatives. This example shows that second derivatives of odd extensions of smooth functions which vanish quadratically need not be continuous.
\end{rem}
\beg{rem} The extension of the change of variables $x\mapsto Y(x)$ to the whole space does not necessarily map the whole domain $\Omega$ to the upper half plane, only a very small piece of it, near the boundary point $x_0=0$.  The extensions $\mathcal O$ and 
$\mathcal E$ can be used only on functions in $\Omega$ which have been properly localized near $x_0$.
\end{rem}

We compute now $\mathcal F(\chi e^{t\D}\theta)$. We denote $\rho = e^{t\D}\theta $ and therefore we have $\pa_t \rho= \D \rho$. Moreover,
$\D(\chi \rho) = \chi\D \rho + 2\na\chi\cdot\na \rho + (\Delta\chi)\rho$.  
Therefore, in view of \eqref{interdeltaL}
\be
(\pa_t +L)(\mathcal F(\chi \rho) )=-B\mathcal F (\chi \rho) -\mathcal F(2\na\chi\cdot\na \rho +(\Delta\chi) \rho):=g_{\theta}(t).
\la{heatone}
\ee
Using the Duhamel formula, we have shown the following proposition:
\beg{prop}\la{heaterr}
Let $\theta\in L^2(\Omega)$ and $t\ge 0$. Then
\begin{align}
\mathcal F(\chi e^{t\D}\theta) - e^{-tL}\mathcal F(\chi\theta) &=\int_0^t e^{-(t-s)L} g_{\theta}(s)ds =: R_{\theta}(t).\la{repheaterr}
\end{align}
\end{prop}
The right hand side $R_{\theta}(t)$  of \eqref{repheaterr} plays an important role. 

\beg{prop}\la{mboundR} Let $\theta\in L^{\infty}(\Omega)$. For any $0 \le \beta<2,0\leq r<1, \beta\geq r$, there exists a constant $C_{\chi}$ depending only on $\chi$ and $r$, $\text{diam}(\Omega)$, such that for any $t\ge 0$, we have
\be
\|R_{\theta}(t)\|_{C^\beta(\Rr^2)} \le C_{\chi}  \frac{\log(2+t)}{t+1}\min\{t,1\}^{\frac{r+1-\beta}{2}} \|\theta\|_{C^r(\Omega)}.
\la{Rmbound}
\ee
\end{prop}
\beg{proof} We use the bound
\be
\|e^{s\D}\theta\|_{C^{1}(\Omega)} \lesssim \min\{s,1\}^{-\frac{1-r}{2}}e^{-c_0 s}\|\theta\|_{C^{r}(\Omega)}
\la{heatomegalphaub}
\ee
valid for $0\le s$ and for some $c_0>0$.
This bound follows from a priori bounds on the heat kernel $H_D$ of the heat operator, see for instance \cite{sqgb}. Thus, in view of \eqref{z87a} and \eqref{z87b} we have 
\begin{equation}\label{z76}
||g_{\theta}(s)||_{L^\infty\cap L^1}\lesssim 	||g_{\theta}(s)||_{L^\infty}\lesssim \min\{s,1\}^{-\frac{1-r}{2}}e^{-c_0s}\|\theta\|_{C^{r}(\Omega)}.
\end{equation}
Using \eqref{z17a} and \eqref{z17c2}, \eqref{z76}, we obtain 
\begin{align}
&||R_\theta(t)||_{L^\infty}\lesssim	\int_0^t\frac{1}{1+t-s}||g_{\theta}(s)||_{L^\infty\cap L^1}ds \lesssim \frac{\min\{t,1\}^{\frac{1+r}{2}}}{1+t} \|\theta\|_{C^{r}(\Omega)},\\&
||\na R_\theta(t)||_{L^\infty}\lesssim	\int_0^t\frac{\log(2+t-s)}{1+t-s} \frac{1}{\min\{t-s,1\}^{\frac{1}{2}}}||g_{\theta}(s)||_{L^\infty\cap L^1}ds \lesssim \frac{\log(2+t)\min\{t,1\}^{\frac{r}{2}}}{t+1} \|\theta\|_{C^{r}(\Omega)}.
\end{align}
In view of \eqref{z67b} and \eqref{z67c},  \eqref{z76},  we obtain  for any $|h|\leq 1$. 
\begin{align}\nonumber
&	||\delta_h\na R_\theta(t)||_{L^\infty}\lesssim \int_{0}^{t}\mathbf{1}_{t-s\leq 1}  \left(\frac{\min\{\frac{|h|}{\sqrt{t-s}},1\}}{\sqrt{t-s}}+|h|\frac{\log(2+\frac{\sqrt{t-s}}{|h|})}{(t-s)^{\frac{1}{2}}}\right)\min\{s,1\}^{-\frac{1-r}{2}}e^{-c_0s} ds\|\theta\|_{C^{r}(\Omega)}\\&\quad\quad+ \int_{0}^{t}\mathbf{1}_{t-s> 1}|h|\log(2+\frac{1}{|h|})\frac{\log(2+t-s)}{t-s} \min\{s,1\}^{-\frac{1-r}{2}}e^{-c_0s} ds\|\theta\|_{C^{r}(\Omega)}\nonumber\\&\quad\lesssim {{\left(\min\{\frac{|h|}{\sqrt{t}},1\}+|h| \right)\log(\frac{2}{|h|})}}\min\{t,1\}^{\frac{r}{2}}e^{-\frac{c_0}{2}t}\|\theta\|_{C^{r}(\Omega)}+ \mathbf{1}_{t\geq 1}\frac{\log(2+t)}{t+1}|h|\log(\frac{2}{|h|})\|\theta\|_{C^{r}(\Omega)}.\label{z77}
\end{align}
{{Here we used the fact that 
\begin{equation}
	\int_{0}^{t}\mathbf{1}_{t-s\leq 1}  \frac{\min\{\frac{|h|}{\sqrt{t-s}},1\}}{\sqrt{t-s}}\lesssim  \log(\frac{2}{|h|})\sqrt{t}\min\{\frac{|h|}{\sqrt{t}},1\}.
\end{equation}}}
Therefore, we obtain \eqref{Rmbound}.
\end{proof}
\begin{rem} In view of \eqref{z77}, we have $\na R_\theta\in C^{1-\log}$ and 
\begin{equation}
\sup_{|h|\leq 1}\frac{	||\delta_h\na R_\theta(t)||_{L^\infty}}{|h|\log(2/|h|)}\lesssim  {{ \frac{\log(2+t)}{t+1}\min\{t,1\}^{\frac{r-1}{2}}\|\theta\|_{C^{r}(\Omega)}.}}
\end{equation}
This is an optimal regularity because $L$ has only Lipschitz coefficients. 
\end{rem}

We take $\theta\in L^{\infty}(\Omega)$ and consider  the stream function
\be
\psi_\theta = \l^{-1}\theta = \frac{1}{\Gamma(\frac{1}{2})}\int_0^{\infty} t^{-\fr{1}{2}}e^{t\D}\theta dt.
\la{rdtheta}
\ee
We have directly from \eqref{repheaterr}:
\beg{prop}\la{psierr} Let  $\theta\in L^{\infty}(\Omega)$.  We have
\be
\mathcal F(\chi \psi_{\theta}) - L^{-\fr{1}{2}}\mathcal F(\chi\theta) = \frac{1}{\Gamma(\frac{1}{2})}\int_0^{\infty}t^{-\fr{1}{2}} R_{\theta}(t)dt=: S_{\theta,\chi}
\la{psierrrep}
\ee
where 
\begin{equation}\label{ssmooth}
||S_{\theta,\chi}||_{C^\beta(\mathbb{R}^2)} \lesssim \frac{1}{2-\beta} \|\theta\|_{L^{\infty}(\Omega)}
\end{equation}
holds for all $0\leq \beta<2$. Here $L^{-\frac{1}{2}}$ is defined as the inverse operator of $L^{\frac{1}{2}}$ and is given by \eqref{inveropera}. 
\end{prop}
\beg{proof} In view of \eqref{Rmbound}, we have 
\begin{align}
||S_{\theta,\chi}||_{C^\beta(\mathbb{R}^2)}\lesssim \int_{0}^{\infty}t^{-\frac{1}{2}}\frac{\log(2+t)}{t+1}\min\{t,1\}^{\frac{1-\beta}{2}}dt \|\theta\|_{L^{\infty}(\Omega)}\lesssim \frac{1}{2-\beta} \|\theta\|_{L^{\infty}(\Omega)}.
\end{align}
This implies \eqref{ssmooth}.

\end{proof}

We represent the localized and extended operator relationship for $\l$.
\beg{prop}\la{lerrep} Let  $\theta\in L^{\infty}(\Omega)$.  Then
\be
\mathcal F(\chi\l\theta) - L^{\fr{1}{2}}\mathcal F(\chi \theta) = -\frac{1}{2\Gamma(\frac{1}{2})}\int_0^{\infty}t^{-\fr{3}{2}}R_{\theta}(t)dt=:\mathbf{R}_{\chi}(\theta)
\la{FLam}
\ee
holds. Moreover, we have 
\begin{equation}\label{CsR}
||\mathbf{R}_{\chi}(\theta) ||_{C^r(\mathbb{R}^2)}\lesssim \|\theta\|_{C^r(\Omega)},
\end{equation}
for any $0<r<1$.
\end{prop}
\beg{proof}
Using the heat operator representations of $\l$ and $L^{\fr{1}{2}}$, we have
\begin{align}
\mathcal F(\chi\l\theta) - L^{\fr{1}{2}}\mathcal F(\chi \theta) = \frac{1}{2\Gamma(\frac{1}{2})}\int_0^{\infty} t^{-\fr{3}{2}}\left( e^{-tL}\mathcal F(\chi\theta) -	\mathcal F(\chi e^{t\D}\theta)\right)dt.\label{commul}
\end{align}
Using \eqref{repheaterr} we arrive at \eqref{FLam}.\\ In view of \eqref{Rmbound}, we have for any $0<r<\beta<1$ and $|h|\leq 1$
\begin{align}\label{z78a}
&	|\mathbf{R}_{\chi}(\theta)|\lesssim \int_0^{\infty}t^{-\fr{3}{2}}\frac{\log(2+t)}{t+1}\min\{t,1\}^{\frac{r+1}{2}}dt \|\theta\|_{C^r(\Omega)}\lesssim \frac{1}{r}\|\theta\|_{C^r(\Omega)},\\&
|\delta_{h}\mathbf{R}_{\chi}(\theta)|\lesssim \int_0^{\infty}t^{-\fr{3}{2}}\frac{\log(2+t)}{t+1}\min\{t,1\}^{\frac{r+1}{2}}\min\left\{\frac{|h|}{\sqrt{\min\{t,1\}}},1\right\}^\beta dt \|\theta\|_{C^r(\Omega)}.\label{z78b}
\end{align}
We split {{ the integral   \eqref{z78b} in}} $\int_{0}^{\infty}=\int_{0}^{|h|^2}+\int_{|h|^2}^{2}+\int_{2}^\infty$ to get 
\begin{align}\nonumber
|\delta_{h}\mathbf{R}_{\chi}(\theta)|&\lesssim\left(  \int_{0}^{|h|^2}t^{-1+\frac{r}{2}} dt+|h|^\beta\int_{|h|^2}^2 t^{-1-\fr{\beta-r}{2}} dt+ |h|^\beta\int_{2}^\infty t^{-\fr{3}{2}}\frac{\log(2+t)}{t+1} dt\right) \|\theta\|_{C^r(\Omega)}\\&\lesssim |h|^r \|\theta\|_{C^r(\Omega)}.
\end{align}
Combining this with \eqref{z78a} we obtain the result. 
\end{proof}
\begin{rem} \label{remaforin}Similarly, we have 
\begin{align}\label{z82}
&	||\chi_{in} \psi_{\theta} - \Lambda_{\mathbb{R}^2}^{-1}(\chi_{in}\theta)||_{C^\beta(\mathbb{R}^2)}\lesssim  \|\theta\|_{L^{\infty}(\Omega)},\\&
||\chi_{in} \l\theta - \Lambda_{\mathbb{R}^2}(\chi_{in}\theta)||_{C^r(\mathbb{R}^2)}\lesssim  \|\theta\|_{C^r(\Omega)},\label{z82b}
\end{align}
for any $0<\beta<2$ and $0<r<1$	where $\chi_{in} $ is a cutoff function satisfying $\chi_{in} =1$ in $B(x_0,r_0)$ and   $\chi_{in} =0$ in $\mathbb{R}^2\backslash B(x_0,\frac{5}{4}r_0)$  with $B(x_0,2r_0)\subset \Omega$.
\end{rem}
\begin{rem} In view of Remark \ref{remaforin}, Proposition \ref{psierr} and \eqref{lerrep}, and \eqref{z27b}, we have the bounds
\begin{align}\label{z86a}
&		||\Lambda_D^{-1} \theta||_{C^r(\Omega)}\lesssim ||\theta||_{C^{(r-1)_+}(\Omega)},\quad   {{\text{for any }}}~~r\in (0,2)\backslash\{1\},\\&
||\Lambda_D^{-1} \theta||_{C^1(\Omega)}\lesssim ||\theta||_{C^r(\Omega)},\quad\quad\quad\text{for any}~~r\in (0,1),\label{z86b}\\&
||\Lambda_D \theta||_{C^{r}(\Omega)}\lesssim ||\theta||_{C^{1+r}(\Omega)},\quad\quad {{\text{for any }}}~~r\in (0,1).
\end{align}
\end{rem}

We consider now the localization and extension of the nonlinear term. We denote the usual Poisson bracket by $\{\psi, \theta\} =J(\psi,\theta)$,
\be
\{\psi, \theta\} = \na^{\perp}\psi\cdot\na \theta
\la{poipsitheta}
\ee
and use its behavior under composition 
\be
\{f\circ X, g\circ X\}  = (det\,\na X)\left(\{f, g\}\circ X\right).
\la{poinv}
\ee
Thus, in particular,
\be
\{{{\chi_1}}\psi, {{\chi}}\theta\}\circ X = \mathbf{\widetilde a}\cdot \{  (\chi_1\psi)\circ X, {{ ({{\chi}}\theta)}}\circ X\}.
\la{invpoi}
\ee
holds for smooth cutoffs $\chi, {\chi_1}$ supported in $\overline\Omega$, and where $\mathbf{\widetilde a}= (\det\na Y) \circ X)$ (see \eqref{aij}).\\
We also use the important observation that odd extensions commute with the Poisson bracket. This follows from the properties
\eqref{comuodeven} and from the product rules \eqref{prouodeven}.
We have thus, recalling our definition \eqref{awa},  $ \mathbf{a} = \mathcal E\mathbf{\widetilde a}$,
\be
\mathcal O( \{{{\chi_1}}\psi,  {{\chi}}\theta\}\circ X) =\mathbf{a}\cdot \{\mathcal O( {{ ({{\chi_1}}\psi) }}\circ X), \mathcal O( {{({{\chi}}\theta)}}\circ X)\}.
\ee
Therefore, we have
\be
\mathcal O\left((\na^{\perp}({{\chi_1}} \psi_\theta)\cdot\na({{\chi}}\theta))\circ X\right) = \mathbf{a} \na^{\perp}\left(\mathcal O( {{({{\chi_1}} \psi_\theta)}}\circ X \right)\cdot\na(\mathcal O( {{({{\chi}}\theta)}}\circ X)).
\ee
We proved
\beg{prop} \la{Fnon}Let  $\theta\in L^{\infty}(\Omega)$, let $\psi_{\theta}$ be a stream function defined by \eqref{rdtheta} and let $\chi_1$, $\chi$ be smooth cutoffs supported in $\overline{\Omega}$. Then
\be
\mathcal F\left(\na^{\perp}_x({{\chi_1}}\psi_\theta)\cdot \na_x ({{\chi}}\theta )\right) = \mathbf{a}\na^{\perp}_y(\mathcal F({{\chi_1}}\psi_\theta))\cdot \na_y\mathcal F({{\chi}}\theta)
\ee
holds.
\end{prop}

\section{Extended localized critical SQG}\la{exlsqg}
We start by computing, with $\chi = \chi_i^0$ and $\chi_1 = \chi_i^1$ two localizers, $1\le i\le N$,
\be
\ba
\chi (\na^{\perp}\psi)\cdot\na \theta = \na^{\perp}\psi\cdot\na(\chi\theta)-
\left(\na^{\perp}\psi \cdot\na\chi\right)\theta \\\quad\quad\quad\quad\quad\quad= \na^{\perp}(\chi_1\psi)\cdot\na(\chi\theta) -\left(\na^{\perp} (\chi_1\psi) \cdot\na\chi\right) \theta.
\ea
\la{locprodone}
\ee
The last equality follows because $\chi_1\equiv 1$ is on the support of $\chi$. 

Applying the product rules \eqref{prouodeven}, Proposition \ref{Fnon} we obtained
\beg{prop}\la{nonerr} Let $\theta\in L^{\infty}(\Omega)$, let 
$\psi_{\theta}$ be a stream function defined by \eqref{rdtheta}. Then we have
\be
\mathcal F\left(\chi(\na^{\perp}\psi_{\theta}\cdot \na\theta)\right) - \mathbf{a}\na^{\perp}\mathcal F(\chi_1\psi_\theta)\cdot \na \mathcal F(\chi\theta) = -\mathbf{a}\na^{\perp}\mathcal F(\chi_1\psi_{\theta})\cdot \mathcal F((\na\chi) \theta).
\la{nonerro}
\ee
\end{prop}
In view of \eqref{psierrrep}, we have 
\begin{equation}\label{z78}
\mathcal	F(\chi_1\psi_\theta)= L^{-\fr{1}{2}}\mathcal F(\chi_1\theta) +S_{\theta,\chi_1}.
\end{equation}
We denote by $\theta_i, \widetilde{\theta}_i$ the functions
\begin{align}\label{wideth}
&	\theta_i = \mathcal F(\chi\theta),\\&
\widetilde{\theta}_i=\mathcal F(\chi_1\theta),\label{overtheta}
\end{align}
by $u_i$
\begin{equation}\label{tilu}
{u_i}= \mathbf{a}\na^{\perp}L^{-\fr{1}{2}}(	\widetilde{\theta}_i) +u_{re},
\end{equation}
with 
\begin{equation}
u_{re}=\mathbf{a}\na^{\perp}S_{\theta,\chi_1},
\end{equation}
and by $	\widetilde{\gamma}$ the vector 
\begin{equation}\label{tilga}
\widetilde{\gamma}=	\mathcal F((\na\chi) \theta).
\end{equation}
Note that 
\begin{equation}
{\theta}_i=	\eta\widetilde{\theta_i},
\end{equation}
where 	 $\eta=\mathcal E(\chi\circ X)$ is a Lipschitz cutoff function satisfying $\eta=1$ in $B(x_0,r_1)$, $\eta=1$ in $\mathbb{R}^2\backslash B(x_0,4r_1)$ for some $x_0\in \mathbb{R}^2$ and $r_1>0$.  \\	
Multiplying  the SQG equation \eqref{sqg} by $\chi^0_i$, using the definitions \eqref{wideth}, \eqref{overtheta}, \eqref{tilu} and \eqref{tilga} above, using \eqref{z78}, \eqref{nonerro} and \eqref{FLam}, we arrive at
\be
\pa_t \theta_i+ 	{u_i}\cdot \na\theta_i + L^{\fr{1}{2}}\theta_i =f,
\la{widesqg}
\ee
where 
\begin{equation}
f=	-\mathbf{R}_{\chi}(\theta) +	{u_i}\cdot \widetilde{\gamma}
\end{equation}
for $i\le N_1$.
\begin{lemma} For any $0<r<1$ and $0<\beta<1$, the following inequalities hold 
\begin{align}\label{esu_re}
&	||u_{re}||_{C^r(\mathbb{R}^2)}\lesssim \|\theta\|_{L^{\infty}(\Omega)},\\&
||f||_{C^r(\mathbb{R}^2)}\lesssim ||\theta||_{C^r(\Omega)}\left(1+||\theta||_{C^{\beta}(\Omega)}\right).\label{fo}
\end{align}
\end{lemma}
\begin{proof} In view of \eqref{ssmooth}, \eqref{z27b}, \eqref{z87a} and \eqref{z87b}, we obtain \eqref{esu_re} and 
\begin{equation}
||{u_i}||_{C^\beta(\mathbb{R}^2)}\lesssim ||\theta||_{C^\beta(\Omega)}+||\theta||_{L^1(\Omega)}\lesssim ||\theta||_{C^\beta(\Omega)}.
\end{equation}
Combining this with  \eqref{CsR} and $	||\widetilde{\gamma}||_{C^\beta(\mathbb{R}^2)}\lesssim ||\theta||_{C^\beta(\Omega)}$ to obtain \eqref{fo}.
\end{proof}
Similarly, the equation for interior balls ($N_1+1 \le i\le N$) for $\theta_i = \chi_i^0\theta$ is
\begin{align}\label{z35a1}
&	\pa_t \theta_i + u_i\cdot\na \theta_i + \Lambda_{\mathbb{R}^2}\theta_i = f_{in},\\&
u_i = \na^{\perp}\Lambda_{\mathbb{R}^2}^{-1}(\chi_{1}\theta)+  u_{er,in},\label{z35b1}
\end{align}

where 
\begin{align}
&	f_{in}=\Lambda_{\mathbb{R}^2}(\chi\theta )- \chi\l (\theta )+(u\cdot\na \chi)\theta,\\&
u_{er,in}=\na^{\perp}(\chi_{1}\l^{-1}(\theta))-\na^{\perp}\Lambda_{\mathbb{R}^2}^{-1}(\chi_{1}\theta),\end{align}
and	 $\chi=\chi_i^0, \chi_1=\chi_i^1$ for $N_1<i\leq N$.\\ In view of Remark \ref{remaforin}, we have 
\begin{align}\label{esu_rein}
&	||u_{re,in}||_{C^r(\mathbb{R}^2)}\lesssim \|\theta\|_{L^{\infty}(\Omega)},\\&
||f_{in}||_{C^r(\mathbb{R}^2)}\lesssim ||\theta||_{C^r(\Omega)}\left(1+||\theta||_{C^{\beta}(\Omega)}\right),\label{foin}
\end{align}
for any $0<r,\beta<1$.

\section{Bounds for a Linear Dissipative Advection Equation}\la{bld}
In this section, we consider the linear advection equation
	\begin{equation}\label{z48}
	\partial_tv(x,t) + b(x,t)\cdot\na v(x,t)+L^{\fr12}v(x,t)=f(x,t),
\end{equation}
in $\mathbb{R}^d\times [0,T]$,  {{ with $d\geq 2$}} where $L^{\frac{1}{2}}$ is given by \eqref{z52}.  

	\begin{lemma} \label{estidiffusion} Assume that  $v\in L^\infty([0,T],C^{\alpha_0})\cap {{L^\infty_{loc}((0,T]}},C^{1+\alpha_0})$ is a solution of \eqref{z48} for some $\alpha_0\in (0,1)$. Then, the following inequalities hold for any $\alpha_1\in(0,1/3),\alpha_2\in (0,1)$ and $\alpha_2>\alpha_1$, 
	\begin{align}\nonumber
		||v||_{L^1_T(\dot C^{1+\alpha_1})}&\lesssim
		M_1{{(T+1)}}\left({{||v(0)||_{ C^{\alpha_2}}}}+||v||_{L^1_T(L^\infty)}+{{||f||_{L^1_T( C^{\alpha_2})}}}\right)\\&\quad\quad+M_1^{\frac{3}{\alpha_1}}{{(T+1)^{\frac{\alpha_2}{\alpha_1}}}}\int_{0}^{T} (1+	{{||b||_{ C^{\alpha_2}}}})^{\frac{\alpha_2}{\alpha_1^2}}||v||_{\dot C^{\alpha_1}}ds;\label{z31b}
	\end{align}
	and 
	\begin{align}\nonumber
		||v||_{L^\infty_{T}(\dot C^{2\alpha_1})}&+\sup_{s\in [0,T]}s^{1-\alpha_1}	||v(s)||_{\dot C^{1+\alpha_1}}\lesssim M_1 ||v(0)||_{\dot C^{2\alpha_1}}+M_1 \sup_{s\in [0,T]}s^{1-\alpha_1}||f(s)||_{\dot C^{\alpha_1}}
		\\&\quad\quad\quad\quad\quad+\left(M_1^{\frac{6}{\alpha_1}}T^{2(1-\alpha_1)}\left(1+	||b||_{L^\infty(\dot C^{\alpha_1})}\right)^{\frac{2}{\alpha_1}}+M_1 T^{1-\alpha_1} \right)||v||_{L^\infty_{T}(L^\infty)}\label{z57};
	\end{align}
	and 
	\begin{align}\nonumber
		&\sup_{s\in [0,T]}s^{1-\alpha_1}	||v(s)||_{\dot C^{1+\alpha_1}}\lesssim M_1 T^{\alpha_1}||v(0)||_{\dot C^{3\alpha_1}}+M_1 \sup_{s\in [0,T]}s^{1-\alpha_1}||f(s)||_{\dot C^{\alpha_1}}
		\\&\quad\quad\quad\quad\quad\quad\quad\quad+M_1^{\frac{3}{\alpha_1}}T^{1-\alpha_1}\left(1+ 	||b||_{L^\infty(\dot C^{\alpha_1})}\right)^{\frac{1}{\alpha_1}}||v||_{L^\infty_{T}(\dot C^{\alpha_1})}+M_1 T^{1-\alpha_1} ||v||_{L^\infty_{T}(L^\infty)},\label{z58}
	\end{align}
	where $M_1=1+||b||_{L^\infty_{T}(L^\infty)} $.
\end{lemma}

\begin{rem}
	 We see from \eqref{z58}, that $\sup_{s\in [0,T]}s^{1-\alpha_1}	||v(s)||_{\dot C^{1+\alpha_1}}$ is small when $T$ is small. This property  is used in the proof of the Lemma \ref{keyle} which is key for the proof of Theorem \ref{calphaglob}.
\end{rem}
\begin{rem} When $b(x,t)\equiv 0$, estimates \eqref{z31b}, \eqref{z57} and \eqref{z58} are proven by  K. Chen, R. Hu and the third author  in \cite[Theorem 1.1]{Chenhunguyen}. When $b(x,t)\not= 0$, these estimates are new. 
\end{rem}
The proof of Lemma \eqref{estidiffusion} is based on a method of freezing coefficients, to avoid directly differentiating $b$ or the variable coefficients of $L^{\fr{1}{2}}$.
We start by taking a fixed $x_0$ and computing the kernel of the semigroup generated by the constant coefficients operator $L_{x_0}^{\fr{1}{2}}$.  
A direct calculation verifies that
	\begin{align}
		\mathbf{H}_{A(x_0)}(x,t)=\frac{1}{\sqrt{\pi}}\int_{0}^{\infty}\frac{e^{-\rho}}{\sqrt{\rho}}  G_{A(x_0)}(x-y,\frac{t^2}{4\rho})d\rho=\frac{c_1}{\sqrt{\det A(x_0)}}\frac{t}{\left(t^2+|A(x_0)^{-\fr{1}{2}}x|^2\right)^{\frac{d+1}{2}}}
	\end{align}
	is the kernel of $\partial_t+L_{x_0}^{\frac{1}{2}}$, 
	where $c_1\int_{\mathbb{R}^d}\frac{dx}{\left(1+|x|^2\right)^{\frac{d+1}{2}}}=1$. Above $A(x_0)^{-\fr{1}{2}}$ is the square root of the positive symmetric matrix $A(x_0)^{-1}$.\vspace*{0.3cm}\\
		We write 
	\begin{align}
		\partial_tv(x,t) + b(x_0,t)\cdot \na v(x,t)+	L_{x_0}^{\frac{1}{2}}v(x,t)=F(x,x_0,t)
	\end{align}
	for any $x_0\in \mathbb{R}^d$	where 
	\begin{align}\nonumber
		F(x,x_0,t)&=\left\{f(x,t)+(L_{y}^{\frac{1}{2}}v(x,t)\vert_{y=x}-	L^{\frac{1}{2}}v(x,t))\right\}\\&\nonumber+	\left\{(L_{x_0}^{\frac{1}{2}}v(x,t)-	L_{y}^{\frac{1}{2}}v(x,t)\vert_{y=x})+(b(x_0,t)-b(x,t))\na v(x,t)\right\}\\&:=	F_1(x,t)+	F_2(x,x_0,t).
	\end{align}
	Using \eqref{z54} and \eqref{z53}, \eqref{z53b}, we have 
	\begin{align}
		&||F_1(t)||_{\dot C^{\alpha_2}}\lesssim||f(t)||_{\dot C^{\alpha_2}}+||v(t)||_{ C^{\alpha_2+\kappa}}\label{z49a},\\&
		|F_2(x,x_0,t)|\lesssim
		|x-x_0|^{\alpha_2}\left(1+	||b||_{L^\infty_{T}(\dot C^{\alpha_2})}\right)||v(t)||_{\dot C^{\alpha_1}}^{\alpha_1}||v(t)||_{\dot C^{1+\alpha_1}}^{1-\alpha_1},\label{z49b}
	\end{align}
and 
{{\begin{align}\nonumber
	&	|\delta_{h}^{x_0}F_2(x,x_0,t)-\delta_{h}^{x_0}F_2(y,x_0,t)|
\\&\quad\quad\leq\nonumber |(L_{x_0+h}^{\frac{1}{2}}-L_{x_0}^{\frac{1}{2}})v(x,t)-(L_{x_0+h}^{\frac{1}{2}}-L_{x_0}^{\frac{1}{2}})v(y,t)|+|\delta_{h}b(x_0,t)||\na v(x,t)-\na v(y,t)|
\\&\quad\quad	\lesssim |x-y|^\kappa|h|^{\alpha_2-\kappa} \left(1+ 	||b(t)||_{\dot C^{\alpha_2-\kappa}}\right) ||v(t)||_{\dot C^{1+\kappa}},		\label{z49c}
\end{align}}}
	for any $\kappa\in (0,\min\{1-\alpha_2,\alpha_2\})$.\\
	 Now we  compute the kernel of $	\partial_t + b(x_0,t)\na +	L_{x_0}^{\frac{1}{2}}$. 
	Integrating by parts, we have 
	\begin{align}\nonumber
		&\int_{0}^{t}	\int_{\mathbb{R}^d}		\mathbf{H}_{A(x_0)}(x-y-\int_{s}^{t} b(x_0,\tau)d\tau,t-s) 	F(y,x_0,s)dyds\\&\nonumber= \int_{0}^{t}	\int_{\mathbb{R}^d}	\mathbf{H}_{A(x_0)}(x-y-\int_{s}^{t} b(x_0,\tau)d\tau,t-s) \left(\partial_sv(y,s) + b(x_0,s)\na v(y,s)+L^{\fr12}_{x_0}v(y,s)\right)dyds\\&\nonumber=v(x,t)-	\int_{\mathbb{R}^d}	\mathbf{H}_{A(x_0)}(x-y-\int_{0}^{t} b(x_0,\tau)d\tau,t) v(y,0)dy \\&\nonumber- \int_{0}^{t}	\int_{\mathbb{R}^d}	\partial_s\left[\mathbf{H}_{A(x_0)}(x-y-\int_{s}^{t} b(x_0,\tau)d\tau,t-s)\right] v(s,y) dyds\\&\nonumber+ \int_{0}^{t}	\int_{\mathbb{R}^d}	(b(x_0,s)\na_x+L_{x_0}^{\fr12}\mathbf{H}_{A(x_0)})(x-y-\int_{s}^{t} b(\tau,x_0)d\tau,t-s) v(y,s)dyds\\&=v(x,t)-	\int_{\mathbb{R}^d}	\mathbf{H}_{A(x_0)}(x-y-\int_{0}^{t} b(x_0,\tau)d\tau,t) v(y,0)dy.\end{align}
	Thus, 
	\begin{align}\nonumber
		v(x,t)&=\int_{\mathbb{R}^d}	\mathbf{H}_{A(x_0)}(x-y-\int_{0}^{t} b(\tau,x_0)d\tau,t) v(y,0)dy\\&+\int_{0}^{t}	\int_{\mathbb{R}^d}		\mathbf{H}_{A(x_0)}(x-y-\int_{s}^{t} b(\tau,x_0)d\tau,t-s) 	F(y,x_0,s)dyds,\label{z50}
	\end{align}
	for any $x_0\in \mathbb{R}^d$.\\
	
	This verifies that the map  $(t,s,x,y)\to \mathbf{H}_{A(x_0)}(t-s,x-y-\int_{s}^{t} b(\tau,x_0)d\tau) $ is the kernel of
	the semigroup generated by the operator $-\left(b(t,x_0)\cdot\na +	L_{x_0}^{\frac{1}{2}}\right)$.\vspace*{0.3cm}\\

In the proof of Lemma \eqref{estidiffusion}, we make use of the following bound. 
\begin{lemma}\la{betalemma} For $\beta\in [0,1)$, and $j=0,1$, we have
	\begin{equation}\label{z51}
		\sup_x	\int_{\mathbb{R}^d}|(\delta_h^z\na^j_z	\mathbf{H}_{A(x)})|(z,t)|z|^\beta dz\lesssim\min\{\frac{|h|}{t},1\}^{1-\beta} \frac{|h|^\beta}{t^{j}}.
	\end{equation}
\end{lemma}
\begin{proof} Case 1: $ |h|\leq 4 t$. We have 
	\begin{align}
		|(\delta_h^z\na^j_z	\mathbf{H}_{A(x)})|(z,t)\lesssim \frac{|h| t}{(t+|z|)^{d+j+2}}. 
	\end{align}
	So, 
	\begin{align}
		\sup_x	\int_{\mathbb{R}^d}|(\delta_h\na^j	\mathbf{H}_{A(x)})|(z,t)|z|^\beta dz\lesssim  \int_{\mathbb{R}^d} \frac{|h| t}{(t+|z|)^{d+j+2}}|z|^\beta dz\sim \frac{|h|}{t^{j+1-\beta}}.
	\end{align}
	Case 2: $ |h|\geq  4 t$. We have
	\begin{align}
		&	\sup_x	\int_{\mathbb{R}^d}|(\delta_h\na^j	\mathbf{H}_{A(x)})|(z,t)|z|^\beta dz\lesssim  \int_{\mathbb{R}^d}\left( \frac{t}{(t+|z-h|)^{d+j+1}}+ \frac{t}{(t+|z|)^{d+j+1}}\right)|z|^\beta dz\lesssim\frac{|h|^\beta}{t^{j}}.
	\end{align}
	The two cases together yield the result. 
\end{proof}
\begin{proof}[Proof of Lemma \ref{estidiffusion}]  
 We apply $\delta_h\na_x^j, j=0,1$  to both sides of \eqref{z50}, then we take $x_0=x$ to obtain  that
 \begin{align}\nonumber
		&\delta_h\na_x^jv(x,t)=\int_{\mathbb{R}^d}		(\delta_{h}\na^j\mathbf{H}_{A(x_0)})(x-y-\int_{0}^{t} b(x_0,\tau)d\tau,t)\vert_{x_0=x} v(y,0)dy\\&\quad+\int_{0}^{t}	\int_{\mathbb{R}^d}		(\delta_h\na^j	\mathbf{H}_{A(x_0)})(x-y-\int_{s}^{t} b(x_0,\tau)d\tau,t-s)\vert_{x_0=x} \left(F_1(y,s)+	F_2(y,x,s)\right)dyds.\label{z70}
	\end{align}
	Then, we  write 
	\begin{align}\nonumber
		&	\delta_h\na_xv(x,t)=\int_{\mathbb{R}^d}		(\na\mathbf{H}_{A(x_0)})(x-y-\int_{0}^{t} b(x_0,\tau)d\tau,t)\vert_{x_0=x} ( \delta_{-h}v(y,0)-\delta_{-h}v(x,0))dy\\&\nonumber+\int_{0}^{t}\mathbf{1}_{\frac{|h|}{t-s}\leq 1}	\int_{\mathbb{R}^d}		(\delta_{h}\na	\mathbf{H}_{A(x_0)})(x-y-\int_{s}^{t} b(x_0,\tau)d\tau,t-s)\vert_{x_0=x} (F_1(y,s)-F_1(x,s)+F_2(y,x,s))dyds\nonumber
		\\&\nonumber+\int_{0}^{t}\mathbf{1}_{\frac{|h|}{t-s}> 1}	\int_{\mathbb{R}^d}		(\na	\mathbf{H}_{A(x_0)})(x-y-\int_{s}^{t} b(x_0,\tau)d\tau,t-s)\vert_{x_0=x} 	(\delta_{-h}F_1(y,s)-\delta_{-h}F_1(x,s))dyds
		\\&\nonumber+\sum_{i=0,1}(-1)^{i+1}\int_{0}^{t}\mathbf{1}_{\frac{|h|}{t-s}> 1}		\int_{\mathbb{R}^d}		(\na	\mathbf{H}_{A(x_0)})(x+ih-y-\int_{s}^{t} b(x_0,\tau)d\tau,t-s)\vert_{x_0=x} 	F_2(y,x+ih,s)dyds
		\\&\nonumber+\int_{0}^{t}\mathbf{1}_{\frac{|h|}{t-s}> 1}		\int_{\mathbb{R}^d}		(\na	\mathbf{H}_{A(x_0)})(x+h-y-\int_{s}^{t} b(x_0,\tau)d\tau,t-s)\vert_{x_0=x}\\&\quad\quad\quad\quad\quad\quad\times \left(	F_2(y,x,s)-	F_2(y,x+h,s)- (	F_2(x+h,x,s)-	F_2(x+h,x+h,s))\right)dyds.
	\end{align}
	Here we used the fact that
	\begin{align}&
		\int_{\mathbb{R}^d}(\delta_h u_1)(x-y)u_2(y)dy=\int_{\mathbb{R}^d}u_1(x-y)\delta_{-h} u_2(y)dy,~~\forall~~u_1,u_2,\\&
		\int_{\mathbb{R}^d}		(\na	\mathbf{H}_{A(x_0)})(x-y-\int_{s}^{t} b(x_0,\tau)d\tau,t-s) \vert_{x_0=x} 	dy=0~~\forall~~x,s,t.
	\end{align}
	Thus, using \eqref{z49a}, \eqref{z49b} and \eqref{z49c} we have for any $ \alpha_1,\alpha_2,\alpha_4\in (0,1)$, $\alpha_4\geq \alpha_1,\alpha_1\leq 2\alpha_2$
	\begin{align}\nonumber
		&|\delta_h\na v(x,t)|\leq \int_{\mathbb{R}^d}		|(\na\mathbf{H}_{A(x_0)})|(x-y-\int_{0}^{t} b(x_0,\tau)d\tau,t)\vert_{x_0=x}|h|^{\alpha_1} |x-y|^{\alpha_4-\alpha_1}dy||v(0)||_{\dot C^{\alpha_4}}
		\\&\nonumber\quad+\int_{0}^{t}\mathbf{1}_{\frac{|h|}{t-s}\leq  1}	\int_{\mathbb{R}^d}		|(\delta_h\na\mathbf{H}_{A(x_0)})|(x-y-\int_{s}^{t} b(x_0,\tau)d\tau,t-s)\vert_{x_0=x}|x-y|^{\alpha_2}E(s)dyds\\&\quad\nonumber+\int_{0}^{t}\mathbf{1}_{\frac{|h|}{t-s}> 1}	\int_{\mathbb{R}^d}		|\na_x	\mathbf{H}_{A(x_0)}|(x-y-\int_{s}^{t} b(x_0,\tau)d\tau,t-s)\vert_{x_0=x}	|h|^{\frac{\alpha_1}{2}}|x-y|^{\alpha_2-\frac{\alpha_1}{2}}E(s)dyds
		\\&\quad\nonumber+\sum_{i=0,1}\int_{0}^{t}\mathbf{1}_{\frac{|h|}{t-s}> 1}	\int_{\mathbb{R}^d}		|\na_x	\mathbf{H}_{A(x_0)}|(x+ih-y-\int_{s}^{t} b(x_0,\tau)d\tau,t-s)\vert_{x_0=x}	|x+ih-y|^{\alpha_2}E(s)dyds\\&\quad\nonumber
		+\int_{0}^{t}\mathbf{1}_{\frac{|h|}{t-s}> 1}\int_{\mathbb{R}^d}		|(\na_x	\mathbf{H}_{A(x_0)})|(x+h-y-\int_{s}^{t} b(x_0,\tau)d\tau,t-s)\vert_{x_0=x}\\&\quad\quad\quad\quad\quad\times|x+h-y|^\kappa dy	{{ |h|^{\alpha_2-\kappa}\left(1+ ||b(s)||_{\dot C^{\alpha_2-\kappa}}\right)}} ||v(s)||_{\dot C^{1+\kappa}}ds,
	\end{align}
	where 
	\begin{align}\nonumber
		E(s):&=||f(s)||_{\dot C^{\alpha_2}}+\left(1+	||b(s)||_{\dot C^{\alpha_2}}\right)||v(s)||_{\dot C^{\alpha_1}}^{\alpha_1}||v(s)||_{\dot C^{1+\alpha_1}}^{1-\alpha_1}+||v(s)||_{ C^{\alpha_2+\kappa}}\\&\lesssim ||f(s)||_{\dot C^{\alpha_2}}+\left(1+	||b(s)||_{\dot C^{\alpha_2}}\right)||v(s)||_{\dot C^{\alpha_1}}^{\alpha_1}||v(s)||_{\dot C^{1+\alpha_1}}^{1-\alpha_1}+||v(s)||_{L^\infty}.\label{z68}
	\end{align}
	Using \eqref{z51}, we deduce
	\begin{align}\nonumber
		&	\int_{\mathbb{R}^d}		|(\delta_h\na	\mathbf{H}_{A(x)})|(y-\int_{s}^{t} b(x,\tau)d\tau,t-s)|y|^{\beta}dy\\&\nonumber\quad\quad\quad\quad\leq \int_{\mathbb{R}^d}		|(\delta_h\na	\mathbf{H}_{A(x)})|(z,t-s)(|z|+|t-s| ||b{{||}}_{L^\infty_{T}(L^\infty)} )^{\beta}dy\\&\quad\quad\quad\quad\lesssim M_1 \min\{\frac{|h|}{t-s},1\}^{1-\beta} \frac{|h|^{\beta}}{t-s},\\&
		\int_{\mathbb{R}^d}		|(\na	\mathbf{H}_{A(x)})|(y-\int_{s}^{t} b(x,\tau)d\tau,t-s)|y|^{\beta}dy\lesssim  M_1(t-s)^{-1+\beta},
	\end{align}
	for any $\beta\in (0,1)$.
	Thus, 
	\begin{align}\nonumber
		|\delta_h\na v(x,t)|&\lesssim M_1|h|^{\alpha_1} t^{-1+\alpha_4-\alpha_1}||v(0)||_{\dot C^{\alpha_4}}+M_1\int_{0}^{t}\min\{\frac{|h|}{t-s},1\}^{1-\frac{\alpha_1}{2}}	\frac{|h|^{\frac{\alpha_1}{2}}}{(t-s)^{1-\alpha_2+\frac{\alpha_1}{2}}}E(s)ds\\&
		+M_1\int_{0}^{t}\mathbf{1}_{\frac{|h|}{t-s}> 1}(t-s)^{-1+\kappa} |h|^{\alpha_2-\kappa} \left(1+ 	||b(s)||_{\dot C^{\alpha_2-\kappa}}\right) ||v(s)||_{\dot C^{1+\kappa}}ds.\label{z55}
	\end{align}
	3) {{Because $||g||_{ C^{\beta_1}}\lesssim ||g||_{ C^{\beta_2}}$  for any $0<\beta_1<\beta_2$, it is enough to prove \eqref{z31b} when $<\alpha_2-\alpha_1 \ll \alpha_1$}}.  Using \eqref{z55} with  $\alpha_4=\alpha_2> \alpha_1$ and $\kappa=\alpha_2-\alpha_1\ll \alpha_1$, we obtain
	\begin{align}\nonumber
		||v(t)||_{\dot C^{1+\alpha_1}}&\lesssim M_1 t^{-1+\alpha_2-\alpha_1}||v(0)||_{\dot C^{\alpha_2}}+M_1\int_{0}^{t}	\frac{E(s)ds}{(t-s)^{1-(\alpha_2-\alpha_1)}}\\&\quad\quad\quad\quad\nonumber
		+M_1\int_{0}^{t}\left(1+	||b(s)||_{\dot C^{\alpha_1}}\right) ||v(s)||_{\dot C^{1+\alpha_2-\alpha_1}}\frac{ds}{(t-s)^{1-(\alpha_2-\alpha_1)}}\\&\lesssim\nonumber M_1 t^{-1+\alpha_2-\alpha_1}||v(0)||_{\dot C^{\alpha_2}}+M_1\int_{0}^{t}	\frac{E(s)ds}{(t-s)^{1-(\alpha_2-\alpha_1)}}\\&\quad\quad\quad\quad
		+M_1^2\int_{0}^{t}\left(1+	||b(s)||_{\dot C^{\alpha_2}}\right)^{\frac{\alpha_1}{\alpha_2}}||v||_{\dot C^{\alpha_1}}^{2\alpha_1-\alpha_2}||v||_{\dot C^{1+\alpha_1}}^{1+\alpha_2-2\alpha_1}\frac{ds}{(t-s)^{1-(\alpha_2-\alpha_1)}}.
	\end{align}
	In the last inequality we used interpolation inequalities.\\
	Thus, in view of  \eqref{z68} we obtain
	\begin{align}\nonumber
		&||v||_{L^1_T(\dot C^{1+\alpha_1})}
		\lesssim
		M_1T^{\alpha_2-\alpha_1}\left(||v(0)||_{\dot C^{\alpha_2}}+||v||_{L^1_T(L^\infty)}+||f||_{L^1_T(\dot C^{\alpha_2})}\right)\\&+M_1^2T^{\alpha_2-\alpha_1}\left(\int_{0}^{T} (1+	||b||_{\dot C^{\alpha_2}})^{\frac{\alpha_1}{\alpha_2}}||v||_{\dot C^{\alpha_1}}^{2\alpha_1-\alpha_2}||v||_{\dot C^{1+\alpha_1}}^{1+\alpha_2-2\alpha_1}+\left(1+	||b||_{\dot C^{\alpha_2}}\right)||v||_{\dot C^{\alpha_1}}^{\alpha_1}||v||_{\dot C^{1+\alpha_1}}^{1-\alpha_1}\right).
	\end{align}
	Using Holder's inequality, we deduce
	\begin{align}\nonumber
		&||v||_{L^1_T(\dot C^{1+\alpha_1})}\lesssim
		M_1T^{\alpha_2-\alpha_1}\left(||v(0)||_{\dot C^{\alpha_2}}+||v||_{L^1_T(L^\infty)}+||f||_{L^1_T(\dot C^{\alpha_2})}\right)\nonumber\\&+M_1^{\frac{2}{2\alpha_1-\alpha_2}}T^{\frac{\alpha_2-\alpha_1}{2\alpha_1-\alpha_2}}\int_{0}^{T} (1+	||b||_{\dot C^{\alpha_2}})^{\frac{\alpha_1}{\alpha_2(2\alpha_1-\alpha_2)}}||v||_{\dot C^{\alpha_1}}+M_1^{\frac{2}{\alpha_1}}T^{\frac{\alpha_2}{\alpha_1}-1}\int_{0}^{T} \left(1+	||b||_{\dot C^{\alpha_2}}\right)^{\frac{1}{\alpha_1}}||v||_{\dot C^{\alpha_1}}\nonumber\\&\lesssim
	M_1(T+1)\left(||v(0)||_{ C^{\alpha_2}}+||v||_{L^1_T(L^\infty)}+||f||_{L^1_T( C^{\alpha_2})}\right)+M_1^{\frac{3}{\alpha_1}}(T+1)^{\frac{\alpha_2}{\alpha_1}}\int_{0}^{T} (1+	||b||_{ C^{\alpha_2}})^{\frac{\alpha_2}{\alpha_1^2}}||v||_{\dot C^{\alpha_1}}.
	\end{align}
	Here we  used  $\frac{\alpha_1}{\alpha_2(2\alpha_1-\alpha_2)},\frac{1}{\alpha_1}<\frac{\alpha_2}{\alpha_1^2}$ when 	{{$\alpha_2-\alpha_1\ll \alpha_1$}} . This implies \eqref{z31b}.\\
	4) Using \eqref{z55} with $\alpha_1=\alpha_2\leq \alpha_4/2$ and $\kappa=\frac{\alpha_1}{2}$, we have 
	\begin{align}\nonumber
		&|\delta_h\na v(x,t)|\lesssim  M_1|h|^{\alpha_1} t^{-1+\alpha_4-\alpha_1}||v(0)||_{\dot C^{\alpha_4}}+M_1\int_{0}^{t}\min\{\frac{|h|}{t-s},1\}^{1-\frac{\alpha_1}{2}}	\frac{|h|^{\frac{\alpha_1}{2}}s^{-1+\alpha_1}}{(t-s)^{1-\frac{\alpha_1}{2}}}ds\sup_{s\in [0,T]} s^{1-\alpha_1} E(s)\nonumber\\&\quad
		+M_1\left(1+ 	||b||_{L^\infty(\dot C^{\frac{\alpha_1}{2}})}\right)|h|^{\frac{\alpha_1}{2}}\int_{0}^{t}\mathbf{1}_{\frac{|h|}{t-s}> 1}(t-s)^{-1+\frac{\alpha_1}{2}}  s^{-1+\alpha_1}ds\sup_{s\in [0,T]}s^{1-\alpha_1}||v(s)||_{\dot C^{1+\frac{\alpha_1}{2}}}\nonumber\\&\lesssim M_1|h|^{\alpha_1} t^{-1+\alpha_4-\alpha_1}||v(0)||_{\dot C^{\alpha_4}}+M_1|h|^{\alpha_1} t^{-1+\alpha_1}\sup_{s\in [0,T]} s^{1-\alpha_1} E(s)\nonumber\\&\quad
		+M_1\left(1+ 	||b||_{L^\infty(\dot C^{\frac{\alpha_1}{2}})}\right)|h|^{\alpha_1} t^{-1+\alpha_1}||v||_{L^\infty_{T}(\dot C^{\alpha_1})}^{\frac{\alpha_1}{2}}\sup_{s\in [0,T]}s^{1-\alpha_1}||v(s)||_{\dot C^{1+\alpha_1}}^{1-\frac{\alpha_1}{2}},
	\end{align}
	where $E(s)$ satisfies \eqref{z68} with $\alpha_2=\alpha_1$.\\
	Thus, 
	\begin{align}\nonumber
		&\sup_{s\in [0,T]}s^{1-\alpha_1}	||v(s)||_{\dot C^{1+\alpha_1}}\lesssim M_1 T^{\alpha_4-2\alpha_1}||v(0)||_{\dot C^{\alpha_4}}+M_1 T^{1-\alpha_1} ||v||_{L^\infty_{T}(L^\infty)}+M_1\sup_{s\in [0,T]} s^{1-\alpha_1} ||f(s)||_{\dot C^{\alpha_1}}\\&+M_1T^{\alpha_1(1-\alpha_1)}\left(1+	||b||_{L^\infty_{T}(\dot C^{\alpha_1})}\right)||v||_{L^\infty_{T}(\dot C^{\alpha_1})}^{\alpha_1}\left(\sup_{s\in [0,T]}s^{1-\alpha_1}||v(s)||_{\dot C^{1+\alpha_1}}\right)^{1-\alpha_1}\nonumber\\&\quad
		+M_1^{\frac{3}{2}}T^{\frac{\alpha_1(1-\alpha_1)}{2}}\left(1+ 	||b||_{L^\infty(\dot C^{\alpha_1})}\right)^{\frac{1}{2}}||v||_{L^\infty_{T}(\dot C^{\alpha_1})}^{\frac{\alpha_1}{2}}\left(\sup_{s\in [0,T]}s^{1-\alpha_1}||v(s)||_{\dot C^{1+\alpha_1}}\right)^{1-\frac{\alpha_1}{2}}.\label{z72}
	\end{align}
	Here we used the interpolation inequality $\left(1+ 	||b||_{L^\infty(\dot C^{\frac{\alpha_1}{2}})}\right)\lesssim M_1^{\frac{1}{2}}\left(1+ 	||b||_{L^\infty(\dot C^{\alpha_1})}\right)^{\frac{1}{2}}$.\\
	Using Holder's inequality, 
	\begin{align}\nonumber
		&\sup_{s\in [0,T]}s^{1-\alpha_1}	||v(s)||_{\dot C^{1+\alpha_1}}\lesssim M_1 T^{\alpha_4-2\alpha_1}||v(0)||_{\dot C^{\alpha_4}}+M_1 T^{1-\alpha_1} ||v||_{L^\infty_{T}(L^\infty)}+M_1\sup_{s\in [0,T]} s^{1-\alpha_1} ||f(s)||_{\dot C^{\alpha_1}}\\&
		+M_1^{\frac{3}{\alpha_1}}T^{1-\alpha_1}\left(1+ 	||b||_{L^\infty(\dot C^{\alpha_1})}\right)^{\frac{1}{\alpha_1}}||v||_{L^\infty_{T}(\dot C^{\alpha_1})}.\label{z69}
	\end{align}
	This implies \eqref{z58} by taking {{$\alpha_4=3\alpha_1$.}}\\
	5) It follows from \eqref{z70}, \eqref{z49a} and \eqref{z49b} that for $\alpha_1=\alpha_2$
	\begin{align}\nonumber
		&	|\delta_hv(x,t)|\lesssim\int_{\mathbb{R}^d}		\mathbf{H}_{A(x_0)}(x-y-\int_{0}^{t} b(x,\tau)d\tau,t)\vert_{x_0=x} |\delta_{-h}v(y,0)|dy\\&\nonumber+\int_{0}^{t}	\int_{\mathbb{R}^d}			|\delta_h\mathbf{H}_{A(x_0)}(x-y-\int_{s}^{t} b(x,\tau)d\tau,t-s)\vert_{x_0=x}||y-x|^{\alpha_1}dyE(s)ds\\&\overset{\eqref{z51}}\lesssim |h|^{2\alpha_1}||v(0)||_{\dot C^{2\alpha_1}}+ M_1 \int_{0}^{t} |h|^{\alpha_1}\min\{\frac{|h|}{t-s},1\}^{1-\alpha_1}E(s)ds\nonumber
		\\&\nonumber\lesssim |h|^{2\alpha_1}||v(0)||_{\dot C^{2\alpha_1}}+ M_1 \int_{0}^{t} |h|^{\alpha_1}\min\{\frac{|h|}{t-s},1\}^{1-\alpha_1}s^{-1+\alpha_1} ds \sup_{s\in [0,T]} s^{1-\alpha_1} E(s)
		\\&\sim |h|^{2\alpha_1}\left(||v(0)||_{\dot C^{2\alpha_1}}+ M_1  \sup_{s\in [0,T]} s^{1-\alpha_1} E(s)\right),
	\end{align}
	where  $E(s)$ satisfies \eqref{z68} with $\alpha_2=\alpha_1$.\\
	Therefore, as in \eqref{z72} we obtain 
	\begin{align}\nonumber
		&	||v||_{L^\infty_{T}(\dot C^{2\alpha_1})}\lesssim	||v(0)||_{\dot C^{2\alpha_1}}+M_1 T^{1-\alpha_1} ||v||_{L^\infty_{T}(L^\infty)}+M_1\sup_{s\in [0,T]} s^{1-\alpha_1} ||f(s)||_{\dot C^{\alpha_1}}\\&+M_1T^{\alpha_1(1-\alpha_1)}\left(1+	||b||_{L^\infty_{T}(\dot C^{\alpha_1})}\right)||v||_{L^\infty_{T}(\dot C^{\alpha_1})}^{\alpha_1}\left(\sup_{s\in [0,T]}s^{1-\alpha_1}||v(s)||_{\dot C^{1+\alpha_1}}\right)^{1-\alpha_1}.
	\end{align}
	Combining this with \eqref{z69} (for $\alpha_4=2\alpha_2$) and Holder's inequality, we deduce
	\begin{align}\nonumber
		&||v||_{L^\infty_{T}(\dot{ C}^{2\alpha_1})}+\sup_{s\in [0,T]}s^{1-\alpha_1}	||v(s)||_{\dot C^{1+\alpha_1}}\lesssim M_1 ||v(0)||_{\dot {{ C^{2\alpha_1}}}}+M_1 T^{1-\alpha_1} ||v||_{L^\infty_{T}(L^\infty)}\\&+M_1\sup_{s\in [0,T]} s^{1-\alpha_1} ||f(s)||_{\dot C^{\alpha_1}}
		+M_1^{\frac{3}{\alpha_1}}T^{1-\alpha_1}\left(1+	||b||_{L^\infty(\dot C^{\alpha_1})}\right)^{\frac{1}{\alpha_1}}||v||_{L^\infty_{T}(\dot C^{\alpha_1})}.
	\end{align}
	Now \eqref{z57} follows because $||v||_{L^\infty_{T}(\dot C^{\alpha_1})}\lesssim ||v||_{L^\infty_{T}(L^\infty)}^{\frac{1}{2}}||v||_{L^\infty_{T}(\dot C^{2\alpha_1})}^{\frac{1}{2}}$  {{ and Holder's inequality.}}
\end{proof}

\section{The Nonlinear Maximum Principle}\la{nlmp}
We consider a solution $\theta$ of \eqref{sqg} defined on an interval of time $[0,T]$. In this section we assume that the solution belongs to 	{{$C^{1+\alpha_0}([0,T]\times\Omega)$}} for some $0<\alpha_0<1$. The local existence Theorem \eqref{locex} 	{{guarantees}} this to be the case if $T$ is small. Here we obtain the basic a priori estimates which allow the continuation of the solution.
We consider one of the boundary or interior balls $1\le i\le N$, and take the function $\theta_i =\chi_i^0\theta $ as the basic variable. The extended localized equation \eqref{widesqg}  for $i\le N_1$ and its interior counterpart for $N_1+ 1 \le i\le N$,  \eqref{z35a1}, are both represented below by the equation

\begin{align}\label{sqg4}
\partial_t \theta+ u\cdot\na\theta +L_1^{\fr12}\theta=f
\end{align}
in $\mathbb{R}^{d}$, $d\geq 2$, where $\theta=\eta \tilde{\theta}$ and $u$ obeys
\be
u =  \mathbf{J}\na L_2^{-\fr12} (\tilde{\theta})+u_1.
\la{u4}
\ee
with $u_1$ a lower order term, and where
  \begin{description}
	\item[i] $\mathbf{J}=\mathbf{J}(x)$ is a matrix valued function which satisfies 
	\begin{equation}\label{z89}
		||\mathbf{J}||_{W^{1,\infty}(\mathbb{R}^d)}\lesssim 1
	\end{equation}
\item[ii] $L_j=-\operatorname{div}(A_j\na.), j=1,2$, $A_1,A_2$ are  symmetric matrix valued functions in $\mathbb{R}^d$ which  satisfy \eqref{z65a} and \eqref{z65b}; 
\item[iii]$\eta$ is a Lipschitz cutoff function $\eta=1$ in $B(0,r_0)$ and $\eta=0$ in $B(0,2r_0)^c.$  
\end{description}

 We reiterate that in this section we use the name $\theta$, but this variable corresponds to one of the $\theta_i$ and not to the solution $\theta$ of \eqref{sqg}, \eqref{u}.
 In \eqref{u4} have implicitly assumed that $\widetilde\theta$ is known. In the application to \eqref{widesqg}, $\widetilde\theta$ is compactly supported obtained from the solution of \eqref{sqg}, \eqref{u} by the formula  \eqref{overtheta}. 
 
We apply the method of nonlinear maximum principle introduced in  \cite{cv1} for the whole space, used in \cite{cvt} in the periodic case and in \cite{sqgb} to establish interior H\"{o}lder bounds in bounded domains. In this section we do not explicitly use the divergence-free property of velocity $u$.  We expand here the range of applicability of the method to allow for nonlinear forcing and the absence of translation invariance. We make use of the following result (\cite[Lemma B.1]{cvt}).

\begin{lemma}\label{consttarvic}Let $\Upsilon(t,z): [0,T]\times \mathbb{R}^m\to [0,\infty),{{m\in \mathbb{N}}}$ be such that  $\text{supp}\Upsilon(t,.)\subset B_R\subset \mathbb{R}^m$ for any $t\in [0,T]$ and for some $R>0$. Assume that $\Upsilon(t,z)\in C^{1}([0,T],C^{\beta}( \mathbb{R}^m))\cap C^{\beta}([0,T],C^{1}( \mathbb{R}^m))$ for some $\beta\in (0,1)$. Let  $\varrho(t)=\sup_{z}\Upsilon(t,z)$ for any $t\in [0,T]$.  Then, $\varrho$ is Lipschitz continuous in $[0,T]$ with 
\begin{align}
||\varrho||_{Lip([0,T])}\leq  ||\partial_t\Upsilon||_{L^\infty([0,T]\times\mathbb{R}^m)},
\end{align}	
for almost every $t\in [0,T]$ the function $\varrho$ is differentiable at $t$ and there exists $z(t)\in \mathbb{R}^m$ such that 
simultaneously
\begin{equation}
\frac{d}{dt}\varrho(t)=(\partial_t\Upsilon)(t,z(t))~~\text{and}~~\varrho(t)=\Upsilon(t,z(t))
\end{equation}
hold. In particular, $\varrho$ is  absolutely continuous functions on $[0,T]$ and 
\begin{equation}\label{z80}
\varrho(t_2)=\varrho(t_1)+\int_{t_1}^{t_2}\frac{d}{dt}\varrho(t) dt
\end{equation}
for any $0\leq t_1<t_2\leq T$. 
\end{lemma}

\begin{lemma} \la{drivencalpha}Assume that $\theta\in C^{1+\alpha_0}([0,T],C^{\alpha_0}(\mathbb{R}^d))\cap {{C^{\alpha_0}}}([0,T],C^{1+\alpha_0}(\mathbb{R}^d))$, for some $\alpha_0\in (0,1)$ is a solution of \eqref{sqg4}-\eqref{u4}.
Then
  \begin{align}\nonumber
&	\sup_{t\in [0,T]}||\theta(t)||_{\dot C^\alpha}^{\frac{1}{\alpha}-\frac{\alpha}{2-\alpha}} +{{||\theta||_{L^\infty_{T}(L^\infty)}^{-\frac{1}{\alpha}}}}\int_{0}^{T} ||\theta||_{\dot C^\alpha}^{\frac{2}{\alpha}-\frac{\alpha}{2-\alpha}} \lesssim ||\theta_0||_{\dot C^\alpha}^{\frac{1}{\alpha}-\frac{\alpha}{2-\alpha}}+M\int_{0}^{T}||\theta||_{\dot C^{1+\frac{\alpha}{2}}}\\&+M\left(||\tilde \theta||_{L^1_T( L^\infty\cap L^1)}+\int_{0}^{T}||\tilde \theta||_{\dot C^{\alpha}}^{\frac{1}{\alpha}-\frac{\alpha}{2-\alpha}}+\int_{0}^{T}||\tilde \theta||_{\dot C^{\frac{3}{4}}} ||\tilde \theta||_{\dot C^{\frac{1}{4}}}+\int_{0}^{T}||u_1||_{\dot C^{\frac{1+2\alpha}{3}}}^6+\int_{0}^{T}||f||_{\dot C^{\alpha}}^2\right)\label{z28}
\end{align}
holds for $\alpha$ such that  $\alpha(1+||\tilde{\theta}||_{L^\infty_{T}(L^\infty)})$ is small enough. Above, the constant $M$ is given by 
 $$M=(||\tilde \theta||_{L^\infty_T(L^1\cap L^\infty)}+1)^{\frac{2}{\alpha}}.$$
\end{lemma}
\begin{proof}
We take $0<\alpha\leq \frac{\alpha_0}{16}$ and consider the equation obeyed by $D^{\alpha}_h \theta=\frac{\delta_h\theta}{|h|^\alpha}$ with $\delta_h f  = f(x+h)-f(x)$.
We apply first the finite difference $\delta_h$: 
\be
(\pa_t +u \cdot \na_x +  (\delta_h u)\cdot\na_h +  L_1^{\fr12})\delta_h\theta =  L_1^{\fr12}\delta_h\theta-\delta_h (L_1^{\fr12}\theta) + \delta_h f.
\la{delhbarteq}
\ee
Then we obtain the equation obeyed by  
$
q(t,x,h) = D_h^\alpha \theta:$
\begin{align}\label{z13}
(\pa_t + u\cdot \na_x +  (\delta_h u)\cdot\na_h +  L_1^{\fr12}) q = L_1^{\fr12} D_h^\alpha \theta-D_h^\alpha  (L_1^{\fr12}\theta)+D_h^\alpha f - \alpha \left(\delta_h u\cdot {\fr{h}{|h|^2}}\right) q.
\end{align}
We multiply \eqref{z13} by $q$ and use the quadratic difference lower bound \eqref{z1}{{,\eqref{z16}}}, where $D(v)$ is defined in \eqref{Ddefi}, and \eqref{z89}, to obtain 
\begin{align}\nonumber
&	\frac{1}{2}(\pa_t + u\cdot \na_x +  (\delta_h u)\cdot\na_h + L_1^{\fr12})( q^2) +c||\theta||_{L^\infty}^{-1}|h|^{-1+\alpha} |q|^{3}+c |h|^{-2\alpha}D(\delta_h(\eta\tilde \theta))\\&\leq   \frac{\alpha}{c} \left(|h|^{-1-2\alpha} |\delta_h \na L_2^{-\fr12} (\tilde{\theta})|(\delta_h(\eta\tilde \theta))^2+ ||\na L_2^{-\fr12} (\tilde{\theta})||_{L^\infty}q^2\right) + |q||L_1^{\fr12} D_h^\alpha \theta-D_h^\alpha  (L_1^{\fr12}\theta)|\nonumber\\&\quad\quad\quad\quad\quad\quad\quad\quad+q|D_h^\alpha f |+\alpha |h|^{-1} ||\delta_hu_1||_{L^\infty}q^2,\label{z13b}
\end{align}
where the constant $c>0$ does not depend on $\alpha$.\\
Note that  
\begin{equation}
\sup_{h,x\in \mathbb{R}^d} q(t,x,h)^2=\sup_{h,x\in \mathbb{R}^d} \left[ \varphi_1(x)  \varphi_2(h)q(t,x,h)^2\right]
\end{equation}
where  $\varphi_1, \varphi_2$ {{are}} cut-off functions such that $\varphi_1=1$ in  $B(0,\lambda_0)$ and  $\varphi_2=1$ in $B(0,\lambda_0)\backslash B(0,1/\lambda_0)$; $\varphi_1=\varphi_2=0$ in  {{$\mathbb{R}^d\backslash B(0,2\lambda_0)$, $\varphi_2=0$  in  $ B(0,1/(2\lambda_0))$}} for some $\lambda_0\geq 1$ large enough.\\
Thus, we apply  Lemma \ref{consttarvic} to deduce that there exists  $(x_t,h_t)\in \mathbb{R}^d\times  \mathbb{R}^d$ such that 
\begin{equation}
q(t):=q(t, x_t,h_t)=\sup_{h,x}q(t,x,h)\in W^{1,\infty}([0,T]).
\end{equation} 
Moreover,  $q$ is an absolutely continuous function on $[0,T]$ and
\begin{equation}\label{z79}
\frac{d}{dt}q(t)=(\partial_tq)(t, x_t,h_t) ~~\text{a.e. in }~~t\in [0,T].
\end{equation}
We take great advantage of the fact that
\begin{align}\label{z26}
0<|h_t|\leq \min\left\{ 4r_0, \left(\frac{2||\theta||_{L^\infty}}{q(t)}\right)^{\frac{1}{\alpha}}\right\},~~~q(t)\geq c r_0^{-\alpha} ||\theta(t)||_{L^\infty}.
\end{align}
In \eqref{z13b}, we take $(x,h)=(x_t,h_t)$ and use \eqref{z79} to obtain that 
\begin{align}\nonumber
	q\pa_t q +\frac{c}{2}||\theta||_{L^\infty}^{-1}|h|^{-1+\alpha} |q|^{3}&+c |h|^{-2\alpha}D(\delta_h(\eta\tilde \theta))\\&\nonumber\leq  \frac{\alpha}{c}|h|^{-1-2\alpha} |\delta_h \na L_2^{-\fr12} (\tilde{\theta})|(\delta_h(\eta\tilde \theta))^2+C|h|^{2(1-\alpha)}||\na L_2^{-\fr12} (\tilde{\theta})||_{L^\infty}^{3}||\theta||_{L^\infty}^2 \\&\quad+ |q||L_1^{\fr12} D_h^\alpha \theta-D_h^\alpha  (L_1^{\fr12}\theta)|+q||f||_{\dot C^{\alpha}}+{{\alpha^3 C ||\theta||_{L^\infty}^2||u_1||_{\dot C^{\frac{1+2\alpha}{3}}}^3}}.\label{z13c}
\end{align}
Here we have used Holder's inequalities
\begin{equation}
	\frac{\alpha}{c}||\na L_2^{-\fr12} (\tilde{\theta})||_{L^\infty}q^2\leq {{\frac{c}{4}}}||\theta||_{L^\infty}^{-1}|h|^{-1+\alpha} |q|^{3}+C|h|^{2(1-\alpha)}||\na L_2^{-\fr12} (\tilde{\theta})||_{L^\infty}^{3} {{||\theta||_{L^\infty}^2,}}
\end{equation}
and 
{{\begin{align}
	\alpha	 |h|^{-1} ||\delta_hu_1||_{L^\infty}q^2\leq \frac{c}{4}||\theta||_{L^\infty}^{-1}|h|^{-1+\alpha} |q|^{3} {{+ (\frac{4}{c})^2\alpha^3 ||\theta||_{L^\infty}^2||u_1||_{\dot C^{\frac{1+2\alpha}{3}}}^3.}}
\end{align}}}
{{We invoke Corollary \ref{threquarts} with the estimate \eqref{z25} to obtain
\begin{align}
	&|h|^{-1-2\alpha} |\delta_h \na L_2^{-\fr12} (\tilde{\theta})| (\delta_h(\eta\tilde \theta))^2\leq C ||\tilde \theta||_{L^\infty}	|h|^{-2\alpha}D(\delta_h (\eta\tilde \theta))+C||\tilde \theta||_{L^\infty}||\tilde \theta||_{\dot C^{\alpha}}^2\nonumber\\&\quad\quad\quad+C|h|^{2-2\alpha}||\tilde \theta||_{L^\infty} ||\tilde \theta||_{\dot C^{\frac{3}{4}}} ||\tilde \theta||_{\dot C^{\frac{1}{4}}}+C_\alpha|h|^{1-2\alpha}( ||\tilde \theta||_{L^\infty\cap L^1}+1)||\tilde \theta||_{L^\infty\cap L^1}^2\log(2+||\tilde \theta||_{\dot C^\alpha})\nonumber
	\\&\nonumber\quad\quad\leq C ||\tilde \theta||_{L^\infty}	|h|^{-2\alpha}D(\delta_h (\eta\tilde \theta))+C|h|^{2-2\alpha}||\tilde \theta||_{L^\infty} ||\tilde \theta||_{\dot C^{\frac{3}{4}}} ||\tilde \theta||_{\dot C^{\frac{1}{4}}}\\&\quad\quad\quad+C_\alpha(1+r_0)( ||\tilde \theta||_{L^\infty\cap L^1}+1)^2||\tilde \theta||_{L^\infty\cap L^1}(1+||\tilde \theta||_{\dot C^{\alpha}}^2),\label{z93}
\end{align}
where $C$ does not depend on $\alpha$, but $C_\alpha$ depends on $\alpha$.\\ 
We emphasize the key point that in the above estimate the constant $C$ in front of $||\tilde \theta||_{L^\infty}	|h|^{-2\alpha}D(\delta_h (\eta\tilde \theta))$ does not depend on $\alpha$. Thus, crucially,  if $\alpha$ is small enough, the corresponding term, which comes in \eqref{z13c} multiplied by $\fr{\alpha}{c}$, and is the most dangerous term, is less than a fraction of the term provided by the dissipation $D(q)$.\\
On the other hand, using  \eqref{z2}, \eqref{z27b} and interpolation inequality, we also have
\begin{align}
|L_1^{\fr12} D_h^\alpha \theta-D_h^\alpha  (L_1^{\fr12}\theta)|\lesssim C_\alpha |h|^{1-\alpha}q^{\frac{\alpha}{2-\alpha}} ||\theta||_{\dot C^{1+\frac{\alpha}{2}}}^{\frac{2(1-\alpha)}{2-\alpha}}+C_\alpha|h|^{1-\alpha}||\theta||_{L^1},
\end{align}
and 
\begin{align}\nonumber
	||\na L_2^{-\fr12} (\tilde{\theta})||_{L^\infty}^3&\leq C_\alpha ||\na L_2^{-\fr12} (\tilde{\theta})||_{C^{\frac{\alpha}{3}}}^3\leq C_\alpha \left(||\tilde{\theta}||_{C^{\frac{\alpha}{3}}}+||\tilde{\theta}||_{L^1}\right)^3 \\&\leq\nonumber C_\alpha||\tilde \theta||_{L^\infty\cap L^1}^{2}(||\tilde \theta||_{ L^1}+||\tilde \theta||_{\dot C^{\alpha}})\\&\leq C_\alpha( ||\tilde \theta||_{L^\infty\cap L^1}+1)^2||\tilde \theta||_{L^\infty\cap L^1}(1+||\tilde \theta||_{\dot C^{\alpha}}).
\end{align}}}
Let $\alpha\in (0,1)$ be small enough such that 
\begin{align}
C	\alpha ||\tilde{\theta}||_{L^\infty_{t,x}}+\alpha
\leq {{\frac{\min\{c,1\}}{8}}}.\end{align}
We obtain  \begin{align}\nonumber
	q\pa_t q +\frac{c}{2}||\theta||_{L^\infty}^{-1}|h|^{-1+\alpha} |q|^{3}&\lesssim ( ||\tilde \theta||_{L^\infty\cap L^1}+1)^4||\tilde \theta||_{L^\infty\cap L^1}(1+||\tilde \theta||_{\dot C^{\alpha}}^2)\\&\quad\nonumber+|h|^{2-2\alpha}||\tilde \theta||_{L^\infty} ||\tilde \theta||_{\dot C^{\frac{3}{4}}} ||\tilde \theta||_{\dot C^{\frac{1}{4}}}+|h|^{1-\alpha}q^{\frac{2}{2-\alpha}}||\theta||_{ C^{1+\frac{\alpha}{2}}}^{\frac{2(1-\alpha)}{2-\alpha}}\\&\quad+|h|^{1-\alpha}q||\theta||_{L^1\cap L^\infty}+q||f||_{\dot C^{\alpha}}+||\theta||_{L^\infty}^2||u_1||_{\dot C^{\frac{1+2\alpha}{3}}}^3.
\end{align}
Using  $|h_t|\leq  \left(\frac{2||\theta||_{L^\infty}}{q(t)}\right)^{\frac{1}{\alpha}}$, we deduce 
\begin{align}\nonumber
	q\pa_t q +c'||\theta||_{L^\infty}^{-\frac{1}{\alpha}} |q|^{\frac{1}{\alpha}+2}&\lesssim ( ||\tilde \theta||_{L^\infty\cap L^1}+1)^4||\tilde \theta||_{L^\infty\cap L^1}(1+||\tilde \theta||_{\dot C^{\alpha}}^2)\\&\nonumber+  ||\theta||_{L^\infty}^{\frac{2-2\alpha}{\alpha}}q^{-\frac{2}{\alpha}+2}||\tilde \theta||_{L^\infty} ||\tilde \theta||_{\dot C^{\frac{3}{4}}} ||\tilde \theta||_{\dot C^{\frac{1}{4}}} +||\theta||_{L^\infty}^{\frac{1-\alpha}{\alpha}}q^{-\frac{1}{\alpha}+\frac{4-\alpha}{2-\alpha}}||\theta||_{\dot C^{1+\frac{\alpha}{2}}}^{\frac{2(1-\alpha)}{2-\alpha}}\\&+||\tilde\theta||_{L^\infty}^{\frac{1-\alpha}{\alpha}}q^{-\frac{1}{\alpha}+2}||\tilde\theta||_{L^1\cap L^\infty}+q||f||_{\dot C^{\alpha}}+||\tilde\theta||_{L^\infty}^2||u_1||_{\dot C^{\frac{1+2\alpha}{3}}}^3.
\end{align}
We divide both sides by $q^{-\frac{1}{\alpha}+\frac{4-\alpha}{2-\alpha}}$
\begin{align}\nonumber
\pa_t (q^{\frac{1}{\alpha}-\frac{\alpha}{2-\alpha}}) &+c''||\theta||_{L^\infty}^{-\frac{1}{\alpha}} q^{\frac{2}{\alpha}-\frac{\alpha}{2-\alpha}}\lesssim q^{\frac{1}{\alpha}-\frac{4-\alpha}{2-\alpha}} ( ||\tilde \theta||_{L^\infty\cap L^1}+1)^4||\tilde \theta||_{L^\infty\cap L^1}(1+||\tilde \theta||_{\dot C^{\alpha}}^2)\\&\quad\quad\nonumber +  ||\theta||_{L^\infty}^{\frac{2-2\alpha}{\alpha}}q^{-\frac{1}{\alpha}-\frac{\alpha}{2-\alpha}}||\tilde \theta||_{L^\infty} ||\tilde \theta||_{\dot C^{\frac{3}{4}}} ||\tilde \theta||_{\dot C^{\frac{1}{4}}} +||\theta||_{L^\infty}^{\frac{1-\alpha}{\alpha}}||\theta||_{\dot C^{1+\frac{\alpha}{2}}}^{\frac{2(1-\alpha)}{2-\alpha}}\\&\quad\quad+||\tilde\theta||_{L^\infty}^{\frac{1-\alpha}{\alpha}}q^{-\frac{\alpha}{2-\alpha}}||\tilde\theta||_{L^1\cap L^\infty}+q^{\frac{1}{\alpha}-\frac{2}{2-\alpha}}||f||_{\dot C^{\alpha}}+q^{\frac{1}{\alpha}-\frac{4-\alpha}{2-\alpha}}||\theta||_{L^\infty}^2||u_1||_{\dot C^{\frac{1+2\alpha}{3}}}^3.\label{z83}
\end{align}
Note that the positive constant $c''$ now depends on $\alpha$.\\ 
Using $q(t)\geq c r_0^{-\alpha} ||\theta(t)||_{L^\infty}$; $q\lesssim ||\tilde \theta||_{\dot C^{\alpha}}$ and $||\theta||_{\dot C^{1+\frac{\alpha}{2}}}^{\frac{2(1-\alpha)}{2-\alpha}}\lesssim_{r_0} ||\theta||_{L^\infty}^{-\frac{\alpha}{2-\alpha}} ||\theta||_{\dot C^{1+\frac{\alpha}{2}}}$, we have 
 {{\begin{align}\nonumber
	q^{\frac{1}{\alpha}-\frac{4-\alpha}{2-\alpha}} ( ||\tilde \theta||_{L^\infty\cap L^1}+1)^4||\tilde \theta||_{L^\infty\cap L^1}(1+||\tilde \theta||_{\dot C^{\alpha}}^2)&\lesssim ( ||\tilde \theta||_{L^\infty\cap L^1}+1)^4||\tilde \theta||_{L^\infty\cap L^1}(1+||\tilde \theta||_{\dot C^{\alpha}}^2)||\tilde \theta||_{\dot C^{\alpha}}^{\frac{1}{\alpha}-\frac{4-\alpha}{2-\alpha}}\\&\lesssim M  ||\tilde \theta||_{\dot C^{\alpha}}^{\frac{1}{\alpha}-\frac{\alpha}{2-\alpha}}; 
\end{align}
\begin{align}\nonumber
	||\theta||_{L^\infty}^{\frac{2-2\alpha}{\alpha}}q^{-\frac{1}{\alpha}-\frac{\alpha}{2-\alpha}}||\tilde \theta||_{L^\infty} ||\tilde \theta||_{\dot C^{\frac{3}{4}}} ||\tilde \theta||_{\dot C^{\frac{1}{4}}} &+||\theta||_{L^\infty}^{\frac{1-\alpha}{\alpha}}||\theta||_{\dot C^{1+\frac{\alpha}{2}}}^{\frac{2(1-\alpha)}{2-\alpha}}+||\tilde\theta||_{L^\infty}^{\frac{1-\alpha}{\alpha}}q^{-\frac{\alpha}{2-\alpha}}||\tilde\theta||_{L^1\cap L^\infty}\\&\lesssim M\left(||\tilde \theta||_{\dot C^{\frac{3}{4}}} ||\tilde \theta||_{\dot C^{\frac{1}{4}}}+||\theta||_{\dot C^{1+\frac{\alpha}{2}}}+||\tilde\theta||_{L^1\cap L^\infty}\right);
\end{align}
\begin{align}
	&q^{\frac{1}{\alpha}-\frac{2}{2-\alpha}}||f||_{\dot C^{\alpha}}\leq \varepsilon ||\theta||_{L^\infty}^{-\frac{1}{\alpha}} q^{\frac{2}{\alpha}-\frac{\alpha}{2-\alpha}}+C_\varepsilon M ||f||_{\dot C^{\alpha}}^2;\\&
	q^{\frac{1}{\alpha}-\frac{4-\alpha}{2-\alpha}}||\theta||_{L^\infty}^2||u_1||_{\dot C^{\frac{1+2\alpha}{3}}}^3\leq \varepsilon||\theta||_{L^\infty}^{-\frac{1}{\alpha}} q^{\frac{2}{\alpha}-\frac{\alpha}{2-\alpha}}+C_\varepsilon M ||u_1||_{\dot C^{\frac{1+2\alpha}{3}}}^6,
\end{align}
for any $\varepsilon>0.$
Combining these with \eqref{z83}, we deduce  }}
\begin{align}\nonumber
&	\pa_t (q^{\frac{1}{\alpha}-\frac{\alpha}{2-\alpha}}) +c''||\theta||_{L^\infty}^{-\frac{1}{\alpha}} q^{\frac{2}{\alpha}-\frac{\alpha}{2-\alpha}}\\&\lesssim M\left(||\tilde\theta||_{L^1\cap L^\infty}+ ||\tilde \theta||_{\dot C^{\alpha}}^{\frac{1}{\alpha}-\frac{\alpha}{2-\alpha}} +||\tilde \theta||_{\dot C^{\frac{3}{4}}} ||\tilde \theta||_{\dot C^{\frac{1}{4}}}+||u_1||_{\dot C^{\frac{1+2\alpha}{3}}}^6+||f||_{\dot C^{\alpha}}^2+||\theta||_{\dot C^{1+\frac{\alpha}{2}}}\right).
\end{align}
This implies 	\eqref{z28} by using \eqref{z80}. 
\end{proof}

\begin{coro}\la{corol} Assume that $\theta\in C^{1+\alpha_0}([0,T],C^{\alpha_0}(\mathbb{R}^d))\cap {{C^{\alpha_0}}}([0,T],C^{1+\alpha_0}(\mathbb{R}^d))$, $\alpha_0\in (0,1)$ is a solution of \eqref{sqg4}-\eqref{u4}. Then, for {{$\alpha\in (0,\alpha_0)$}} such that  $\alpha(1+||\tilde{\theta}||_{L^\infty_{T}(L^\infty)})$ is small enough, we have 
{{\begin{align}\nonumber
&	\sup_{t\in [0,T]}||\theta(t)||_{\dot C^\alpha}^{\frac{1}{\alpha}-\frac{\alpha}{2-\alpha}} +\int_{0}^{T} ||\theta||_{\dot C^\alpha}^{\frac{2}{\alpha}-\frac{\alpha}{2-\alpha}} +\int_{0}^{T}||\theta||_{\dot C^{1+\frac{\alpha}{2}}}\lesssim \tilde M||\theta_0||_{\dot C^\alpha}^{\frac{1}{\alpha}-\frac{\alpha}{2-\alpha}}\\&\nonumber+\tilde M\left(\int_{0}^{T}||\tilde \theta||_{\dot C^{\alpha}}^{\frac{1}{\alpha}-\frac{\alpha}{2-\alpha}}+\int_{0}^{T}||\tilde \theta||_{\dot C^{\frac{3}{4}}} ||\tilde \theta||_{\dot C^{\frac{1}{4}}}+\int_{0}^{T}||u_1||_{\dot C^{\frac{1+2\alpha}{3}}}^6+\int_{0}^{T}||f||_{\dot C^{\alpha}}^2\right)\\&+\tilde MM_2(||\theta_0||_{ C^{\frac{5\alpha}{8}}}+||\tilde \theta||_{L^1_T( L^\infty\cap L^1)}+||f||_{L^1_T( C^{\frac{5\alpha}{8}})})+\tilde MM_2^{\frac{6}{\alpha}}\int_{0}^{T} (1+		||\tilde\theta||_{C^{\alpha}}^{\frac{5}{8}}	+	||u_1||_{ C^{\frac{5\alpha}{8}}})^{\frac{5}{2\alpha}}||\tilde \theta||_{\dot C^{\frac{\alpha}{2}}},\label{z34}
\end{align}}}
where 
\begin{align}
&	\tilde M=(||\tilde \theta||_{L^\infty_T(L^1\cap L^\infty)}+T+1)^{\frac{3}{\alpha}},\\&
{{M_2=(1+|| \tilde\theta||_{L^\infty_{T}(L^\infty\cap L^1)})\log(2+|| \tilde\theta||_{L^\infty_{T}( \dot C^{\alpha})})+		||u_1||_{L^\infty_{T}(L^\infty)}.}}
\end{align}
\end{coro}
\begin{proof} We apply  \eqref{z31b} to $(v,b)=(\theta, u)$ with  $\alpha_1=\frac{\alpha}{2},\alpha_2=\frac{5\alpha}{8}$,
\begin{align}\nonumber
||\theta||_{L^1_T(\dot C^{1+\frac{\alpha}{2}})}&\lesssim
M_1(T+1)\left(||\theta_0||_{ C^{\frac{5\alpha}{8}}}+||\theta||_{L^1_T(L^\infty)}+||f||_{L^1_T( C^{\frac{5\alpha}{8}})}\right)\\&+M_1^{\frac{6}{\alpha}}(1+T)^{\frac{5}{4}}\int_{0}^{T} (1+	||u||_{ C^{\frac{5\alpha}{8}}})^{\frac{5}{2\alpha}}||\theta||_{\dot C^{\frac{\alpha}{2}}}ds,
\end{align}
with 
$M_1=1+|| u||_{L^\infty_T(L^\infty)}$.\\
In view of \eqref{z27b}, \eqref{z27} and \eqref{z89}, we estimate 
\begin{align}\nonumber
&||u(s)||_{ C^{\frac{5\alpha}{8}}}\lesssim ||\tilde\theta(s)||_{C^{\frac{5\alpha}{8}}}+||\tilde\theta(s)||_{L^1}	+	||u_1(s)||_{ C^{\frac{5\alpha}{8}}}\\&\quad\quad\quad\quad\quad\quad\lesssim  ||\tilde\theta(s)||_{L^\infty}^{\frac{3}{8}}	||\tilde\theta(s)||_{C^{\alpha}}^{\frac{5}{8}}+||\tilde\theta(s)||_{L^1}	+	||u_1(s)||_{ C^{\frac{5\alpha}{8}}},\\&
M_1\lesssim \left(1+	|| \tilde\theta||_{L^\infty_{T}(L^\infty\cap L^1)}\right)\log(2+|| \tilde\theta||_{L^\infty_{T}( \dot C^{\alpha})})+		||u_1||_{L^\infty_{T}(L^\infty)}:= M_2,
\end{align}
we get 
\begin{align}\nonumber
&||\theta||_{L^1_T(\dot C^{1+\frac{\alpha}{2}})}\lesssim
M_2(T+1)\left(||\theta_0||_{ C^{\frac{5\alpha}{8}}}+||\theta||_{L^1_T(L^\infty)}+||f||_{L^1_T( C^{\frac{5\alpha}{8}})}\right)\\&\quad\quad\quad\quad+M_2^{\frac{6}{\alpha}}(T+1)^{\frac{5}{4}}\int_{0}^{T} (1+	||\tilde\theta||_{L^\infty}^{\frac{3}{8}}	||\tilde\theta||_{C^{\alpha}}^{\frac{5}{8}}+||\tilde\theta||_{L^1}		+	||u_1||_{ C^{\frac{5\alpha}{8}}})^{\frac{5}{2\alpha}}||\tilde \theta||_{\dot C^{\frac{\alpha}{2}}}.
\end{align}
Combining this with \eqref{z28} we obtain \eqref{z34}.
\end{proof}

\section{Holder regularity }\la{hore}
In this section, we prove 
\begin{thm} \label{holderR}  Let $\theta\in C([0,T],C^{\alpha_0}(\Omega))\cap  C^{1+\alpha_0}_{loc}((0,T],C^{\alpha_0}(\Omega))\cap C^{\alpha_0}_{loc}((0,T],C^{1+\alpha_0}(\Omega))$, $\alpha_0\in (0,1)$  be a solution to \eqref{sqg}, \eqref{u}.  Let $\alpha>0$ be such that $\alpha (\|\theta_0\|_{L^{\infty}} +1)$ is small enough and $\alpha<\fr{\alpha_0}{2}$.  Then
\begin{align}\label{z81}
\sup_{t\in [0,T]}||\theta||_{\dot C^\alpha(\Omega)}^{\frac{1}{\alpha}-\frac{\alpha}{2-\alpha}} +\int_{0}^{T} ||\theta(s)||_{\dot C^\alpha(\Omega)}^{\frac{2}{\alpha}-\frac{\alpha}{2-\alpha}} ds+	\int_{0}^{T}|| \theta(s)||_{ \dot C^{1+\frac{\alpha}{2}}(\Omega)}ds\lesssim (||\theta_0||_{ C^\alpha(\Omega)}+\text{diam}(\Omega)+T+1)^{\frac{36}{\alpha^2}}
\end{align}
holds.
\end{thm}

\begin{proof}  
We know from \eqref{sqg} that
{{
\begin{equation}\label{z59}
\sup_{t>0}||\theta(t)||_{L^\infty(\Omega)}\leq 	||\theta_0||_{L^\infty(\Omega)}
\end{equation}}}
holds. \\
{{First, we prove \eqref{z81} with assumption $\theta\in  C^{1+\alpha_0}([0,T],C^{\alpha_0}(\Omega))\cap C^{\alpha_0}([0,T],C^{1+\alpha_0}(\Omega))$.}}
We apply \eqref{z34} with {{$(f,u_1)=(f_{in},u_{re,in})$}} to the system \eqref{z35a1}-\eqref{z35b1} and use \eqref{esu_rein}, \eqref{foin} and \eqref{z59} to obtain that

\begin{align}\nonumber
\sup_{t\in [0,T]}||\theta||_{\dot C^\alpha(B_i^0)}^{\frac{1}{\alpha}-\frac{\alpha}{2-\alpha}} &+\int_{0}^{T} ||\theta||_{\dot C^\alpha (B_i^0)}^{\frac{2}{\alpha}-\frac{\alpha}{2-\alpha}} +	\int_{0}^{T}|| \theta||_{ \dot C^{1+\frac{\alpha}{2}}(B_i^0)}\lesssim	 \tilde M_1+\\&\nonumber\quad+\tilde M_1\left(\int_{0}^{T}|| \theta||_{ C^{\alpha}}^{\frac{1}{\alpha}-\frac{\alpha}{2-\alpha}}+\int_{0}^{T}|| \theta||_{ C^{\frac{3}{4}}} ||\theta||_{ C^{\frac{1}{4}}}+\int_{0}^{T} ||\theta||_{ C^\alpha}^{4}\right)\\&\quad+{{ \tilde M_1 \log(2+|| \theta||_{L^\infty_{T}(C^{\alpha})})^{\frac{6}{\alpha}} \int_{0}^{T} \left(1+ ||\theta||_{ C^{\alpha}}\right)^{\frac{25}{16\alpha}+\frac{1}{2}}}}:=M_4,\label{z83a}
\end{align}
for $N_1+1\leq i\leq N$, where 
\begin{equation}
\tilde M_1=({{ ||\theta_0||_{ C^\alpha(\Omega)}}}+\text{diam}(\Omega)+T+1)^{\frac{12}{\alpha}}.
\end{equation}
{{ In this inequality,  we estimated
\begin{equation}
		\tilde M||\theta_0||_{\dot C^\alpha}^{\frac{1}{\alpha}-\frac{\alpha}{2-\alpha}}+\tilde M\int_{0}^{T}||u_1||_{\dot C^{\frac{1+2\alpha}{3}}}^6+\tilde MM_2\left(||\theta_0||_{ C^{\frac{5\alpha}{8}}}+||\tilde \theta||_{L^1_T( L^\infty\cap L^1)}\right)\lesssim \tilde M_1\log(2+|| \theta||_{L^\infty_{T}(  C^{\alpha})});
\end{equation}
\begin{equation}
		\tilde M\int_{0}^{T}||f||_{\dot C^{\alpha}}^2+\tilde MM_2||f||_{L^1_T( C^{\frac{5\alpha}{8}})}\lesssim \tilde M_1+\tilde M_1\int_{0}^{T} ||\theta||_{ C^\alpha}^{4};
\end{equation}
and
\begin{equation}
\tilde MM_2^{\frac{6}{\alpha}}\int_{0}^{T} (1+		||\tilde\theta||_{C^{\alpha}}^{\frac{5}{8}}	+	||u_1||_{ C^{\frac{5\alpha}{8}}})^{\frac{5}{2\alpha}}||\tilde \theta||_{\dot C^{\frac{\alpha}{2}}}\lesssim \tilde M_1 \log(2+|| \theta||_{L^\infty_{T}(  C^{\alpha})})^{\frac{6}{\alpha}} \int_{0}^{T} \left(1+ ||\theta||_{ C^{\alpha}}\right)^{\frac{25}{16\alpha}+\frac{1}{2}}.
\end{equation}}}
Similarly, we apply \eqref{z34} to the equation \eqref{widesqg}  {{ and use \eqref{esu_re}, \eqref{fo}, \eqref{z87b}}} and \eqref{z59}  to deduce that 
\begin{align}	\sup_{t\in [0,T]}||\theta||_{\dot C^\alpha(B_i^0\cap\Omega)}^{\frac{1}{\alpha}-\frac{\alpha}{2-\alpha}} +\int_{0}^{T} ||\theta||_{\dot C^\alpha (B_i^0\cap\Omega)}^{\frac{2}{\alpha}-\frac{\alpha}{2-\alpha}} ds+	\int_{0}^{T}|| \theta(t)||_{ \dot C^{1+\frac{\alpha}{2}}(B_i^0\cap\Omega)}\lesssim M_4\label{z83b}
\end{align}
for $1\leq i\leq N_1$.\\ As we have a cover of $\Omega$ by $\cup_{i=1}^NB_{i}^0$, then we obtain from \eqref{z83a} and  \eqref{z83b} that  
\begin{align}
\sup_{t\in [0,T]}||\theta||_{\dot C^\alpha(\Omega)}^{\frac{1}{\alpha}-\frac{\alpha}{2-\alpha}} +\int_{0}^{T} ||\theta||_{\dot C^\alpha(\Omega)}^{\frac{2}{\alpha}-\frac{\alpha}{2-\alpha}} +	\int_{0}^{T}|| \theta||_{ \dot C^{1+\frac{\alpha}{2}}(\Omega)}\lesssim M_4.
\end{align}
By interpolation, Holder's inequalities and \eqref{z94b},  \eqref{z94c} , we obtain
\begin{equation}
		\tilde M_1\left(\int_{0}^{T}|| \theta||_{ C^{\alpha}}^{\frac{1}{\alpha}-\frac{\alpha}{2-\alpha}}+\int_{0}^{T} ||\theta||_{ C^\alpha}^{4}\right)\lesssim  \varepsilon\int_{0}^{T} ||\theta||_{\dot C^\alpha}^{\frac{2}{\alpha}-\frac{\alpha}{2-\alpha}} + C_\varepsilon	\tilde M_1^3;
\end{equation}
\begin{equation}
\tilde M_1\int_{0}^{T}|| \theta||_{ C^{\frac{3}{4}}} ||\theta||_{ C^{\frac{1}{4}}}\overset{\eqref{z59}}\lesssim \tilde M_1\int_{0}^{T}|| \theta_0||_{ L^\infty}^{\frac{2(1+\alpha)}{2+\alpha}} ||\theta||_{ C^{1+\frac{\alpha}{2}}}^{\frac{2}{2+\alpha}}\lesssim \varepsilon\int_{0}^{T}|| \theta||_{ \dot C^{1+\frac{\alpha}{2}}}+C_\varepsilon \tilde M_1^{\frac{3}{\alpha}};	
\end{equation}
and 
\begin{equation}
\tilde M_1 \log(2+|| \theta||_{L^\infty_{T}(  C^{\alpha})})^{\frac{6}{\alpha}} \int_{0}^{T} \left(1+ ||\theta||_{ C^{\alpha}}\right)^{\frac{25}{16\alpha}+\frac{1}{2}}\lesssim 	\varepsilon(\sup_{[0,T]}||\theta||_{\dot C^\alpha}^{\frac{1}{\alpha}-\frac{\alpha}{2-\alpha}} +\int_{0}^{T} ||\theta||_{\dot C^\alpha}^{\frac{2}{\alpha}-\frac{\alpha}{2-\alpha}} )+C_\varepsilon\tilde M_1^3.
\end{equation}
Thus, 
\begin{equation}
\sup_{[0,T]}||\theta||_{\dot C^\alpha(\Omega)}^{\frac{1}{\alpha}-\frac{\alpha}{2-\alpha}} +\int_{0}^{T} ||\theta||_{\dot C^\alpha(\Omega)}^{\frac{2}{\alpha}-\frac{\alpha}{2-\alpha}} ds+	\int_{0}^{T}|| \theta||_{ \dot C^{1+\frac{\alpha}{2}}(\Omega)}\lesssim \tilde M_1^{\frac{3}{\alpha}}.
\end{equation}
This implies \eqref{z81}  with assumption $\theta\in  C^{1+\alpha_0}([0,T],C^{\alpha_0}(\Omega))\cap C^{\alpha_0}([0,T],C^{1+\alpha_0}(\Omega))$. Now, we prove  \eqref{z81} without this assumption. \\
Indeed, because  $\theta\in  C^{1+\alpha_0}([\epsilon,T],C^{\alpha_0}(\Omega))\cap C^{\alpha_0}([\epsilon,T],C^{1+\alpha_0}(\Omega))$, for $\epsilon>0$, we apply  \eqref{z81} on $\epsilon+[0,T-\epsilon]=[\epsilon,T]$ and deduce that 
\begin{align}\label{z81b}
	\sup_{t\in [\epsilon,T]}||\theta||_{\dot C^\alpha(\Omega)}^{\frac{1}{\alpha}-\frac{\alpha}{2-\alpha}} +\int_{\epsilon}^{T} ||\theta(s)||_{\dot C^\alpha(\Omega)}^{\frac{2}{\alpha}-\frac{\alpha}{2-\alpha}} ds+	\int_{\epsilon}^{T}|| \theta(s)||_{ \dot C^{1+\frac{\alpha}{2}}(\Omega)}ds\lesssim (||\theta(\epsilon)||_{ C^\alpha(\Omega)}+\text{diam}(\Omega)+T-\epsilon+1)^{\frac{36}{\alpha^2}}
\end{align}
provided that $\alpha (\|\theta(\epsilon)\|_{L^{\infty}} +1)\overset{\eqref{z59}}\leq \alpha (\|\theta_0\|_{L^{\infty}} +1)$ is small enough and $\alpha<\fr{\alpha_0}{2}$. \\
Using $\theta\in C([0,T],C^{\alpha_0}(\Omega))$ and $\alpha<\alpha_0/2$ and letting $\epsilon\to 0$ to deduce that 
\begin{align}\nonumber
	&\sup_{t\in [0,T]}||\theta||_{\dot C^\alpha(\Omega)}^{\frac{1}{\alpha}-\frac{\alpha}{2-\alpha}} +\int_{0}^{T} ||\theta(s)||_{\dot C^\alpha(\Omega)}^{\frac{2}{\alpha}-\frac{\alpha}{2-\alpha}} ds+	\limsup_{\epsilon\to 0}\int_{\epsilon}^{T}|| \theta(s)||_{ \dot C^{1+\frac{\alpha}{2}}(\Omega)}ds
\\&\nonumber=\limsup_{\epsilon\to 0}	\left(\sup_{t\in [\epsilon,T]}||\theta||_{\dot C^\alpha(\Omega)}^{\frac{1}{\alpha}-\frac{\alpha}{2-\alpha}} +\int_{\epsilon}^{T} ||\theta(s)||_{\dot C^\alpha(\Omega)}^{\frac{2}{\alpha}-\frac{\alpha}{2-\alpha}} ds+	\int_{\epsilon}^{T}|| \theta(s)||_{ \dot C^{1+\frac{\alpha}{2}}(\Omega)}ds\right)\\&\nonumber\quad\quad\lesssim \limsup_{\epsilon\to 0} (||\theta(\epsilon)||_{ C^\alpha(\Omega)}+\text{diam}(\Omega)+T-\epsilon+1)^{\frac{36}{\alpha^2}}\\&\quad\quad \leq (||\theta_0||_{ C^\alpha(\Omega)}+\text{diam}(\Omega)+T+1)^{\frac{36}{\alpha^2}}.
\end{align}
This implies the result by using monotone convergence theorem.
\end{proof}
\begin{lemma} For any $\beta_0,\beta_1,...,\beta_k>0,\beta_1,...,\beta_k\leq \beta_0$, and $\varkappa_0,\varkappa_1,...,\varkappa_k\geq 0$, we have 
	\begin{equation}\label{z94}
		\int_{0}^{T}||\theta||_{C^{\beta_1}(\Omega)}^{\varkappa_1}...||\theta||_{C^{\beta_k}(\Omega)}^{\varkappa_k} \lesssim (1+||\theta_0||_{L^\infty})^{\sum_{j=1}^{k}\varkappa_j} \left(\int_{0}^{T} ||\theta||_{ C^{\beta_0}(\Omega)}^{\varkappa_0} +T+1\right)
	\end{equation}
	provided $\sum_{j=1}^{k}\beta_j\varkappa_j\leq \beta_0 \varkappa_0$. 
\end{lemma}
\begin{proof} Using Holder's inequality, we only need to prove \eqref{z94} with $\sum_{j=1}^{k}\beta_j\varkappa_j= \beta_0 \varkappa_0$. Indeed, in view of \eqref{interdeltagamma}, we have  
	\begin{equation}
		||\theta(t)||_{C^{\beta_j}(\Omega)}^{\varkappa_j}\lesssim  	||\theta(t)||_{L^\infty(\Omega)}^{\varkappa_j-\frac{\beta_j\varkappa_j}{\beta_0}}||\theta(t)||_{C^{\beta_2}(\Omega)}^{\frac{\beta_j\varkappa_j}{\beta_0}}\lesssim ||\theta_0||_{L^\infty(\Omega)}^{\varkappa_j-\frac{\beta_j\varkappa_j}{\beta_0}}||\theta(t)||_{C^{\beta_0}(\Omega)}^{\frac{\beta_j\varkappa_j}{\beta_0}}.
	\end{equation}
Thus, 
	\begin{equation}
	||\theta(t)||_{C^{\beta_1}(\Omega)}^{\varkappa_1}...	||\theta(t)||_{C^{\beta_k}(\Omega)}^{\varkappa_k}\lesssim ||\theta_0||_{L^\infty(\Omega)}^{\sum_{j=1}^{k}\varkappa_j-\varkappa_0}||\theta(t)||_{C^{\beta_0}(\Omega)}^{\varkappa_0}.
\end{equation}
This implies \eqref{z94} with $\sum_{j=1}^{k}\beta_j\varkappa_j= \beta_0 \varkappa_0$.
\end{proof}
\begin{coro} For any $\beta_1,...,\beta_k>0,\beta_1,...,\beta_k\leq1+\frac{\alpha}{2}$, and $\varkappa_1,...,\varkappa_k\geq 0$, we have 
	\begin{equation}\label{z94b}
		\int_{0}^{T}||\theta||_{C^{\beta_1}(\Omega)}^{\varkappa_1}...||\theta||_{C^{\beta_k}(\Omega)}^{\varkappa_k} \lesssim (1+||\theta_0||_{L^\infty})^{\sum_{j=1}^{k}\varkappa_j} \left(\int_{0}^{T} ||\theta||_{ C^{1+\frac{\alpha}{2}}(\Omega)}+T+1\right)
	\end{equation}
	provided $\sum_{j=1}^{k}\beta_j\varkappa_j\leq 1+\frac{\alpha}{2}$. 
\end{coro}
\begin{coro} For any $\beta_1,...,\beta_k>0,\beta_1,...,\beta_k\leq \alpha$, and $\varkappa_1,...,\varkappa_k\geq 0$, we have 
	\begin{equation}\label{z94c}
		\int_{0}^{T}||\theta||_{C^{\beta_1}(\Omega)}^{\varkappa_1}...||\theta||_{C^{\beta_k}(\Omega)}^{\varkappa_k} \lesssim (1+||\theta_0||_{L^\infty})^{\sum_{j=1}^{k}\varkappa_j} \left(\int_{0}^{T}||\theta||_{ C^\alpha(\Omega)}^{\frac{2}{\alpha}-\frac{\alpha}{2-\alpha}} +T+1\right)
	\end{equation}
	provided $\sum_{j=1}^{k}\beta_j\varkappa_j\leq 2-\frac{\alpha^2}{2-\alpha}$. 
\end{coro}
\section{Local Existence with Holder Initial Data and Global Higher Regularity}\la{glcr}

We consider
\begin{equation}\label{linearedrift}
\pa_t v + b\cdot\na v + \l v = f~~\text{in}~~\Omega,~~~v\vert_{\partial \Omega}=0,
\end{equation}
with  $b=\na^{\perp} w \in L^\infty([0,T]\times\Omega)$,   $w=f=0$ on $\partial\Omega$.\\
\begin{prop}\label{eslinear+drift} Let $\alpha_0\in (0,1/10]$. Assume  $v\in L^\infty([0,T],C^{2\alpha_0}(\overline{\Omega}))\cap {{L^\infty_{loc}((0,T]}},C^{1+\alpha_0}(\overline{\Omega}))$ and  $b \in L^\infty([0,T],C^{\alpha_0}(\overline{\Omega}))$. The following inequalities hold for any $T>0$
\begin{align}\nonumber
&	\sup_{t\in [0,T]}	|| v(t)||_{ \dot C^{2\alpha_0}(\Omega)}+t^{1-\alpha_0}	|| v(t)||_{ \dot C^{1+\alpha_0}(\Omega)}
\\&\quad\quad	\lesssim\nonumber\left(1+||b{{ ||}}_{L^\infty_{T}(L^\infty(\Omega))}\right)^4(T+1)^2\left( ||v_0||_{ C^{2\alpha_0}(\Omega)}+\sup_{s\in [0,T]}s^{1-\alpha_0}{{||f(s)||_{ C^{\alpha_0}(\Omega)}}}\right)
\\&\quad\quad\quad+\left(1+||b{{ ||}}_{L^\infty_{T}(L^\infty(\Omega))}\right)^{\frac{6}{\alpha_0}}T^{1-\alpha_0}(1+T)\left(1+	||b||_{L^\infty_{T}( C^{\alpha_0}(\Omega))}\right)^{\frac{2}{\alpha_0}}||v_0||_{L^\infty(\Omega)}:=J_1,\label{z37}
\end{align}
and 
\begin{align}\nonumber
&		\sup_{t\in [0,T]}	t^{1-\alpha_0}	|| v(t)||_{ \dot C^{1+\alpha_0}(\Omega)}	\lesssim (1+||b{{ ||}}_{L^\infty_{T}(L^\infty(\Omega))})^2(1+T) \left(T^{\alpha_0}{{||v_0||_{ C^{3\alpha_0}(\Omega)}}}+\sup_{s\in [0,T]}s^{1-\alpha_0}{{||f(s)||_{ C^{\alpha_0}(\Omega)}}}\right)\\&\quad\quad\quad\quad\quad\quad
+(1+||b{{ ||}}_{L^\infty_{T}(L^\infty(\Omega))})^{\frac{3}{\alpha_0}}T^{1-\alpha_1}\left(1+	||b||_{L^\infty_{T}( C^{\alpha_0}(\Omega))}\right)^{\frac{1}{\alpha_0}}||v||_{L^\infty_{T}( C^{\alpha_0}(\Omega))}:=J_2.\label{z37b}
\end{align}
{{Moreover, we also have 
	\begin{align}\nonumber
	&	\sup_{t\in [0,T]}t^{1-\alpha_0+\gamma}	|| v(t)||_{ \dot C^{1+\alpha_0}(\Omega)}
	\\&	\lesssim\nonumber \left(1+||b ||_{L^\infty_{T}(L^\infty(\Omega))}\right)^4(T+1)^2\left(\sup_{t\in [0,T]} t^\gamma ||v(t)||_{ C^{2\alpha_0}(\Omega)}+\sup_{t\in [0,T]}t^{1-\alpha_0+\gamma}||f(t)||_{ C^{\alpha_0}(\Omega)}\right)
	\\&\quad+\left(1+||b ||_{L^\infty_{T}(L^\infty(\Omega))}\right)^{\frac{6}{\alpha_0}}T^{1-\alpha_0}(1+T)\left(1+	||b||_{L^\infty_{T}( C^{\alpha_0}(\Omega))}\right)^{\frac{2}{\alpha_0}}\sup_{t\in [0,T]} t^\gamma||v(t)||_{L^\infty(\Omega)},\label{z84}
\end{align}
for any $\gamma>0$.}}
\end{prop}
\begin{proof} We have 
\begin{equation}\label{z88}
\sup_{t\ge 0}	||v(t)||_{L^p(\Omega)}\leq 	||v_0||_{L^p(\Omega)}\quad\quad\forall p\in [1,\infty].
\end{equation}
1) \textit{Boundary estimate}: Let $1\leq i\leq N_1$. As \eqref{widesqg}, we multiply \eqref{linearedrift} by $\chi=\chi_i^0$, using  \eqref{nonerro}, \eqref{FLam} and $f\vert_{\partial \Omega}=0$, we arrive at
\begin{align}
&	\pa_t \tilde v + \tilde b\cdot\na \tilde v +L^{\frac{1}{2}}\tilde v = f_2,\\&
\tilde v=\mathcal{F}(\chi_i^0 v),~~~\tilde b=\mathbf{a}\na^{\perp}\mathcal F(\chi_i^1w),
\end{align}
where $\mathbf{a}$ is given in \eqref{Odeltachi} and satisfies \eqref{z87b};
\begin{equation}
f_2=\tilde b\cdot \mathcal F((\na\chi_i^0) v)+\mathcal F(\chi_i^0 f)+{{\mathbf{R}_{\chi_i^0}(v)}}.
\end{equation}
{{Using \eqref{CsR} \eqref{z87b}, and $w=f=0$ on $\partial\Omega$}}, we have 
\begin{align}\label{z84a}
&	||\tilde b||_{C^\beta(\mathbb{R}^2)}\lesssim 	|| b||_{C^\beta(\Omega)}~~~\forall~~\beta\in [0,1),\\&
||f_2||_{C^\beta(\mathbb{R}^2)}\lesssim  ||f||_{C^{\beta}(\Omega)}+(	1+||b||_{L^\infty_{T}(L^\infty)}) ||v||_{C^{\beta}(\Omega)}+||v_0||_{L^\infty(\Omega)} ||b||_{C^{\beta}(\Omega)}~\forall~\beta\in (0,1).\label{z84b}
\end{align}
We apply  \eqref{z57}, \eqref{z58} with $\alpha_1=\alpha_0$  and use \eqref{z84a}, \eqref{z84b} and \eqref{z88} to get that 
\begin{align}\nonumber
&||v||_{L^\infty_{T}(\dot C^{2\alpha_0}(B_i^0\cap \Omega))}+\sup_{s\in [0,T]}s^{1-\alpha_0}	||v(s)||_{\dot C^{1+\alpha_0}(B_i^0\cap \Omega)}\\&\nonumber\lesssim \left(1+||b||_{L^\infty_{T}(L^\infty(\Omega))}\right)^2\left( ||v_0||_{ C^{2\alpha_0}(\Omega)}+\sup_{s\in [0,T]}s^{1-\alpha_0}{{||f(s)||_{ C^{\alpha_0}(\Omega)}}}+T^{1-\alpha_0}||v||_{L^\infty_T( C^{\alpha_0}(\Omega))}\right)\\&
\quad\quad\quad\quad	+\left(1+||b||_{L^\infty_{T}(L^\infty(\Omega))}\right)^{\frac{6}{\alpha_0}}T^{1-\alpha_0}(1+T)\left(1+||b||_{L^\infty_{T}( C^{\alpha_0}(\Omega))}\right)^{\frac{2}{\alpha_0}}||v_0||_{L^\infty(\Omega)}:=J_0,\\&
\sup_{s\in [0,T]}s^{1-\alpha_0}	||v(s)||_{\dot C^{1+\alpha_0}(B_i^0)}\lesssim J_2.
\end{align}
2) \textit{Interior estimate}:
Define $\overline v=\chi_i^0 v$ for $N_1<i\leq N$. We have
\begin{equation}
\pa_t \overline v+ (\chi_i^1b)\cdot\na \overline v + \Lambda_{\mathbb{R}^2}\overline v =  f_1~~\text{in}~~\mathbb{R}^2,
\end{equation}
where 
\begin{align}
f_1=\chi_i^0f+\Lambda_{\mathbb{R}^2}\overline v -\chi_i^0 \l v.
\end{align}
In view of \eqref{z82b}, we have for any $\beta\in (0,1)$
\begin{align}
||	 f_1(s)||_{C^{\beta}(\mathbb{R}^2)}\lesssim ||f(s)||_{C^{\beta}(\Omega)}+ ||v(s)||_{C^{\beta}(\Omega)}.
\end{align}
We apply  \eqref{z57}, \eqref{z58} with $\alpha_1=\alpha_0$ to get that 
\begin{align}
&||v||_{L^\infty_{T}(\dot C^{2\alpha_0}(B_i^0))}+\sup_{s\in [0,T]}s^{1-\alpha_0}	||v(s)||_{\dot C^{1+\alpha_0}(B_i^0)}\lesssim J_0,\\&
\sup_{s\in [0,T]}s^{1-\alpha_0}	||v(s)||_{\dot C^{1+\alpha_0}(B_i^0)}\lesssim J_2.
\end{align}
We  cover $\Omega$ by $\cup_{i=1,...,N}B_{i}^0$, then we {{obtain \eqref{z37b} }}and 
\begin{align}
\sup_{t\in [0,T]}	|| v(t)||_{ \dot C^{2\alpha_0}}+t^{1-\alpha_0}	|| \theta(t)||_{ \dot C^{1+\alpha_0}}\lesssim J_0.
\end{align}
This follows  \eqref{z37}  by using  $||v||_{L^\infty_T( C^{\alpha_0}(\Omega))}\lesssim ||v_0||_{L^\infty(\Omega)}^{\frac{1}{2}}||v||_{L^\infty_T( C^{2\alpha_0}(\Omega))}^{\frac{1}{2}}$ and holder's inequality. \\
 We apply \eqref{z37} in $[s,T]$ with $s\in [0,T/4]$ to obtain 
\begin{align}\nonumber
	&	\sup_{t\in [s,T]}(t-s)^{1-\alpha_0}	|| v(t)||_{ \dot C^{1+\alpha_0}(\Omega)}
	\\&\quad	\lesssim\nonumber\left(1+||b ||_{L^\infty_{[s,T]}(L^\infty(\Omega))}\right)^4(T+1)^2\left( ||v(s)||_{ C^{2\alpha_0}(\Omega)}+\sup_{t\in [s,T]}(t-s)^{1-\alpha_0}||f(t)||_{ C^{\alpha_0}(\Omega)}\right)
	\\&\quad+\left(1+||b ||_{L^\infty_{[s,T]}(L^\infty(\Omega))}\right)^{\frac{6}{\alpha_0}}T^{1-\alpha_0}(1+T)\left(1+	||b||_{L^\infty_{[s,T]}( C^{\alpha_0}(\Omega))}\right)^{\frac{2}{\alpha_0}}||v(s)||_{L^\infty(\Omega)}.\nonumber
	\\&\quad	\lesssim\nonumber s^{-\gamma}\left(1+||b ||_{L^\infty_{T}(L^\infty(\Omega))}\right)^4(T+1)^2\left(\sup_{t\in [0,T]} t^\gamma ||v(t)||_{ C^{2\alpha_0}(\Omega)}+\sup_{t\in [0,T]}t^{1-\alpha_0+\gamma}||f(t)||_{ C^{\alpha_0}(\Omega)}\right)
	\\&\quad+s^{-\gamma}\left(1+||b ||_{L^\infty_{T}(L^\infty(\Omega))}\right)^{\frac{6}{\alpha_0}}T^{1-\alpha_0}(1+T)\left(1+	||b||_{L^\infty_{T}( C^{\alpha_0}(\Omega))}\right)^{\frac{2}{\alpha_0}}\sup_{t\in [0,T]} t^\gamma||v(t)||_{L^\infty(\Omega)},
\end{align}
for any $\gamma>0$.\\
Using the fact that
\begin{equation}
	\sup_{s\in [0,T/4]} \left(s^{\gamma}	\sup_{t\in [s,T]}(t-s)^{1-\alpha_0}	|| v(t)||_{ \dot C^{1+\alpha_0}(\Omega)}\right)\gtrsim 	\sup_{t\in [0,T]}t^{1-\alpha_0+\gamma}	|| v(t)||_{ \dot C^{1+\alpha_0}(\Omega)},
\end{equation}
we deduce \eqref{z84}.
\end{proof}\vspace{0.2cm}
 {{\begin{lemma} \label{bouncon}Let $f_1,f_2\in C^{2n+1}(\overline{\Omega})$ be such that $\Delta^k f_1=\Delta^k f_2=0$ on $\partial\Omega$ for any $k=0,...,n$. Then,
	\begin{equation}\label{z100}
		\Delta^k\{f_1, f_2\} =0
	\end{equation}
holds on $\pa\Omega$ for all $k=0,...,n$, where we recall \eqref{poipsitheta} that the Poisson bracket  is given by 
$\{f_1, f_2\} = \na^\bot f_1\cdot\na f_2.$ 
\end{lemma}
\begin{proof} 
First, we note by induction that our assumptions regarding the boundary conditions imply that for any $x\in \partial \Omega$,  we have
	\begin{equation}\la{indvan}
		\partial_{N}^{2k_1}	\partial_{N^{\perp}}^{k} f_1(x) = 	\partial_{N}^{2k_1}	\partial_{N^\bot}^{k} f_2(x) = 0
	\end{equation}
for all  $k\ge 0$, $k_1\ge 0$,  $2k_1+k\leq 2n+1$, where $N=n(x)$ is the unit outward normal vector at $x\in \partial\Omega$ and $N^{\perp}$ is the tangent vector $(-N_2, N_1)$. Indeed, it follows by induction that such a derivative can be represented as a linear combination of tangential derivatives of powers of the Laplacian.
The proof of this representation uses the rotation invariance of the Laplacian, which implies that $\pa^2_N = \Delta -\pa^2_{N^{\perp}}$, an identity that is then employed at the inductive step, establishing the representation, and thus \eqref{indvan}.   \\
Then, using the rotation invariance of both the Laplacian ad the Poisson bracket, we have that
	\begin{equation}\label{z101}
		\Delta^k\{f_1, f_2\} = (\partial_{N}^2+\partial_{N^\bot}^2)^{k}\left(\partial_{N^\bot } f_1\partial_{N } f_2 -\partial_{N } f_1\partial_{N^\bot } f_2\right)
	\end{equation}	
	holds  at each $x\in \overline{\Omega}$. Opening brackets with the Leibniz formula and using \eqref{indvan} in  \eqref{z101}, we obtain  \eqref{z100}.
\end{proof}}}

Now, we use Proposition \ref{eslinear+drift} to prove  local existence and higher regularity of the system \eqref{sqg} and \eqref{u}. We define 
\begin{align}\nonumber
S_{T,R}&=\left\{v\in L^\infty_{T}(L^\infty(\Omega)):v\vert_{\partial \Omega}=\Delta v\vert_{\partial \Omega}=0,\right.\\ &\left.\quad\quad\quad\quad\quad \quad\quad \quad\quad \quad\quad\quad v(t=0)=\theta_0, ||v||_{ L^\infty_{T}(L^\infty(\Omega))}\leq ||\theta_0||_{L^\infty(\Omega)},~~|||v|||_{3,T}\leq R\right\},
\end{align}
with 
\begin{align}&
|||v|||_{0,T}:=\sup_{t\in [0,T]} ||v(t)||_{\dot C^{2\alpha_0}(\Omega)},\\&
|||v|||_{1,T}:=\sup_{t\in [0,T]}t^{1-\alpha_0}	|| v(t)||_{\dot  C^{1+\alpha_0}(\Omega)},\\&
|||v|||_{3,T}:=\sup_{t\in [0,T]} ||v(t)||_{\dot C^{2\alpha_0}(\Omega)}+t^{3-\alpha_0}	|| v(t)||_{\dot  C^{3+\alpha_0}(\Omega)}.
\end{align}
We define the map $\mathcal{T}(\theta)=v$ as a solution of 
\begin{align}\label{z85}
		&	\pa_t v + b\cdot\na v + \l v = 0~~\text{in}~~\Omega,\\&
		b=\na^\bot \Lambda_D^{-1}\theta, \theta\in S_{T,R},
\end{align}
with $v(t=0)=\theta_0\in C^{3\alpha_0}(\Omega)$ and 
$\theta_0=0$ on $\partial\Omega$.  \\
In view of \eqref{z86a} and \eqref{z86b} and we have 
\begin{align}\label{z87}
&	||b||_{L^\infty_{T}( C^{\alpha_0}(\Omega))}\lesssim ||\theta_0||_{L^\infty(\Omega)}^{\frac{1}{2}} |||\theta|||_{0,T}^{\frac{1}{2}},
||b||_{L^\infty_{T}(L^\infty(\Omega))}\lesssim ||\theta_0||_{L^\infty(\Omega)}^{1-\kappa} |||\theta|||_{0,T}^{\kappa}
\end{align}
for any $\kappa\in (0,1]$. 
\begin{lemma}\label{keyle} There exist $R>0$  large enough  and $T>0$ small enough  depending on $||\theta_0||_{C^{3\alpha_0}(\Omega)}$ and $\Omega$ such that 
$\mathcal{T}(	S_{T,R})\subset S_{T,R}$  and 
\begin{align}\label{z43}
|||	\mathcal{T}(\theta_1)-\mathcal{T}(\theta_2)|||_{3,T}\leq \frac{1}{4} |||	\theta_1-\theta_2|||_{3,T}~~~\quad\forall\theta_1,\theta_2\in S_{T,R}.
\end{align}
\end{lemma}
\begin{proof} 
1) Let $\theta\in S_{T,R}$, $v=\mathcal{T}(\theta).$ We have 
\begin{equation}
||v||_{ L^\infty_{T}(L^\infty(\Omega))}\leq ||\theta_0||_{L^\infty(\Omega)}.
\end{equation}
Thus,  to get $\mathcal{T}(	S_{T,R})\subset S_{T,R}$, we have to show that $|||v|||_{3,T}\leq R$ for some  $R>0$  large and $T>0$ small enough.\\
Using Proposition \ref{eslinear+drift} with $f=0$ and $T\in (0,1)$; and using \eqref{z87} we have 
\begin{align}&
|||v|||_{0,T}	\lesssim\left(1+||\theta_0||_{L^\infty(\Omega)}^{1-\kappa} |||\theta|||_{0,T}^{\kappa}\right)^4||\theta_0||_{ C^{2\alpha_0}(\Omega)}+\left(1+|||\theta|||_{0,T}\right)^{\frac{8}{\alpha_0}}T^{1-\alpha_0}||\theta_0||_{L^\infty(\Omega)}\label{z38},\\&
|||v|||_{1,T}	\lesssim (1+|||\theta|||_{0,T})^{\frac{4}{\alpha_0}}\left( T^{\alpha_0}||\theta_0||_{{{ C^{3\alpha_0}}}(\Omega)}
+T^{1-\alpha_0}|||v|||_{0,T}\right),\label{z38b}
\end{align}
for any $\kappa\in (0,1]$.\\
In order to obtain higher regularity of $v$, we apply $\Delta_D$ to the equation \eqref{z85} and obtain
\begin{equation}
	\pa_t \Delta v + b\cdot\na \Delta v + \l \Delta v =\na^\bot \Lambda_D\theta\cdot\na  v-\sum_{j=0,1}2\partial_{x_j}b \cdot\na \partial_{x_j}v:=f.
\end{equation}
In view of  Lemma \eqref{bouncon}, we have $f=0$ on $\partial\Omega$. So, we  can apply   \eqref{z84} in proposition \ref{eslinear+drift} with  $\gamma=2$  and for any $T>0$ to obtain
	\begin{align}\nonumber
	&	\sup_{t\in [0,T]}t^{3-\alpha_0}	|| \Delta v (t)||_{ \dot C^{1+\alpha_0}(\Omega)}
	\\&	\lesssim\nonumber \left(1+||\theta_0||_{L^\infty(\Omega)}^{1-\kappa} |||\theta|||_{0,T}^{\kappa}\right)^4(T+1)^2\left(\sup_{t\in [0,T]} t^2 ||\Delta v (t)||_{ C^{2\alpha_0}(\Omega)}+\sup_{t\in [0,T]}t^{3-\alpha_0}||f(t)||_{ C^{\alpha_0}(\Omega)}\right)
	\\&\quad+\left(1+|||\theta|||_{0,T}\right)^{\frac{8}{\alpha_0}}T^{1-\alpha_0}(1+T)\sup_{t\in [0,T]} t^2||\Delta v (t)||_{L^\infty(\Omega)}.\label{z86}
\end{align}
By interpolation, 
\begin{align}\nonumber
&	\sup_{t\in [0,T]}t^{3-\alpha_0}||f(t)||_{ C^{\alpha_0}(\Omega)}=	\sup_{t\in [0,T]}t^{3-\alpha_0}||\na^\bot \Lambda_D\theta\cdot\na  v-\partial_{x_j}b \cdot\na \partial_{x_j} v(t)||_{ C^{\alpha_0}(\Omega)}\\&\quad\quad\quad\quad\quad\quad\quad\quad\quad\quad\lesssim (1+T)T^{2\alpha_0}\left(	|||\theta|||_{0,T}	+|||v|||_{0,T}\right)^{1+\alpha_0}\left(	|||\theta|||_{3,T}	+|||v|||_{3,T}\right)^{1-\alpha_0},\\&\label{z90}
\sup_{t\in [0,T]}s^2||\Delta v(t)||_{ C^{2\alpha_0}(\Omega)} \lesssim  |||v|||_{0,T}^{\frac{1-\alpha_0}{3-\alpha_0}} |||v|||_{3,T}^{\frac{2}{3-\alpha_0}},\\&
\sup_{t\in [0,T]}s^2||\Delta v(t)||_{ L^\infty(\Omega)} \lesssim T^{2\alpha_0} |||v|||_{0,T}^{\frac{1+\alpha_0}{3-\alpha_0}} |||v|||_{3,T}^{\frac{2(1-\alpha_0)}{3-\alpha_0}}.\label{z90b}
\end{align}
Thus, we obtain for any $T>0$
\begin{align}\nonumber
|||v|||_{3,T}&\lesssim |||v|||_{0,T}+ \left(1+||\theta_0||_{L^\infty(\Omega)}^{1-\kappa} |||\theta|||_{0,T}^{\kappa}\right)^4(T+1)^2  |||v|||_{0,T}^{\frac{1-\alpha_0}{3-\alpha_0}} |||v|||_{3,T}^{\frac{2}{3-\alpha_0}}
\\&\quad +\nonumber\left(1+|||\theta|||_{0,T}\right)^4(T+1)^3T^{2\alpha_0}\left(	|||\theta|||_{0,T}	+|||v|||_{0,T}\right)^{1+\alpha_0}\left(	|||\theta|||_{3,T}	+|||v|||_{3,T}\right)^{1-\alpha_0}
\\&\quad+\left(1+|||\theta|||_{0,T}\right)^{\frac{8}{\alpha_0}}T^{1+\alpha_0}(1+T) |||v|||_{0,T}^{\frac{1+\alpha_0}{3-\alpha_0}} |||v|||_{3,T}^{\frac{2(1-\alpha_0)}{3-\alpha_0}}.
\end{align}
By the Holder inequality, 
\begin{align}\nonumber
|||v|||_{3,T}&\lesssim |||v|||_{0,T}+ \left(1+||\theta_0||_{L^\infty(\Omega)}^{1-\kappa} |||\theta|||_{0,T}^{\kappa}\right)^{\frac{4(3-\alpha_0)}{1-\alpha_0}}(T+1)^{\frac{2(3-\alpha_0)}{1-\alpha_0}}  |||v|||_{0,T}
\\&\quad +\nonumber\left(1+|||\theta|||_{0,T}\right)^4(T+1)^3T^{2\alpha_0}\left(	|||\theta|||_{0,T}	+|||v|||_{0,T}\right)^{1+\alpha_0}\left(	|||\theta|||_{3,T}	+|||v|||_{3,T}\right)^{1-\alpha_0}
\\&\quad+\left(1+|||\theta|||_{0,T}\right)^{\frac{8(3-\alpha_0)}{\alpha_0(1+\alpha_0)}}T^{3-\alpha_0}(1+T)^{\frac{3-\alpha_0}{1+\alpha_0}} |||v|||_{0,T}.\label{z60}
\end{align}
Combining this with \eqref{z38} we deduce for any $T\in (0,1)$
\begin{align}\nonumber
&|||v|||_{3,T}\lesssim |||v|||_{0,T}+ \left(1+||\theta_0||_{L^\infty(\Omega)}^{1-\kappa} R^{\kappa}\right)^{\frac{4(3-\alpha_0)}{1-\alpha_0}} |||v|||_{0,T}
\\&\quad\quad\quad\quad +\left(1+R\right)^4T^{2\alpha_0}\left(R	+|||v|||_{3,T}\right)^{2}+\left(1+R\right)^{\frac{8(3-\alpha_0)}{\alpha_0(1+\alpha_0)}}T^{3-\alpha_0}|||v|||_{0,T},\label{z88a}\\&
|||v|||_{0,T}	\lesssim\left(1+||\theta_0||_{L^\infty(\Omega)}^{1-\kappa} R^{\kappa}\right)^4||\theta_0||_{ C^{2\alpha_0}(\Omega)}+\left(1+R\right)^{\frac{8}{\alpha_0}}T^{1-\alpha_0}||\theta_0||_{L^\infty(\Omega)}.\label{z88b}
\end{align}
Choosing $\kappa>0$ small such that $\frac{4\kappa(3-\alpha_0)}{1-\alpha_0}+4\kappa\leq \frac{1}{10}$.  In view of \eqref{z88a} and \eqref{z88b}, we can take $R\geq 1$ large, then $T$ small such that 
\begin{align}
&|||v|||_{3,T}\leq C |||v|||_{0,T}+ R^{\frac{1}{5}} |||v|||_{0,T}+ R^{-10}|||v|||_{3,T}(1+|||v|||_{3,T})+R^{-10},\\&
|||v|||_{0,T}	\leq  R^{\frac{1}{5}}.
\end{align}
These imply $|||v|||_{3,T}\leq R$.\\
2) Let $\theta_1,\theta_2\in S_{T,R}$ be such that $\theta_1(t=0)=\theta_2(t=0)=\theta_0$. Set $v_j=\mathcal{T}(\theta_j)$ and $\overline v=v_1-v_2, \overline \theta=\theta_1-\theta_2$. We can write 
\begin{equation}
\pa_t \overline v + \na^\bot \Lambda_D^{-1}\theta_1\cdot\na \overline v + \l \overline v=-\na^\bot \Lambda_D^{-1}\overline \theta\cdot\na v_2,
\end{equation}
with $ \overline v (t=0)=0$. \\
Using \eqref{z38b}, we have
\begin{equation}
||| v_1|||_{1,T}+||| v_2|||_{1,T}\lesssim C(R,||\theta_0||_{ C^{3\alpha_0}(\Omega)}) T^{\alpha_0}~~~\forall ~~T\in (0,1].
\end{equation}
Because  $\na^\bot \Lambda_D^{-1}\overline \theta\cdot\na v_2\vert_{\partial \Omega} =0$, so we apply \eqref{z37} in  Proposition \ref{eslinear+drift} with $f=-\na^\bot \Lambda_D^{-1}\overline \theta\cdot\na v_2$ to obtain that for any $T\in (0,1)$
\begin{align}\label{z42}
|||\overline v|||_{0,T}	\lesssim (1+	R)^4\sup_{s\in [0,T]}s^{1-\alpha_0}||\na^\bot \Lambda_D^{-1}\overline \theta\cdot\na v_2(s)||_{ C^{\alpha_0}}\lesssim (1+	R)^5T^{\alpha_0}	|||\overline \theta|||_{0,T}.	
\end{align}
On the other hand, as \eqref{z86} we have 
\begin{equation}
\pa_t \Delta\overline v + \na^\bot \Lambda_D^{-1}\theta_1\cdot\na \Delta\overline v + \l \overline \Delta v=f_1,
\end{equation}
where 
\begin{align}
f_1=	- \Delta\left(\na^\bot \Lambda_D^{-1}\overline \theta\cdot\na v_2\right)-\na^\bot \Lambda_D\theta_1\cdot\na \overline v -\sum_{j=0,1}2\na^\bot\partial_{x_j} \Lambda_D^{-1}\theta_1\cdot\na \partial_{x_j}\overline v,
\end{align}
and $f_1\vert_{\partial \Omega} =0$.\\
It is true that 
\begin{align}
\sup_{t\in [0,T]}t^{3-\alpha_0}||f_1(t)||_{ C^{\alpha_0}(\Omega)}\lesssim  C(R,||\theta_0||_{ C^{3\alpha_0}(\Omega)})T^{\alpha_0} \left(	|||\overline \theta|||_{3,T}+	|||\overline v|||_{3,T}\right).
\end{align}
{{Thus, we apply \eqref{z84} of Proposition \ref{eslinear+drift}  with $f=f_1$ and $\gamma=2$ to obtain for any $T\in (0,1)$
	\begin{align}\nonumber
	&	\sup_{t\in [0,T]}t^{3-\alpha_0}	|| \overline v (t)||_{ \dot C^{3+\alpha_0}(\Omega)}\sim \sup_{t\in [0,T]}t^{3-\alpha_0}	|| \Delta\overline v (t)||_{ \dot C^{1+\alpha_0}(\Omega)}
	\\&	\lesssim\nonumber (1+R)^4\left(\sup_{t\in [0,T]} t^2 ||\overline v (t)||_{ C^{2+2\alpha_0}(\Omega)}+\sup_{t\in [0,T]}t^{3-\alpha_0}||f_1(t)||_{ C^{\alpha_0}(\Omega)}\right)
	\\&\quad+(1+R)^{\frac{8}{\alpha_0}}T^{1-\alpha_0}\sup_{t\in [0,T]} t^2||\Delta \overline v (t)||_{L^\infty(\Omega)}\nonumber\\&
	\overset{\eqref{z90},\eqref{z90b}}\lesssim   C(R,||\theta_0||_{ C^{3\alpha_0}(\Omega)})\left(|||\overline v|||_{0,T}^{\frac{1-\alpha_0}{3-\alpha_0}} |||\overline v|||_{3,T}^{\frac{2}{3-\alpha_0}}+T^{\alpha_0} \left(	|||\overline \theta|||_{3,T}+	|||\overline v|||_{3,T}\right)\right)
	\\&\quad+ C(R,||\theta_0||_{ C^{3\alpha_0}(\Omega)})T^{1+\alpha_0}|||\overline v|||_{0,T}^{\frac{1+\alpha_0}{3-\alpha_0}} |||\overline v|||_{3,T}^{\frac{2(1-\alpha_0)}{3-\alpha_0}}.
\end{align}}}
Combining this with \eqref{z42}, we obtain  for any $T\in (0,1)$
\begin{align}
|||\overline v|||_{3,T} 
\lesssim C(R,||\theta_0||_{ C^{3\alpha_0}(\Omega)})T^{\alpha_0} \left(	|||\overline \theta|||_{3,T}+	|||\overline v|||_{3,T}\right).
\end{align} 
Thus, we deduce \eqref{z43} by taking $T>0$ small enough.
\end{proof}
\begin{coro}   Assume that $\theta\in S_{R,T}$ is a solution  of the system \eqref{sqg} and \eqref{u} for some $R,T>0$. Then, 
\begin{align}\label{z41}
|||\theta|||_{3,T}\lesssim |||\theta|||_{0,T}(1+T+|||\theta|||_{0,T})^{\frac{30}{\alpha_0(1-\alpha_0)}}.
\end{align}
\end{coro}
\begin{proof} Using \eqref{z60} with $\kappa=1$, we have 
\begin{align}\nonumber
&|||\theta|||_{3,T}\lesssim |||\theta|||_{0,T}+\left(1+ |||\theta|||_{0,T}\right)^{\frac{4(3-\alpha_0)}{1-\alpha_0}}(T+1)^{\frac{2(3-\alpha_0)}{1-\alpha_0}}  |||\theta|||_{0,T}
\\& +\left(1+|||\theta|||_{0,T}\right)^4(T+1)^3T^{2\alpha_0}	|||\theta|||_{0,T}^{1+\alpha_0}|||\theta|||_{3,T}^{1-\alpha_0}+\left(1+|||\theta|||_{0,T}\right)^{\frac{8(3-\alpha_0)}{\alpha_0(1+\alpha_0)}}T^{3-\alpha_0}(1+T)^{\frac{3-\alpha_0}{1+\alpha_0}} |||\theta|||_{0,T}.
\end{align}
Thus, by Holder's inequality, we obtain \eqref{z41}. 
\end{proof}
 {{\begin{rem} \label{highgre1} As in the proof of \eqref{z60}, we use Lemma \ref{bouncon} and apply \eqref{z84} of Proposition \ref{eslinear+drift}  to 
	\begin{equation}
		\pa_t \Delta^k v + \na^\bot \Lambda_D^{-1}\theta\cdot\na \Delta^ k  v+ \l \Delta^k v =\na^\bot \Lambda_D^{-1}\theta\cdot\na \Delta^ k  v-\Delta^ k\left(\na^\bot \Lambda_D^{-1}\theta\cdot\na   v\right)
	\end{equation}
with $\gamma=2k, k\in \mathbb{N}$ to deduce that for any $n\in \mathbb{N},$
\begin{equation}
	\sup_{t\in [0,T]} ||v(t)||_{\dot C^{2\alpha_0}(\Omega)}+t^{2n-\alpha_0}	|| v(t)||_{\dot  C^{2n+1+\alpha_0}(\Omega)}\lesssim (1+T+ \sup_{t\in [0,T]}||\theta(t)||_{\dot C^{2\alpha_0}(\Omega)})^{c_n}
\end{equation}
for some $c_n$ and $\Delta^k \theta(t)\vert_{\partial\Omega}=0$
for any $t\in (0,T]$ and $k=0,...,n$.\end{rem}}}
\begin{thm}\la{calphaglob} Let $\theta_0\in {{ C^{3\alpha_0}(\Omega)}}$ and  
	$\theta_0=0$ on $\partial\Omega$ for some $\alpha_0\in (0,1/10)$.  Then,  the  equation \eqref{sqg} has a global unique solution $\theta$ satisfying 
\begin{equation}
||\theta(t)||_{C^{2\alpha_0}(\Omega)}+t^{3-\alpha_0}||\na_{t,x}^3 \theta(t)||_{L^\infty(\Omega)}\leq  C(||\theta_0||_{ C^{3\alpha_0}(\Omega)},\Omega) e^{-c t},
\end{equation}
for any $t\geq 0$ and for some $c>0.$
\end{thm}
\begin{proof}  Using Lemma \ref{keyle} and the Banach fixed point theorem,  the equation \eqref{sqg} has a local unique solution $\theta$ in $[0,T_1]$ satisfying 
\begin{equation}\label{z44}
|||\theta|||_{3,T_1}<\infty.
\end{equation}
for some $T_1>0$. Moreover, we also have $\theta\in C([0,T_1],C^{\alpha_0}(\Omega))$.\\
Define 
\begin{align}
T^*:=\sup \{T: \theta ~~\text{exists on} ~[0,T] ~\text{with} ~	|||\theta|||_{4,T}<\infty\}.
\end{align}
where 
\begin{align}
|||\theta|||_{4,T}=\sup_{t\in [0,T]}	||\theta(t)||_{C^{2\alpha_0}(\Omega)}+t^{3-\alpha_0}||\na_x^3 \theta(t)||_{L^\infty(\Omega)}.
\end{align}
In view of \eqref{z44}, we have $T^*>0$. \\
We prove $T^*=+\infty$.
Indeed, we assume $0<T^*<\infty$. So, we can apply  Theorem \ref{holderR} for $\theta$ in $[0,T]$ for any  $T<T*$ to get that 
\begin{align}
||\theta(t)||_{C^{\alpha}(\Omega)}^{\frac{1}{\alpha}-\frac{\alpha}{2-\alpha}}\leq  C(||\theta_0||_{ C^{3\alpha_0}(\Omega)},\Omega)(t+1)^{\frac{36}{\alpha^2}},
\end{align}
for any $t\leq T^*$ and 
for some $\alpha\in (0,\alpha_0)$ small enough.\\ Using \eqref{z41}, we get
\begin{align}\label{z92}
||\theta(t)||_{C^{\alpha}(\Omega)}+t^{3-\alpha/2}	|| \theta(t)||_{\dot  C^{3+\alpha/2}(\Omega)}\leq   C(||\theta_0||_{ C^{3\alpha_0}(\Omega)},\Omega)(t+1)^{m_0},
\end{align}
for any $t\in (0,T^*]$ and for some $m_0>0$. From this and \eqref{z44} we get a contradiction.\\
Therefore, for any $t\geq 1$
\begin{equation}\label{z91}
|| \theta(t)||_{\dot  C^{3+\alpha/2}(\Omega)}\leq   C(||\theta_0||_{ C^{3\alpha_0}(\Omega)},\Omega)(t+1)^{m_0}.
\end{equation}
On the other hand, we have 
\begin{equation}
\partial_t( ||\theta(t)||_{L^2(\Omega)}^2)+2 ||\Lambda_D^{\frac{1}{2}}\theta(t)||_{L^2(\Omega)}^2=0.
\end{equation}
{{Using \eqref{norms},  
\begin{equation}
\partial_t( ||\theta(t)||_{L^2(\Omega)}^2)+2\mu_1 ||\theta(t)||_{L^2(\Omega)}^2\leq  0.
\end{equation}
This implies 
\begin{equation}
||\theta(t)||_{L^2(\Omega)}\leq e^{-2\mu_1t}	||\theta_0||_{L^2(\Omega)}.
\end{equation}
Using  interpolation, we obtain from  this and  \eqref{z91}  that for any $t\geq 1$}}
\begin{equation}
|| \theta(t)||_{  C^{3+\alpha/4}(\Omega)}\leq   C(||\theta_0||_{ C^{3\alpha_0}(\Omega)},\Omega)e^{-ct},
\end{equation}
for some $c>0$.
Since 
$	\pa_t v =- \na^\bot \Lambda_D^{-1}\theta\cdot\na v -\l v $, 
\begin{align}
||\pa_{t,x}^3 v(t)||_{L^\infty(\Omega)} \lesssim	|| \theta(t)||_{  C^{3+\alpha/4}(\Omega)}\left(	|| \theta(t)||_{ L^\infty(\Omega)}+1\right)^{10}.
\end{align}
Thus, we obtain the result. 
\end{proof}
 {{\begin{rem}\label{highgre2} In view of Remark \eqref{highgre1}, we obtain that  the  solution $\theta$ satisfies $\Delta^n\theta(t)\vert_{\partial \Omega}=0$ for $t>0$ and 
	\begin{equation}
	\sup_{t\geq 0}	||\theta(t)||_{C^{2\alpha_0}(\Omega)}+t^{2n+1-\alpha_0}||\na_{t,x}^{2n+1} \theta(t)||_{L^\infty(\Omega)}\leq  C(||\theta_0||_{ C^{3\alpha_0}(\Omega)},\Omega,n) e^{-c_n t},
	\end{equation}
	for any $n\in \mathbb{N}$ and for some $c_n>0$.
\end{rem}}}

\section {Appendix 1}\la{change}
This is a construction in $d=2$. A similar construction can be done in any dimension.
We take without loss of generality $\varphi(x_1)$ to have $\varphi' = 0$ for $|x_1|\ge \ell$ and $|\varphi''|\le C_0$. We assumed $\varphi(0) = \varphi'(0) = 0$ and $|\varphi'(x_1)| \le \epsilon$ for all $x_1$. Clearly, $\varphi (x_1) = h_1 = \varphi(\ell)$ for $x_1\ge \ell$ and $\varphi(x_1) = \varphi (-\ell) = h_2$ 
for $x_1\le -\ell$. 
We take
\be
Y_2(x) = x_2-\varphi(x_1).
\la{y2}
\ee
Now 
\be
\na Y_2(x) = e_2- \vp'(x_1) e_1
\la{nay2}
\ee
is a globally defined vector field
\be
N= e_2 -\vp'(x_1)e_1.
\la{N}
\ee
and if we build a function  $Z(x)$ so that $N\cdot \na Z=0$. We set
\be
\na Z  = \gamma(x)( e_1 + \vp'(x_1)e_2).
\la{nazalpha}
\ee
This can be done if, and only if,
\be
\pa_2\gamma = \pa_1(\vp' \gamma),
\la{eqalpha}
\ee
which is a first order equation
\be
N\cdot\na\gamma = \vp''\gamma.
\la{nalpheq}
\ee
We solve this on characteristics and show that the solution is global. It is good to  set data on the curve 
\be
\Gamma = \{x\left | \right. x_2 = \vp(x_1)\}.
\la{Gmma}
\ee
The characteristics are
\be
\fr{d\xi}{ds} = -\vp'(\xi)
\la{xichar}
\ee
with $\xi(0) = x_1$ 
and
\be
\fr{d\eta}{ds} = 1
\la{etas}
\ee
with initial data $\eta(0) = \vp(x_1)$.  Clearly
\be
\eta(s) = \vp(x_1) + s.
\la{etasexplicit}
\ee
On characteristics, $\gamma$ solves
\be
\fr{d}{ds} \gamma (\xi(s), \eta(s)) = \vp''(\xi(s))\gamma(\xi(s), \eta(s)).
\la{alphaxieta}
\ee
We set the initial data for $\gamma $ on the curve $\Gamma$,
\be
\gamma (\xi(0), \eta(0)) = \gamma(x_1, \vp(x_1)) = 1.
\la{alphainit}
\ee
It is clear from our assumptions that the characteristics exist for all $s$. If $|x_1|\ge \ell$ the characteristics are vertical lines
$\xi(s) = x_1$ . Also, if $x_1=0$ the characteristic is the $x_2$ axis. Moreover $\gamma(x_1,x_2) =  1$  for $|x_1|\ge \ell$.
We note that $\gamma(0, s) = e^{\vp''(0)s}$.  It is instructive to look at the case $\vp = \fr{x_1^2}{2}$ for which $\xi = x_1e^{-s}$
and $\gamma = e^s$. In this case we can determine $\gamma(x_1,x_2)$ implicitly from the relation 
$x_2 = \gamma^2 \fr{x_1^2}{2} + \log \gamma$.\\
We define $Y_1(x)$ by setting
\be
Y_1(x)  = x_1 + \chi(x_2)\int_0^{x_1}(\gamma(z, x_2)-1)dz
\la{Y1}
\ee
where $\chi(x_2) = 1$ for $|x_2|\le H$ and $\chi(x_2) = 0$ for $|x_2| \ge 2H$ with $H> \epsilon\ell $. We have
\be
\pa_1 Y_1 = \chi(x_2) \gamma(x_1,x_2) + (1-\chi(x_2))
\la{pa1y1}
\ee
and, using \eqref{eqalpha}, 
\be
\pa_2 Y_1 = \chi(x_2) \gamma(x_1,x_2) \vp'(x_1)  + \fr{d\chi}{dx_2}\int_0^{x_1}(\gamma(z, x_2)-1)dz.
\la{pa2y2}
\ee
Therefore
\be
\na Y_1 \cdot \na Y_2 = -\vp' (1-\chi(x_2))  +  \fr{d\chi}{dx_2}\int_0^{x_1}(\gamma(z, x_2)-1)dz
\la{cross}
\ee
vanishes in a neighborhood of $\Gamma$. We can arrange the cutoff $\chi$ so that $x\mapsto (Y_1,Y_2)$ is a global diffeomorphism. This is done by noting from \eqref{pa1y1} that  $\pa_1Y_1\ge \fr{1}{2}$ may be arranged  by ensuring $\gamma \ge \fr{1}{2}$ on the support of $\chi$. This follows if $H$ is small enough, which is possible if  $\epsilon\ell$ is small. Then the Jacobian $\det\na Y$ is bounded away from zero, provided $\epsilon$ and $\ell$ are small enough and  $|\chi'|\le CH^{-1}$. This shows that $x\mapsto Y$ is locally injective. To show global invertibility we show that $Y$ can be continuously deformed to the identity, by taking $\vp$ to zero.
We note also 
\be
\ba
Y_1(x) = x_1, \quad \text{for} \; |x_2|\ge 2H,\\
Y_2(x) = x_2. \quad \text{for} \; |x_1|\ge \ell.
\ea
\la{yid}
\ee
Once we have established the existence of the smooth diffeomorphism $X = Y^{-1}$, we have 
the intertwining \eqref{interdeltaL}.

\section {Appendix 2: Estimates of heat kernels}\la{kernelestimes}
In this section, we establish estimates for the heat kernel of the operator
\begin{equation}
L=-\operatorname{div}_x(A(x)\na_x)
\end{equation}
in $\mathbb R^d$, {{ for $d\geq 2$}}
where $A$ is a symmetric matrix-valued function in $\mathbb{R}^d$ satisfying
\begin{align}\label{z65aa}
&A(x)\geq  c_1 I\quad\quad \forall x\in \mathbb{R}^d,\\&
||\na A||_{L^\infty}+||A||_{L^\infty}\leq c_2,\label{z65bb}
\end{align}
with {{ constants}} $c_1,c_2>0$.\\
Let be $H_L(x,y,t)$  the kernel of $\partial_t+L$ in $\mathbb{R}^d\times (0,\infty)$ i.e 
\begin{align}\nonumber
&	\partial_t H_L(x,x+z,t)-\operatorname{div}_z(A(x+z)\nabla_z H_L(x,x+z,t))=0,\\&\lim_{t\to 0}H(x,x+z,t)=\delta_{z=0},\label{z63}
\end{align}
for any $x,z\in \mathbb{R}^d$.\\
It is well-known that the kernel $H_L(x,y,t)$ satisfies 
\begin{equation}\label{z17a}
\frac{1}{c_4t^{\frac{d}{2}}}\exp(-c_4\frac{|z|^2}{t})\leq 		H_L(x,x+z,t)\leq  \frac{1}{c_3t^{\frac{d}{2}}}\exp(-c_3\frac{|z|^2}{t})
\end{equation}
for some $c_3,c_4>0$, see \cite[Theorem 3.3.4]{davies2}. Moreover, we also have 
\begin{equation}\label{z17b}
|\na_xH_L(t,x,y)|+	|\na_yH_L(t,x,y)|\leq  \frac{C}{t^{\frac{d}{2}}\min\{t,1\}^{\frac{1}{2}}}\exp(-\frac{c_3}{4}\frac{|x-y|^2}{t}),
\end{equation}
see \cite[chapter IV, section 13, (13.1)]{ladysolo} for case $t\in (0,1]$; when $t\geq 1$, it follows by using semigroup property:  $H_L(t,x,y)=\int_{\mathbb{R}^d}H_L(\frac{1}{2},x,z)H_L(t-\frac{1}{2},z,y)dz$.\\
In particular, \eqref{z17b} implies 
\begin{equation}\label{z7}
||e^{-tL}\operatorname{div}(g)||_{L^\infty(\mathbb{R}^d)}\leq C \min\{t,1\}^{-\frac{1-a}{2}} ||g||_{\dot C^{a}(\mathbb{R}^d)}
\end{equation}
for any $a\in (0,1)$.\\
For fixed $y\in \mathbb{R}^d$, we {{define}} 
\begin{equation}
	L_{y}=-\operatorname{div}_x(A(y)\na _x).
\end{equation}
Its heat kernel is given in \eqref{z47},
\begin{equation}\label{z47z}
G_{A(y)}(z,t)=\frac{1}{\sqrt{\det A(y)}(4\pi t)^{\frac{d}{2}}}\exp(-\frac{(A(y)^{-1}z\cdot z)}{4t}),
\end{equation}
and the square root of the operator $L_y$ is given by
\begin{align}
L_{y}^{\frac{1}{2}}u(x)=	\frac{	\tilde{c}_0}{\sqrt{\det A(y)}}\int_{\mathbb{R}^d}\frac{u(x)-u(x+z)}{(A(y)^{-1}z\cdot z)^{\frac{d+1}{2}}}dz.
\end{align}
The Lipschitz continuity properties 
\begin{align}\label{z53}
&\sup_{x\in \mathbb{R}^d}	|L_{y_1}^{\frac{1}{2}}u(x)-	L_{y_2}^{\frac{1}{2}}u(x)|\lesssim\min\{|y_1-y_2|,1\} ||u||_{\dot C^{a}}^{a}||u||_{\dot C^{1+a}}^{1-a},\\&
{{ |(L_{y_1}^{\frac{1}{2}}u(x)-	L_{y_2}^{\frac{1}{2}}u(x))-(L_{y_1}^{\frac{1}{2}}u(y)-	L_{y_2}^{\frac{1}{2}}u(y))|\lesssim\min\{|y_1-y_2|,1\} |x-y|^a||u||_{\dot C^{1+a}}}}.\label{z53b}
\end{align}
hold for any $a\in (0,1)$.{{  Indeed, we write 
\begin{align}\nonumber
&	(L_{y_1}^{\frac{1}{2}}u(x)-	L_{y_2}^{\frac{1}{2}}u(x))-(L_{y_1}^{\frac{1}{2}}u(y)-	L_{y_2}^{\frac{1}{2}}u(y))\\&\quad\quad\quad\quad\quad\quad\quad\quad= \int_{\mathbb{R}^d} \digamma(y_1,y_2,z)(\delta_{z}u(y)-\delta_{z}u(x))dz,
\end{align}
with 
\begin{equation}
\digamma(y_1,y_2,z)=	\frac{	\tilde{c}_0}{\sqrt{\det A(y_1)}}\frac{1}{(A(y_1)^{-1}z\cdot z)^{\frac{d+1}{2}}}-	\frac{	\tilde{c}_0}{\sqrt{\det A(y_2)}}\frac{1}{(A(y_2)^{-1}z\cdot z)^{\frac{d+1}{2}}}.
\end{equation}
Then, \eqref{z53b} follows by using the following inequalities
\begin{align}
&	|\delta_{z}u(y)-\delta_{z}u(x)|\leq |x-y|^{a+\epsilon_0}|z|^{1-\epsilon_0} ||u||_{\dot C^{1+a}}~~\text{for}~~|z|\geq |x-y|,\\&
	|(\delta_{z}u(y)-z\na u(y))-(\delta_{z}u(x)-z\na u(x))|\leq |x-y|^{a-\epsilon_0}|z|^{1+\epsilon_0} ||u||_{\dot C^{1+a}}~~~~\text{for}~~|z|\leq |x-y|,\\&
|	\digamma(y_1,y_2,z)|\lesssim\frac{ \min\{1,|y_1-y_2|\}}{|z|^{d+1}},
\end{align}
for $\epsilon_0\in (0,\frac{a}{2})$.}}
\vspace{0.2cm}\\

{{In order to study $\na L^{-\frac{1}{2}}$ and $L^{\frac{1}{2}}$, we make use of the following fine properties of the kernel $H_L(x,y,t)$. }}
\begin{lemma} \label{heatke} The following inequalities hold
\begin{align}\label{z17c1}
&\left|H_L(x,x+z,t)-G_{A(x+z)}(z,t)\right|+\left|H_L(x+z,x,t)-G_{A(x+z)}(z,t)\right|\lesssim \frac{\exp(-c_0\frac{|z|^2}{t})}{t^{\frac{d-1}{2}}(1+t)^{\frac{1}{2}}},\\&\nonumber	\left|\nabla_zH_L(x,x+z,t)-\nabla_zG_{A(x+z_0)}(z,t)\vert_{z_0=z}\right|\\&\quad\quad\quad\quad+\left|\nabla_zH_L(x+z,x,t)-\nabla_zG_{A(x+z_0)}(z,t)\vert_{z_0=z}\right|\lesssim \frac{\log (2+t)}{t^{\frac{d}{2}}}\exp(-c_0\frac{|z|^2}{t}),\label{z17c2}\\&\nonumber
\left|\delta_h^z\nabla_zH_L(x,x+z,t)-\delta_h^z\nabla_zG_{A(x+z_0)}(z,t)\vert_{z_0=z}\right|\\&\quad\quad\quad\quad+\left|\delta_h^z\nabla_zH_L(x+z,x,t)-\delta_h^z\nabla_zG_{A(x+z_0)}(z,t)\vert_{z_0=z}\right|\lesssim|h|\log(2+\frac{\sqrt{t}}{|h|})\frac{\exp(-c_0\frac{|z|^2}{t})}{t^{\frac{d}{2}}\min\{t,1\}^{\frac{1}{2}}}\label{z17c3}
\end{align}
for any $t>0$ and $|h|\leq \sqrt{\min\{t,1\}}/10$. Here $\delta_h f(x):=f(x+h)-f(x)$.\\
\end{lemma}

\begin{proof} {{Our proof is based on a method of freezing coefficients. It is similar to  \cite[chapter IV, section 13]{ladysolo} and it is probably not new.}}\vspace{0.2cm}\\
	1) Using \eqref{z63}, we have for any $z_0\in \mathbb{R}^d$, 
\begin{align}
\partial_t H_L(x,x+z,t)-\operatorname{div}_z(A(x+z_0)\nabla_z H_L(x,x+z,t))=\operatorname{div}_{z}(\left(A(x+z)-A(x+z_0)\right)\nabla_{z} H_L(x,x+z,t))
\end{align}
so, 	integrating by parts, we obtain
\begin{align}\nonumber
&	H_L(x,x+z,t)-G_{A(x+z_0)}(z,t)\\&=
\int_{0}^{t}\int_{\mathbb{R}^d}\nabla_zG_{A(x+z_0)}(z-z',t-s)\left(A(x+z')-A(x+z_0)\right)\nabla_{z'} H_L(x,x+z',s) dz'ds.\label{z64a}
\end{align}
Thus, for any $t\in (0,1]$
\begin{align}\nonumber
\left|H_L(x,x+z,t)-G_{A(x+z)}(z,t)\right|&\overset{\eqref{z17b}}\lesssim
\int_{0}^{t}\int_{\mathbb{R}^d}\frac{\exp(-c_5\frac{|z-z'|^2}{t-s})}{ (t-s)^{\frac{d+1}{2}}}\min\{|z-z'|,1\}\frac{\exp(-c_5\frac{|z'|^2}{s})}{s^{\frac{d}{2}}\min\{s,1\}^{\frac{1}{2}}}dz'ds\\&\lesssim\nonumber \exp(-c_5\frac{|z|^2}{8t})
\int_{0}^{t}\int_{\mathbb{R}^d}\frac{\exp(-c_5\frac{|z-z'|^2}{4(t-s)})}{ (t-s)^{\frac{d}{2}}}\frac{\exp(-c_5\frac{|z'|^2}{4s})}{s^{\frac{d+1}{2}}}dz'ds
\\&\lesssim \nonumber\exp(-c_5\frac{|z|^2}{8t})
\int_{0}^{t}\frac{\min\{t-s,s\}^{\frac{d}{2}}}{ (t-s)^{\frac{d}{2}}}\frac{1}{s^{\frac{d+1}{2}}}ds\\&\lesssim \frac{1}{t^{\frac{d-1}{2}}}\exp(-c_5\frac{|z|^2}{8t}) ,\label{z66a}
\end{align}
holds with $c_5>0$.
Here we have used the fact that 
\begin{align}
&\frac{	\min\{|z-z'|,1\}}{(t-s)^{\frac{1}{2}}}\exp(-c_5\frac{|z-z'|^2}{2(t-s)}){{\lesssim 1}},\\&
\frac{|z-z'|^2}{t-s}+	\frac{|z'|^2}{s}\geq \frac{|z|^2}{2t}.
\end{align}
Therefore, 
\begin{equation}
\left|H_L(x,x+z,t)-G_{A(x+z)}(z,t)\right|\lesssim \frac{1}{t^{\frac{d-1}{2}}(1+t)^{\frac{1}{2}}}\exp(-c_0\frac{|z|^2}{t})
\end{equation}
holds for some $c_0>0$.\\
2) We apply $\na_z$ to both sides of \eqref{z64a}, then we take $z_0=z$ to deduce that 
\begin{align}\nonumber
&	\na_zH_L(x,x+z,t)-	\na_zG_{A(x+z_0)}(z,t)\vert_{z_0=z}\\&=
\int_{0}^{t}\int_{\mathbb{R}^d}\nabla_z^2G_{A(x+z_0)}(z-z',t-s)\vert_{z_0=z}\left(A(x+z')-A(x+z)\right)\nabla_{z'} H_L(x,x+z',s) dz'ds.\label{z64b}
\end{align}
So, as in \eqref{z66a}
\begin{align}\nonumber
\left|\nabla_zH_L(x,x+z,t)-\nabla_zG_{A(x+z_0)}(z,t)\vert_{z_0=z}\right|&\overset{\eqref{z17b}}\lesssim \exp(-c_5\frac{|z|^2}{8t})
\int_{0}^{t}\frac{\min\{t-s,s\}^{\frac{d}{2}}}{(1+t-s)^{\frac{1}{2}} (t-s)^{\frac{d+1}{2}}}\frac{1}{s^{\frac{d}{2}}\min\{s,1\}^{\frac{1}{2}}}ds\\&\lesssim \frac{\log (2+t)}{t^{\frac{d}{2}}}\exp(-c_5\frac{|z|^2}{8t}).\end{align}
3)  We apply $\delta_h^z$ to both sides of \eqref{z64b} with $|h|\leq \sqrt{\min\{t,1\}}/10$
\begin{align}\nonumber
&	\left|\delta_h^z\nabla_zH_L(x,x+z,t)-\left(\nabla_zG_{A(x+z_0)}(z+h,t)\vert_{z_0=z+h}-\nabla_zG_{A(x+z_0)}(z,t)\vert_{z_0=z}\right)\right|\\&\overset{\eqref{z17b}}\lesssim  \nonumber	\int_{0}^{t}\int_{\mathbb{R}^d}\left|\nabla_z^2G_{A(x+z_0)}(z+h-z',t-s)\vert_{z_0=z+h}\left(A(x+z')-A(x+z+h)\right)\right.\\&\quad\quad\quad\quad\left.-\nabla_z^2G_{A(x+z_0)}(z-z',t-s)\vert_{z_0=z}\left(A(x+z')-A(x+z)\right)\right| \frac{\exp(-c_5\frac{|z'|^2}{s})}{s^{\frac{d}{2}}\min\{s,1\}^{\frac{1}{2}}} dz'ds\nonumber\\& := \int_{0}^{t} M(h,s) ds.\label{z65}
\end{align}
When $s\in [t-|h|^2,t]$, as \eqref{z66a} we estimate 
\begin{align}
M(h,s) &\lesssim\nonumber \exp(-c_5\frac{|z|^2}{8t})\sum_{i=0,1}
\int_{\mathbb{R}^d}\frac{\exp(-c_5\frac{|z-z'+ih|^2}{4(t-s)})}{ (t-s)^{\frac{d+1}{2}}}\frac{\exp(-c_5\frac{|z'|^2}{4s})}{s^{\frac{d}{2}}\min\{s,1\}^{\frac{1}{2}}}dz'\\&\lesssim \frac{1}{ (t-s)^{\frac{1}{2}}}\frac{\exp(-c_5\frac{|z|^2}{8t})}{t^{\frac{d}{2}}\min\{t,1\}^{\frac{1}{2}}}.\label{z66}
\end{align}
When $s\in [0, t-|h|^2]$, we have 
\begin{align}\nonumber
M(h,s) &\lesssim\int_{\mathbb{R}^d}\left|\nabla_z^2G_{A(x+z_0)}(z+h-z',t-s)\vert_{z_0=z+h}-\nabla_z^2G_{A(x+z_0)}(z-z',t-s)\vert_{z_0=z+h}\right|\\&\nonumber\quad\quad\quad\times\min\{|z-z'|,1\}\frac{\exp(-c_5\frac{|z'|^2}{s})}{s^{\frac{d}{2}}\min\{s,1\}^{\frac{1}{2}}}  dz'\\&\nonumber+\int_{\mathbb{R}^d}\left|\nabla_z^2G_{A(x+z_0)}(z-z',t-s)\vert_{z_0=z+h}-\nabla_z^2G_{A(x+z_0)}(z-z',t-s)\vert_{z_0=z}\right|\\&\nonumber\quad\quad\quad\times\min\{|z-z'|,1\}\frac{\exp(-c_5\frac{|z'|^2}{s})}{s^{\frac{d}{2}}\min\{s,1\}^{\frac{1}{2}}}  dz'\\&+
\int_{\mathbb{R}^d}\left|\nabla_z^2G_{A(x+z_0)}(z+h-z',t-s)\vert_{z_0=z+h}\right||h|\frac{\exp(-c_5\frac{|z'|^2}{s})}{s^{\frac{d}{2}}\min\{s,1\}^{\frac{1}{2}}}  dz'.
\end{align}
Since $|h|^2\leq\min\{ t-s,1\}$, we get 
\begin{align}\nonumber
M(h,s) &\lesssim	|h|\int_{\mathbb{R}^d}\frac{\exp(-c_5\frac{|z-z'|^2}{t-s})}{ (t-s)^{\frac{d+2}{2}}}\frac{\exp(-c_5\frac{|z'|^2}{s})}{s^{\frac{d}{2}}\min\{s,1\}^{\frac{1}{2}}}  dz'\\&\lesssim |h|\exp(-c_5\frac{|z|^2}{8t}) \frac{\min\{t-s,s\}^{\frac{d}{2}}}{ (t-s)^{\frac{d+2}{2}}}\frac{1}{s^{\frac{d}{2}}\min\{s,1\}^{\frac{1}{2}}}. 
\end{align}
Combining this with \eqref{z65} and \eqref{z66}, we obtain 
\begin{align}\nonumber
&	\left|\delta_h^z\nabla_zH_L(x,x+z,t)-\left(\nabla_zG_{A(x+z_0)}(z+h,t)\vert_{z_0=z+h}-\nabla_zG_{A(x+z_0)}(z,t)\vert_{z_0=z}\right)\right|\\&\lesssim   \int_{t-|h|^2}^{t} \frac{1}{ (t-s)^{\frac{1}{2}}}\frac{\exp(-c_5\frac{|z|^2}{8t})}{t^{\frac{d}{2}}\min\{t,1\}^{\frac{1}{2}}} ds+ \int^{t-|h|^2}_{0} |h|\exp(-c_5\frac{|z|^2}{8t}) \frac{\min\{t-s,s\}^{\frac{d}{2}}}{ (t-s)^{\frac{d+2}{2}}}\frac{1}{s^{\frac{d}{2}}\min\{s,1\}^{\frac{1}{2}}}ds\nonumber\\&\lesssim |h|\log(2+\frac{\sqrt{t}}{|h|})\frac{\exp(-c_5\frac{|z|^2}{8t})}{t^{\frac{d}{2}}\min\{t,1\}^{\frac{1}{2}}}.
\end{align}
Thus, 
\begin{align}\nonumber
&	\left|\delta_h^z\nabla_zH_L(x,x+z,t)-{{\delta_h^z\nabla_zG_{A(x+z_0)}(z,t)\vert_{z_0=z}}}\right|\\&\lesssim\nonumber |h|\log(2+\frac{\sqrt{t}}{|h|})\frac{\exp(-c_5\frac{|z|^2}{8t})}{t^{\frac{d}{2}}\min\{t,1\}^{\frac{1}{2}}}+\left|\nabla_zG_{A(x+z_0)}(z+h,t)\vert_{z_0=z+h}-\nabla_zG_{A(x+z_0)}(z+h,t)\vert_{z_0=z}\right|\\&\lesssim |h|\log(2+\frac{\sqrt{t}}{|h|})\frac{\exp(-c_5\frac{|z|^2}{8t})}{t^{\frac{d}{2}}\min\{t,1\}^{\frac{1}{2}}}. 
\end{align}
Similarly, we also obtain for any $|h|\leq \sqrt{\min\{t,1\}}/10$
\begin{align*}
&\left|H_L(x+z,x,t)-G_{A(x+z)}(z,t)\right|\lesssim \frac{1}{t^{\frac{d-1}{2}}(1+t)^{\frac{1}{2}}}\exp(-c_0\frac{|z|^2}{t}),\\&	\left|\nabla_zH_L(x+z,x,t)-\nabla_zG_{A(x+z_0)}(z,t)\vert_{z_0=z}\right|\lesssim \frac{\log (2+t)}{t^{\frac{d}{2}}}\exp(-c_0\frac{|z|^2}{t}),\\& \left|\delta_h^z\nabla_zH_L(x+z,x,t)-\delta_h^z\nabla_zG_{A(x+z_0)}(z,t)\vert_{z_0=z}\right|\lesssim |h|\log(2+\frac{\sqrt{t}}{|h|})\frac{\exp(-c_0\frac{|z|^2}{t})}{t^{\frac{d}{2}}\min\{t,1\}^{\frac{1}{2}}},
\end{align*}
for some $c_0>0$.
Thus, we obtain \eqref{z17c1} and \eqref{z17c2},\eqref{z17c3}.
\end{proof}
\begin{prop} \label{heatke2}The following inequalities hold
\begin{align}\label{z19a}
&	|	(\nabla_x+\nabla_y)H_L(x,y,t)|\lesssim \frac{\log(2+t)}{t^{\frac{d}{2}}}\exp(-c_0\frac{|x-y|^2}{2t}),\\&
|	(\delta_h^x\nabla_x+\delta_{-h}^y\nabla_y)H_L(x,y,t)|\lesssim |h|\log(2+\frac{\sqrt{t}}{|h|})\frac{\exp(-c_0\frac{|x-y|^2}{t})}{t^{\frac{d}{2}}\min\{t,1\}^{\frac{1}{2}}},\label{z19b}
\end{align}
for any $|h|\leq \sqrt{\min\{t,1\}}/10$ and $t>0$.
In particular, we get
\begin{align}\label{z67a}
&\int_{\mathbb{R}^d}	|	(\delta_h^x\nabla_x+\delta_{-h}^y\nabla_y)H_L(x,y,t)| dy\lesssim \min\left\{\frac{|h|}{\min\{t,1\}^{\frac{1}{2}}},1\right\} \log(2+\frac{\sqrt{t}}{|h|}),\\&
\int_{\mathbb{R}^d}	|\delta_h^x\nabla_xH_L(x,y,t)| dy\lesssim\frac{\min\{\frac{|h|}{\sqrt{t}},1\}}{\sqrt{t}}+|h|\frac{\log(2+\frac{\sqrt{t}}{|h|})}{\min\{t,1\}^{\frac{1}{2}}},\label{z67b}\\&
\sup_{x,y}|\delta_h^x\nabla_xH_L(x,y,t)|\lesssim |h|\log(2+\frac{1}{|h|})\frac{\log(2+t)}{t^{\frac{d}{2}}\min\{t,1\}},\label{z67c}
\end{align}
for any $|h|\leq 1$, $ t>0$.
\end{prop}
\begin{proof} By   \eqref{z17c2} and \eqref{z17c3}, we have  for $|h|\leq \sqrt{\min\{t,1\}}/10$,
\begin{align}\nonumber
&\left|\nabla_xH_L(x,y,t)-\nabla_xG_{A(z_0)}(x-y,t)\vert_{z_0=x}\right|\\&\quad\quad\quad\quad+\left|\nabla_yH_L(x,y,t)+\nabla_xG_{A(z_0)}(x-y,t)\vert_{z_0=y}\right|\lesssim \frac{\log (2+t)}{t^{\frac{d}{2}}}\exp(-c_0\frac{|x-y|^2}{t})\label{z20a},\\&
\left|\delta_h^x\nabla_xH_L(x,y,t)-\delta_h^x\nabla_xG_{A(z_0)}(x-y,t)\vert_{z_0=x}\right|\nonumber\\&\quad\quad\quad+	\left|\delta_h^y\nabla_yH_L(x,y,t)+\delta_{-h}^x\nabla_xG_{A(z_0)}(x-y,t)\vert_{z_0=y}\right|\lesssim |h|\log(2+\frac{\sqrt{t}}{|h|})\frac{\exp(-c_0\frac{|x-y|^2}{t})}{t^{\frac{d}{2}}\min\{t,1\}^{\frac{1}{2}}}.\label{z20b}
\end{align}
So, we obtain \eqref{z19a} and \eqref{z19b}. Moreover, we have for any $|h|\leq \sqrt{\min\{t,1\}}/10$
\begin{align}
&	\int_{\mathbb{R}^d}	|	(\delta_h^x\nabla_x+\delta_{-h}^y\nabla_y)H_L(x,y,t)| dy\overset{\eqref{z19b}}\lesssim  \frac{|h|\log(2+\frac{\sqrt{t}}{|h|})}{\min\{t,1\}^{\frac{1}{2}}};\\&
\int_{\mathbb{R}^d}	|\delta_h^x\nabla_xH_L(x,y,t)| dy\overset{\eqref{z20b}}\lesssim \int_{\mathbb{R}^d}	\frac{|h|}{t^{\frac{d+2}{2}}}\exp(-c_0\frac{|x-y|^2}{t})+|h|\log(2+\frac{\sqrt{t}}{|h|})\frac{\exp(-c_0\frac{|x-y|^2}{t})}{t^{\frac{d}{2}}\min\{t,1\}^{\frac{1}{2}}} dy\nonumber\\&\quad\quad\quad\quad\quad\quad\quad\quad\quad\quad\quad\sim |h|\left(\frac{1}{t}+\frac{\log(2+\frac{\sqrt{t}}{|h|})}{\min\{t,1\}^{\frac{1}{2}}}\right);\\&
\sup_{x,y}|\delta_h^x\nabla_xH_L(x,y,t)|\overset{\eqref{z20b}}\lesssim \frac{|h|}{t^{\frac{d+2}{2}}}+|h|\log(2+\frac{\sqrt{t}}{|h|})\frac{1}{t^{\frac{d}{2}}\min\{t,1\}^{\frac{1}{2}}}\lesssim |h|\log(2+\frac{1}{|h|})\frac{\log(2+t)}{t^{\frac{d}{2}}\min\{t,1\}};
\end{align}
when $\sqrt{\min\{t,1\}}/10\leq |h|\leq 1$
\begin{align}
\nonumber&	\int_{\mathbb{R}^d}	|	(\delta_h^x\nabla_x+\delta_{-h}^y\nabla_y)H_L(x,y,t)| dy\\&\leq 	\int_{\mathbb{R}^d}	|	(\nabla_x+\nabla_y)H_L(x,y,t)| dy+\int_{\mathbb{R}^d}	|	\nabla_xH_L(x+h,y,t)+	\nabla_yH_L(x,y-h,t)| dy\nonumber\\&
\overset{\eqref{z19a}}	\lesssim \int_{\mathbb{R}^d} \frac{\log (2+t)}{t^{\frac{d}{2}}}\exp(-c_0\frac{|x-y|^2}{t}) dy\lesssim \log(2+t)\lesssim \log(2+\frac{\sqrt{t}}{|h|});
\end{align}
and 
\begin{align}\nonumber
\int_{\mathbb{R}^d}	|\delta_h^x\nabla_xH_L(x,y,t)| dy&\overset{\eqref{z20a}}\lesssim 	\int_{\mathbb{R}^d}\frac{\exp(-c_0\frac{|x-y|^2}{t})}{t^{\frac{d+1}{2}}}+\frac{\log (2+t)}{t^{\frac{d}{2}}}\exp(-c_0\frac{|x-y|^2}{t}) dy\\&\sim \frac{\log (2+t)}{\min\{t,1\}^{\frac{1}{2}}}\lesssim \frac{\min\{\frac{|h|}{\sqrt{t}},1\}}{\sqrt{t}}+|h|\frac{\log(2+\frac{\sqrt{t}}{|h|})}{\min\{t,1\}^{\frac{1}{2}}};
\end{align}
\begin{align}
\sup_{x,y}|\delta_h^x\nabla_xH_L(x,y,t)|\lesssim \frac{1}{t^{\frac{d+1}{2}}}+\frac{\log (2+t)}{t^{\frac{d}{2}}}\lesssim |h|\log(2+\frac{1}{|h|})\frac{\log(2+t)}{t^{\frac{d}{2}}\min\{t,1\}}.
\end{align}
These imply \eqref{z67a}, \eqref{z67b} and \eqref{z67c}.
\end{proof}
\begin{rem} In view of Lemma  \ref{heatke} and Proposition \ref{heatke2}, 
it follows that for any $0\leq \alpha<1$ and $0\leq \beta \leq 1+\alpha$, the inequality 
\begin{equation}\label{smooeff}
\sup_{s\in [0,1]}s^{\frac{1+\alpha-\beta}{2}}	||\na e^{sL}v||_{C^\alpha}\lesssim {{ ||v||_{C^\beta}}}
\end{equation}
holds.
\end{rem}

We denote now
\begin{equation}\label{Ddefi}
D(f)(x)=\int_{\mathbb{R}^d}\frac{|f(x)-f(y)|^2}{|x-y|^{d+1}}dy.
\end{equation}
\begin{lemma} We have for any $h\not=0$
\begin{align}\label{z16}
D(\delta_h u)(x)\geq C||u||_{L^\infty}^{-1}|h|^{-1} |\delta_h u(x)|^{3}.
\end{align}
\end{lemma}
\begin{proof} We recall from \cite{cv1, cvt}.  For $\lambda>0$, 
\begin{align}\nonumber
D(\delta_h u)(x)&\geq 	\int_{\mathbb{R}^d}\eta(\frac{|z|}{\lambda})\frac{|\delta_hu(x)|^2}{|z|^{d+1}}dz-2|\delta_hu(x)|\left|\int_{\mathbb{R}^d}\delta_h u(x+z)\frac{\eta(\frac{|z|}{\lambda})}{|z|^{d+1}}dz\right|\\&
\nonumber
\geq C|\delta_hu(x)|^2\lambda^{-1}-2|\delta_hu(x)|||u||_{L^\infty}\int_{\mathbb{R}^d}\left|\frac{\eta(\frac{|z|}{\lambda})}{|z|^{d+1}}-\frac{\eta(\frac{|z-h|}{\lambda})}{|z-h|^{d+1}}\right|dz
\\&
\geq C|\delta_hu(x)|^2\lambda^{-1}-C^{-1}|\delta_hu(x)|||u||_{L^\infty}|h|\lambda^{-2}.
\end{align}
Here $\eta$ is a cutoff function in $[0,\infty)$ such that $\eta=0$ in $[0,1]$ and $\eta=1$ in $[2,\infty).$
This implies  the result. 
\end{proof}
We also define  $L^{-\frac{1}{2}}$ as the inverse operator of  $L^{\frac{1}{2}}$ given by 
\begin{equation}\label{inveropera}
L^{-\frac{1}{2}}u(x)={{\frac{1}{\Gamma(\frac{1}{2})}}}\int_{0}^{\infty}s^{-\frac{1}{2}} \int_{\mathbb{R}^d}H_L(x,y,s)u(y){{dy ds}}.
\end{equation}
\begin{lemma} The following inequalities hold
\begin{align}\label{z21}
|\delta_h\na L^{-\frac{1}{2}}(v)(x)|\lesssim  ||v||_{L^\infty}^{\frac{1}{3}}|h|^{\frac{1}{3}}D(\delta_h v)(x)^{\frac{1}{3}}+|h|\log(2/|h|)^2 ||v||_{L^1\cap L^\infty}^{\frac{2}{3}}||v||_{L^\infty}^{\frac{1}{3}},
\end{align}
for any  {{$|h|\lesssim  1$}}; 
\begin{align}\label{z27}
||	\na L^{-\frac{1}{2}}(v)||_{L^\infty}\lesssim {{( 1+||v||_{L^\infty\cap L^1})\log\left(2+||v||_{\dot C^\alpha}\right)}},
\end{align}
and 
\begin{align}\label{z27b}
||	\na L^{-\frac{1}{2}}(v)||_{C^{\alpha}}\lesssim ||v||_{C^{\alpha}}+||v||_{L^1},
\end{align}
for any $\alpha\in (0,1)$.
\end{lemma}
\begin{proof}1) For $\lambda\in (0,1], |h|\lesssim 1$, we have 
\begin{align}\nonumber
|\delta_h\na L^{-\frac{1}{2}}(v)(x)|&\lesssim\int_{0}^{\lambda^2}\left| \int_{\mathbb{R}^d} \delta_h^x\na_x H_L(x,y,s)v(y) dy\right| s^{-\frac{1}{2}}ds+ \int_{\lambda^2}^2\int_{\mathbb{R}^d} |\delta_h^x\na_x H_L(x,y,s)| dy s^{-\frac{1}{2}}ds ||v||_{L^\infty}\\&\nonumber+ \int_{2}^\infty\int_{\mathbb{R}^d} |\delta_h^x\na_x H_L(x,y,s)||v(y)| dy s^{-\frac{1}{2}}ds\\&=I_1+I_2+I_3.
\end{align}
In view of \eqref{z20b}, \eqref{z67a} and \eqref{z67b}, we have
\begin{align}
&I_3\lesssim 	\int_{2}^{\infty}\int_{\mathbb{R}^d} |h|\log(2+\frac{\sqrt{s}}{|h|})\frac{\exp(-c_0\frac{|x-y|^2}{s})}{s^{\frac{d}{2}}}|v(y)| dy s^{-\frac{1}{2}}ds\lesssim |h|\log(\frac{2}{|h|}) ||v||_{L^1}^{\frac{2}{3}}||v||_{L^\infty}^{\frac{1}{3}};
\\&I_2\nonumber	\lesssim  \int_{\lambda^2}^2\left(\frac{\min\{\frac{|h|}{\sqrt{s}},1\}}{\sqrt{s}}+|h|\frac{\log(2+\frac{\sqrt{s}}{|h|})}{s^{\frac{1}{2}}} \right)s^{-\frac{1}{2}}ds||v||_{L^\infty}\\&
\quad\lesssim \left(\frac{|h|}{\lambda}+|h|\log(\frac{2}{|h|})^2\right)||v||_{L^\infty};
\end{align}
and 
\begin{align}\nonumber
I_1&\lesssim \int_{0}^{\lambda^2}\left| \int_{\mathbb{R}^d} \delta_{-h}^y\na_y H_L(x,y,s)v(y) dy\right| s^{-\frac{1}{2}}ds+\int_{0}^{\lambda^2}\int_{\mathbb{R}^d} |\left(\delta_h^x\na_x +\delta_{-h}^y\na_y\right)H_L(x,y,s)| dys^{-\frac{1}{2}}ds||v||_{L^\infty}\\&\lesssim\nonumber \int_{0}^{\lambda^2} \int_{\mathbb{R}^d} |\na_y H_L(x,y,s)||\delta_hv(y)-\delta_hv(x)  |dy s^{-\frac{1}{2}}ds+{{\int_{0}^{1}}}\min\{\frac{|h|}{s^{\frac{1}{2}}},1\} \log(2+\frac{\sqrt{s}}{|h|})s^{-\frac{1}{2}}ds||v||_{L^\infty}
\\&\lesssim \sqrt{\lambda D(\delta_h v)(x)}+|h|\log(2/|h|)^2||v||_{L^\infty},
\end{align}
here we have used the fact that 
\begin{equation}
\int_{0}^{\lambda^2} \left(\int_{\mathbb{R}^d} |\na_y H_L(x,y,s)|^2|x-y|^{d+1}dy\right)^{\frac{1}{2}} s^{-\frac{1}{2}}ds\lesssim \sqrt{\lambda}.
\end{equation}
Therefore, 
\begin{align}\label{z73}
|\delta_h\na L^{-\frac{1}{2}}(v)(x)|\lesssim  \sqrt{\lambda D(\delta_h v)(x)}+\frac{|h|}{\lambda}||v||_{L^\infty}+|h|\log(2/|h|)^2 ||v||_{L^1\cap L^\infty}^{\frac{2}{3}}||v||_{L^\infty}^{\frac{1}{3}}.
\end{align}
Choosing
\begin{equation}
\lambda=\left(\frac{|h|^2||v||_{L^\infty}^2}{D(\delta_h v)(x)+|h|^2||v||_{L^\infty}^2}\right)^{\frac{1}{3}}
\end{equation}
to get \eqref{z21}.\\
2) We have for $\lambda\in (0,1]$
\begin{align}\nonumber
|\na L^{-\frac{1}{2}}(v)(x)|&\lesssim \int_{0}^{\lambda^2}\left| \int_{\mathbb{R}^d} \na_y H_L(x,y,s)v(y) dy\right| s^{-\frac{1}{2}}ds+\int_{0}^{\lambda^2} \int_{\mathbb{R}^d} |(\na_x+\na_y) H_L(x,y,s)|dy s^{-\frac{1}{2}}ds ||v||_{L^\infty}\\&+ \int_{\lambda^2}^2\int_{\mathbb{R}^d} |\na_x H_L(x,y,s)| dy s^{-\frac{1}{2}}ds ||v||_{L^\infty}+ \int_{2}^\infty \sup_{y\in \mathbb{R}^d}|\na_x H_L(x,y,s)| s^{-\frac{1}{2}}ds||v||_{L^1}.
\end{align}
Using \eqref{z19a} and \eqref{z20a}, we deduce
\begin{align}\nonumber
|\na L^{-\frac{1}{2}}(v)(x)|&\lesssim \int_{0}^{\lambda^2} \int_{\mathbb{R}^d} |\na_y H_L(x,y,s)| |v(y)-v(x)| dy s^{-\frac{1}{2}}ds+\log(2+\frac{1}{\lambda})||v||_{L^\infty}+||v||_{L^1}\\&\lesssim \nonumber\int_{0}^{\lambda^2} \int_{\mathbb{R}^d} |\na_y H_L(x,y,s)| |x-y|^{\alpha}dy s^{-\frac{1}{2}}ds|| v||_{\dot C^\alpha}+\log(2+\frac{1}{\lambda})||v||_{L^\infty}+||v||_{L^1}
\\&\lesssim \lambda^\alpha|| v||_{\dot C^\alpha}+\log(2+\frac{1}{\lambda})||v||_{L^\infty}+||v||_{L^1}.\label{z74}
\end{align}
This implies \eqref{z27} by choosing $\lambda^\alpha= \frac{1}{1+|| v||_{\dot C^\alpha}}.$\vspace*{0.2cm}\\
3) Now we prove  \eqref{z27b}. In \eqref{z74}, we take $\lambda^\alpha=\frac{||v||_{L^\infty}}{||v||_{L^\infty}+|| v||_{\dot C^\alpha}}$ to obtain
\begin{equation}
||\na L^{-\frac{1}{2}}(v)||_{L^\infty}\lesssim ||v||_{L^\infty}+\log(2+\frac{|| v||_{\dot C^\alpha}}{||v||_{L^\infty}})||v||_{L^\infty}+||v||_{L^1}\lesssim ||v||_{C^{\alpha}}+||v||_{L^1}.
\end{equation}
Moreover,  in view of the proof of \eqref{z73}, we have 
\begin{align*}
|\delta_h\na L^{-\frac{1}{2}}(v)(x)|\lesssim  \int_{0}^{1}\left| \int_{\mathbb{R}^d} \delta_{-h}^y\na_y H_L(x,y,s)v(y) dy\right|s^{-\frac{1}{2}}ds+{{|h|\log(2/|h|)^2||v||_{L^1\cap L^\infty},}}
\end{align*}
for any $|h|\leq \frac{1}{4}$. 
Thus, it is enough to show that 
\begin{equation}\label{z75}
\int_{0}^{1}\left| \int_{\mathbb{R}^d} \delta_{-h}^y\na_y H_L(x,y,s)v(y) dy\right|s^{-\frac{1}{2}}ds\lesssim |h|^{\alpha}\left( ||v||_{C^{\alpha}}+||v||_{L^1}\right),
\end{equation}
for any $|h|\leq \frac{1}{4}$. \vspace*{0.2cm}\\Indeed, using \eqref{z20a} and \eqref{z20b}, we have 
\begin{align}\nonumber
&\int_{0}^{|h|^2}\left| \int_{\mathbb{R}^d} \delta_{-h}^y\na_y H_L(x,y,s)v(y) dy\right|s^{-\frac{1}{2}}ds= \int_{0}^{|h|^2}\left| \int_{\mathbb{R}^d} \na_y H_L(x,y,s)(\delta_{h}v(y)-\delta_{h}v(x)) dy\right|s^{-\frac{1}{2}}ds\\&\quad\quad\quad\lesssim \int_{0}^{|h|^2} \int_{\mathbb{R}^d} |\na_y H_L(x,y,s)||y-x|^{\frac{\alpha}{2}}|h|^{\frac{\alpha}{2}}dys^{-\frac{1}{2}}ds ||v||_{\dot C^{\alpha}}\nonumber\\&\quad\quad\quad\lesssim
\int_{0}^{|h|^2} \int_{\mathbb{R}^d} \frac{1}{s^{\frac{d+1}{2}}}\exp(-c_0\frac{|x-y|^2}{s})|y-x|^{\frac{\alpha}{2}}|h|^{\frac{\alpha}{2}}dys^{-\frac{1}{2}}ds ||v||_{\dot C^{\alpha}} \nonumber\\&\quad\quad\quad\lesssim
\int_{0}^{|h|^2} |h|^{\frac{\alpha}{2}}s^{-1+\frac{\alpha}{4}}ds ||v||_{\dot C^{\alpha}}\sim |h|^\alpha ||v||_{\dot C^{\alpha}},
\end{align}
and 
\begin{align}
&	\int_{|h|^2}^1\left| \int_{\mathbb{R}^d} \delta_{-h}^y\na_y H_L(x,y,s)v(y) dy\right|s^{-\frac{1}{2}}ds\leq  \int_{|h|^2}^1\int_{\mathbb{R}^d}| \delta_{-h}^y\na_y H_L(x,y,s)||x-y|^\alpha dys^{-\frac{1}{2}}ds ||v||_{\dot C^{\alpha}}
\nonumber\\&\quad\quad\quad\lesssim \int_{|h|^2}^1\int_{\mathbb{R}^d}\frac{|h|}{s^{\frac{d+2}{2}}}\exp(-c_0\frac{|x-y|^2}{s})|x-y|^\alpha dys^{-\frac{1}{2}}ds ||v||_{\dot C^{\alpha}}\nonumber\\&\quad\quad\quad\lesssim \int_{|h|^2}^1|h|s^{-\frac{3}{2}+\frac{\alpha}{2}}ds ||v||_{\dot C^{\alpha}}\sim |h|^\alpha ||v||_{\dot C^{\alpha}}.
\end{align}
Therefore, we obtain \eqref{z75}. 
\end{proof}
\begin{lemma} We have  for any $|h|\leq  1$,
\begin{align}\label{z22}
|h|^{-1}|\delta_h \na L^{-\frac{1}{2}}(v)|  |\delta_h v(x)|^2\lesssim ||v||_{L^\infty}	\left(D(\delta_h v)(x)+|h|^{2} \log(2/|h|)^6 ||v||_{L^\infty\cap L^1}^2\right).
\end{align}
\end{lemma}
\begin{proof} In view of \eqref{z16} and \eqref{z21}, we have 
\begin{align*}
& |h|^{-1}|\delta_h \na L^{-\frac{1}{2}}(v)|  |\delta_h v(x)|^2\\&\lesssim 	 |h|^{-1}\left(||v||_{L^\infty}^{\frac{1}{3}}|h|^{\frac{1}{3}}D(\delta_h v)(x)^{\frac{1}{3}}+|h|\log(2/|h|)^2 ||v||_{L^1\cap L^\infty}^{\frac{2}{3}}||v||_{L^\infty}^{\frac{1}{3}}\right) 	\left(||v||_{L^\infty}|h|	D(\delta_h v)(x)\right)^{\frac{2}{3}}
\\&\lesssim  ||v||_{L^\infty}	D(\delta_h v)(x)+|h|^{\frac{2}{3}} \log(2/|h|)^2||v||_{L^1\cap L^\infty}^{\frac{2}{3}}||v||_{L^\infty}		D(\delta_h v)(x)^{\frac{2}{3}}
\\&\lesssim  ||v||_{L^\infty}	\left(D(\delta_h v)(x)+|h|^{2} \log(2/|h|)^6 ||v||_{L^\infty\cap L^1}^2\right).	\end{align*}
This gives \eqref{z22}. 
\end{proof}
\begin{coro}\la{threquarts} We have 
{{\begin{align}\nonumber
&	|h|^{-1}|\delta_h \na L^{-\frac{1}{2}}(v)|  |\delta_h (\eta v)(x)|^2\lesssim ||v||_{L^\infty}	D(\delta_h ( \eta v))(x)+||v||_{L^\infty}||\delta_hv||_{L^\infty}^2\\&\nonumber\quad\quad\quad\quad\quad\quad\quad\quad\quad\quad+|h| \left( ||v||_{L^\infty\cap L^1}^2||v||_{L^\infty}	+( 1+||v||_{L^\infty\cap L^1})\log\left(2+||v||_{\dot C^\alpha}\right) ||v||_{L^\infty}^2\right)\\&\quad\quad\quad\quad\quad\quad\quad\quad\quad\quad+|h|^2||v||_{L^\infty} ||v||_{\dot C^{\frac{3}{4}}} ||v||_{\dot C^{\frac{1}{4}}},\label{z25}
\end{align}}}
where $\eta$ is a {{Lipschitz}}  cutoff function such that $\eta=1$ in $B(0,r_0)$ and  $\eta=0$ in $B(0,2r_0)^c$ and $\text{supp} v\subseteq B(0,r_1)$ for some $r_1\geq 4r_0$
\end{coro}
\begin{proof}
Since $
|\delta_h (\eta v)(x)-\eta(x) \delta_hv(x)|\lesssim  |h|||v||_{L^\infty}$, we have
\begin{equation}
|\eta(x) \delta_hv(x)|^2\geq  \frac{1}{2}	|\delta_h (\eta v)(x)|^2-C|h|^2||v||_{L^\infty}^2.
\end{equation}
Thus, 
\begin{align}\label{z23}
\eta^2 |h|^{-1}|\delta_h \na L^{-\frac{1}{2}}(v)|  |\delta_h  v(x)|^2\geq  \frac{1}{2}|h|^{-1}|\delta_h \na L^{-\frac{1}{2}}(v)|  |\delta_h (\eta v)(x)|^2-C|h|||\na L^{-\frac{1}{2}}(v)||_{L^\infty} ||v||_{L^\infty}^2.
\end{align}
Note that
\begin{equation}
\delta_z\delta_h( \eta v)(x)-\eta (x)	\delta_z\delta_h v(x)=v(x+h)	(\delta_z\delta_h\eta)(x)+(\delta_z\eta)(x)\delta_hv(x+z)+(\delta_h \eta)(x+z)(\delta_zv)(x+h),
\end{equation}
we have 
\begin{align}\nonumber \eta^2D(\delta_h v)(x)&\leq 2 D(\delta_h (\eta v))(x)+C\int_{\mathbb{R}^d}v(x+h)^2	|\delta_z\delta_h\eta(x)|^2 \frac{dz}{|z|^{d+1}}\\&\nonumber\quad\quad+C\int_{\mathbb{R}^d}|\delta_hv(x+z)|^2|\delta_z\eta(x)|^2\frac{dz}{|z|^{d+1}}+C\int_{\mathbb{R}^d}|\delta_zv(x+h)|^2|\delta_h\eta(x+z)|^2\frac{dz}{|z|^{d+1}}\\&\leq 
2 D(\delta_h (\eta v))(x)+C\left({{|h|}}||v||_{L^\infty}^2+||\delta_hv||_{L^\infty}^2+C|h|^2 ||v||_{\dot C^{\frac{3}{4}}} ||v||_{\dot C^{\frac{1}{4}}}\right).\label{z24}
\end{align}
Here we have used the fact that 
\begin{equation}\la{weird}
\int_{\mathbb{R}^d}|\delta_zv(x+h)|^2\frac{dz}{|z|^{d+1}}\lesssim ||v||_{\dot C^{\frac{3}{4}}} ||v||_{\dot C^{\frac{1}{4}}}.	
\end{equation}
Combining \eqref{z23}, \eqref{z24} with \eqref{z22}, we have
\begin{align}\nonumber
&	|h|^{-1}|\delta_h \na L^{-\frac{1}{2}}(v)|  |\delta_h (\eta v)(x)|^2\lesssim ||v||_{L^\infty}	D(\delta_h (\eta v))(x)+||v||_{L^\infty}||\delta_hv||_{L^\infty}^2\\&+{{|h|}}||v||_{L^\infty\cap L^1}^2||v||_{L^\infty}+|h|^2||v||_{L^\infty} ||v||_{\dot C^{\frac{3}{4}}} ||v||_{\dot C^{\frac{1}{4}}}	+|h|||\na L^{-\frac{1}{2}}(v)||_{L^\infty} ||v||_{L^\infty}^2\nonumber\\&\overset{\eqref{z27}}\lesssim\nonumber ||v||_{L^\infty}	D(\delta_h (\eta v))(x)+||v||_{L^\infty}||\delta_hv||_{L^\infty}^2++|h|^2||v||_{L^\infty} ||v||_{\dot C^{\frac{3}{4}}} ||v||_{\dot C^{\frac{1}{4}}}\\&\quad\quad+|h|\left( ||v||_{L^\infty\cap L^1}^2||v||_{L^\infty}+ {{( 1+||v||_{L^\infty\cap L^1})\log\left(2+||v||_{\dot C^\alpha}\right)}} ||v||_{L^\infty}^2\right).
\end{align}
This implies \eqref{z25}. 
\end{proof}
\begin{lemma} \label{remindL1/2} We denote 
\begin{equation}
Ju(x):=	(L_{y}^{\frac{1}{2}}u(x)\vert_{y=x}-	L^{\frac{1}{2}}u(x)).
\end{equation}
The following inequalities hold
\begin{align}\label{z54a}
&	|J(x)|\lesssim \int_{\mathbb{R}^d}\frac{|u(x+z)-u(x)|}{|z|^d(|z|+1)}dz,
\end{align}
\begin{align}\label{z54}
||J||_{\dot C^{\alpha_1}}\lesssim ||u||_{ C^{\alpha_1+\kappa}},
\end{align}
for any $\alpha_1\in (0,1)$ and  $\kappa\in (0,1-\alpha_1)$. Moreover, 
\begin{align}\label{z0}
||L^{\frac{1}{2}}u||_{L^\infty}\lesssim ||u||_{\dot C^\alpha}^{\frac{\alpha}{2-\alpha}}||u||_{\dot C^{1+\frac{\alpha}{2}}}^{\frac{2(1-\alpha)}{2-\alpha}}+ ||u||_{L^\infty}.
\end{align}
\end{lemma}
\begin{proof}
We have 
\begin{equation}
J(x)=
c_0\int_{0}^{\infty}\int_{\mathbb{R}^d}J_0(z,x,s)(u(x)-u(x+z))dzs^{-\frac{3}{2}}ds
\end{equation}
with 
\begin{equation}
	J_0(z,x,s):=G_{A(x)}(z,s)-H_L(x,x+z,s).
\end{equation}
In view of \eqref{z17c1} and  \eqref{z19a}, we have  for any $x,z\in \mathbb{R}^d$,
\begin{align}\label{z55a}
\left|J_0(z,x,s)\right|&\lesssim |G_{A(x+z)}(z,s)-G_{A(x)}(z,s)|+ \frac{\exp(-c_0\frac{|z|^2}{s})}{s^{\frac{d-1}{2}}(1+s)^{\frac{1}{2}}}\lesssim \frac{\min\{s,1\}^{\frac{1}{2}}\exp(-c_0\frac{|z|^2}{s})}{s^{\frac{d}{2}}}, 
\end{align}
and 
\begin{align}\label{z55b}
\left|\na_xJ_0(z,x,s)\right|&\leq \left|\na_x(H_L(x,x+z,s))\right|+\left|\na_x\left(G_{A(x)}(z,s)\right)\right|\lesssim\frac{\log(2+s)}{s^{\frac{d}{2}}}\exp(-c_0\frac{|z|^2}{s}).\end{align}
So, we obtain \eqref{z54a} and  \eqref{z0}. \\
Now, we prove  \eqref{z54}.   By an interpolation inequality, we have 
\begin{align}\nonumber
||J||_{\dot C^{\alpha_1}}&\lesssim  \int_{0}^{\infty}\int_{\mathbb{R}^d}\left(||J_0(z,.,s)||_{\dot C^{\alpha_1}}||u(.)-u(.+z)||_{L^\infty}+||J_0(z,.,s)||_{L^\infty}||u(.)-u(.+z)||_{\dot C^{\alpha_1}}\right)dzs^{-\frac{3}{2}}ds\\&\nonumber\lesssim \int_{0}^{\infty}\int_{\mathbb{R}^d}\left(||J_0(z,.,s)||_{L^\infty}^{1-\alpha_1}||\na J_0(z,.,s)||_{L^\infty}^{\alpha_1}\min\{|z|,1\}^{\alpha_1+\kappa}\right.\\&\quad\quad\quad\quad\quad\quad\quad\quad\quad\left.+||J_0(z,.,s)||_{L^\infty}\min\{|z|,1\}^\kappa\right)dzs^{-\frac{3}{2}}ds||u||_{ C^{\alpha_1+\kappa}},
\end{align}
for any $\kappa\in (0,1-\alpha_1)$.\\
Combining this with  \eqref{z55a} and \eqref{z55b}, we obtain 
\begin{align}\nonumber
||J||_{\dot C^{\alpha_1}}&\lesssim \int_{0}^{\infty}\int_{\mathbb{R}^d}\left(\min\{s,1\}^{\frac{1-\alpha_1}{2}}\log(2+s)^{\alpha_1}\min\{|z|,1\}^{\alpha_1+\kappa}\right.\\&\nonumber\quad\quad\quad\quad\quad\quad\quad\quad\left.+\min\{s,1\}^{\frac{1}{2}}\min\{|z|,1\}^\kappa\right)\exp(-c_0\frac{|z|^2}{s})dzs^{-\frac{3+d}{2}}ds||u||_{\dot C^{\alpha_1+\kappa}}\\&\lesssim
\int_{0}^{\infty}\min\{s,1\}^{\frac{1+\kappa}{2}}\log(2+s)^{\alpha_1}s^{-\frac{3}{2}}ds||u||_{ C^{\alpha_1+\kappa}}\sim ||u||_{ C^{\alpha_1+\kappa}}.
\end{align}
Here we have used the fact that  
\begin{align*}
\int_{\mathbb{R}^d}|z|^\beta	\exp(-c_0\frac{|z|^2}{s})dz=s^{\frac{d+\beta}{2}}~~\forall\beta>-d.
\end{align*}
This implies \eqref{z54}.
\end{proof}
\begin{lemma}\label{geoper} Let  $D$ be defined by \eqref{Ddefi}.  The following inequalities hold 
\begin{align}\label{z1}
v(x)L^{\frac{1}{2}}v(x)-\fr12 L^{\frac{1}{2}}(v^2)(x)\geq c D(v)(x),
\end{align}
and 
\begin{align}
\left|	(\delta_h	L^{\frac{1}{2}}v)(x)-	L^{\frac{1}{2}}(\delta_hv)(x)\right|\lesssim |h|^{1-\alpha}\left(||u||_{C^{1-\alpha+\kappa}}||u||_{C^{1-\alpha-\kappa}}\right)^{\frac{1}{2}} +|h|\left(||u||_{C^{1+\kappa}}||u||_{C^{1-\kappa}}\right)^{\frac{1}{2}}+ |h|^{1-\alpha} ||v||_{L^1},\label{z2b}
\end{align}
for any $|h|\leq 1$ and   $0<\kappa\leq \alpha\leq \frac{1}{4}$. In particular, 
\begin{align}
\left|	(\delta_h	L^{\frac{1}{2}}v)(x)-	L^{\frac{1}{2}}(\delta_hv)(x)\right|\lesssim |h|^{1-\alpha} \left( ||u||_{\dot C^\alpha}^{\frac{\alpha}{2-\alpha}}||u||_{\dot  C^{1+\frac{\alpha}{2}}}^{\frac{2(1-\alpha)}{2-\alpha}}+||u||_{L^1}\right),\label{z2}
\end{align} 
for any $\alpha\in (0,\frac{1}{4}]$.
\end{lemma}

\begin{proof} 1)
Because ${{\int_{\mathbb{R}^d}}}H_L(x,y,t)dy=1$ for any $t>0$ and 	\begin{align}
\int_{0}^{\infty}H_L(x,y,t)t^{-\fr32}dt {{\overset{\eqref{z17a}}\sim }} \frac{1}{|x-y|^{d+1}}~~\forall x,y\in \mathbb{R}^d~x\not= y,
\end{align}
we have 
\begin{align}
v(x)L^{\fr12}v(x)-\fr12 L^{\fr12}(v^2)(x)=c\int_{0}^{\infty}\int_{\mathbb{R}^d} H_L(x,y,t)(v(x)-v(y))^2t^{-\fr32}dydt \geq c' D(v)(x).
\la{nonm}
\end{align}
This implies \eqref{z1}. \\
2) Now we prove \eqref{z2b}. Note that  for any $x,h\in \mathbb{R}^d$
\begin{align}
I(x):=(\delta_h	L^{\frac{1}{2}}v)(x)-	L^{\frac{1}{2}}(\delta_hv)(x)=c\int_{0}^{\infty}\int_{0}^{t}e^{-(t-s)L}\left(\operatorname{div}(\delta_h A)\na\right) \tau_h (e^{-sL}v)ds t^{-\fr32}dt,
\la{delhL}
\end{align}
where 
\be
\tau_h f (x) =  f(x+h).
\la{tauh}
\ee
Using \eqref{z7} and the fact that  \begin{align}
&	\int_{s}^{\infty} \min\{t-s,1\}^{-\frac{1-\alpha}{2}} t^{-\fr32}dt\lesssim \min\{s,1\}^{-\frac{1-\alpha}{2}}s^{-\frac{1}{2}}, \\&
||(\delta_hA)\na \tau_h(e^{-sL}v)||_{C^{\alpha}}\lesssim |h|^{1-\alpha} ||\na e^{-sL}v||_{L^\infty}+|h| ||\na e^{-sL}v||_{C^{\alpha}},\\&
||\na e^{sL}v||_{C^\alpha(\mathbb{R}^d)}\lesssim s^{-\frac{d}{2}}||v||_{L^1}~~\forall s\geq 1,
\end{align}  we get 
\begin{align}\nonumber
||I||_{L^\infty}&\lesssim \int_{0}^{\infty}\int_{0}^{t}  \min\{t-s,1\}^{-\frac{1-\alpha}{2}} ||(\delta_hA)\na \tau_h(e^{-sL}v)||_{C^{\alpha}}ds t^{-\fr32}dt
\\&\lesssim \nonumber \int_{0}^{\infty}\min\{s,1\}^{-\frac{1-\alpha}{2}}s^{-\frac{1}{2}}\left(|h|^{1-\alpha} ||\na e^{-sL}v||_{L^\infty}+|h| ||\na e^{sL}u||_{C^{\alpha}}\right)ds\\&\nonumber
\lesssim |h|^{1-\alpha} ||v||_{L^1}+ \int_{0}^{1}s^{-1+\frac{\alpha}{2}}\left(|h|^{1-\alpha} ||\na e^{-sL}v||_{L^\infty}+|h| ||\na e^{-sL}v||_{C^{\alpha}}\right)ds
\\&
\nonumber	\lesssim |h|^{1-\alpha} ||v||_{L^1}+|h|^{1-\alpha}\left(\sup_{s\in [0,1]}s^{\frac{\alpha-\kappa}{2}}||\na e^{-sL}v||_{L^\infty}\right)^{\frac{1}{2}} \left(\sup_{s\in [0,1]}s^{\frac{\alpha+\kappa}{2}}||\na e^{-sL}v||_{L^\infty}\right)^{\frac{1}{2}}\nonumber
\\&
\quad\quad +|h|\left(\sup_{s\in [0,1]}s^{\frac{\alpha-\kappa}{2}}||\na e^{-sL}v||_{C^{\alpha}}\right)^{\frac{1}{2}} \left(\sup_{s\in [0,1]}s^{\frac{\alpha+\kappa}{2}}||\na e^{-sL}v||_{C^{\alpha}}\right)^{\frac{1}{2}}.
\label{z67}
\end{align}This implies \eqref{z2b} by using \eqref{smooeff}. Then, \eqref{z2} follows from \eqref{z2b} and interpolation inequality. 
\end{proof}

\vspace{.5in}
{\bf{Data availability statement.}} This research does not have any associated data.\\

 {\bf {Conflict of Interest.}} The authors declare that there are no conflicts of interest.\\

{\bf{Acknowledgments.}} The work of PC was partially supported by NSF grant DMS-2106528 and by
a Simons Collaboration Grant 601960. The work of MI was partially supported by NSF grant DMS-2204614. Q.H.N.  is supported by the Academy of Mathematics and Systems Science, Chinese Academy of Sciences startup fund; CAS Project for Young Scientists in Basic Research, Grant No. YSBR-031;  and the National Natural Science Foundation of China (No. 12288201);  and the National Key R$\&$D Program of China under grant 2021YFA1000800.

\end{document}